\date{\today} 
\documentclass[10pt]{article}
\usepackage[utf8]{inputenc}
\usepackage{hyperref}
\usepackage{amssymb, amsmath}
\usepackage{stmaryrd}
\usepackage[final]{showlabels}
\usepackage{float}
\usepackage{abstract}

\usepackage{latexsym}

\usepackage{amsthm}

\newcommand{\g}{{\mathfrak g}}

\newcommand{\fa}{{\mathfrak a}}
\newcommand{\fb}{{\mathfrak b}}

\newcommand{\fe}{{\mathfrak e}}

\newcommand{\fg}{{\mathfrak g}}
\newcommand{\fh}{{\mathfrak h}}

\newcommand{\fj}{{\mathfrak j}}
\newcommand{\fk}{{\mathfrak k}}
\newcommand{\fl}{{\mathfrak l}}

\newcommand{\fp}{{\mathfrak p}}

\newcommand{\fs}{{\mathfrak s}}
\newcommand{\ft}{{\mathfrak t}}
\newcommand{\fu}{{\mathfrak u}}

\newcommand{\fz}{{\mathfrak z}}

\renewcommand\sp{\mathfrak {sp}} 
\newcommand\csp{\mathfrak {csp}} 
\newcommand\hsp{\mathfrak {hsp}} 
\newcommand\hcsp{\mathfrak {hcsp}} 
 
\newcommand\heis{\mathfrak {heis}}

\newcommand{\1}{\mathbf{1}}

\newcommand{\cD}{\mathcal{D}}

\newcommand{\cF}{\mathcal{F}}

\newcommand{\cH}{\mathcal{H}}
\newcommand{\cI}{\mathcal{I}}
\newcommand{\cK}{\mathcal{K}}

\newcommand{\cM}{\mathcal{M}}

\newcommand{\cO}{\mathcal{O}}

\newcommand{\cS}{\mathcal{S}}

\newcommand{\cU}{\mathcal{U}}

\newcommand{\cW}{\mathcal{W}}

\newcommand{\derat}[1]{\frac{d}{dt}\big\vert_{t = #1}}
\newcommand{\bigderat}[1]{\frac{d}{dt}\Big\vert_{t = #1}}

\newcommand{\dd}{{\tt d}}

\newcommand{\N}{{\mathbb N}}

\newcommand{\R}{{\mathbb R}}
\newcommand{\C}{{\mathbb C}}

\renewcommand{\H}{{\mathbb H}}
\newcommand{\T}{{\mathbb T}}

\newcommand{\bS}{{\mathbb S}}

\newcommand{\tV}{{\mathtt V}}
\newcommand{\sH}{{\mathsf H}}
\newcommand{\sE}{{\mathsf E}}

\renewcommand{\hat}{\widehat}
\newcommand{\hotimes}{{\hat\otimes}}

\renewcommand{\tilde}{\widetilde}

\renewcommand{\L}{\mathop{\bf L{}}\nolimits}

\newcommand{\Aff}{\mathop{{\rm Aff}}\nolimits}

\newcommand{\SL}{\mathop{{\rm SL}}\nolimits}

\newcommand{\PSL}{\mathop{{\rm PSL}}\nolimits}
\newcommand{\SO}{\mathop{{\rm SO}}\nolimits}

\newcommand{\U}{\mathop{\rm U{}}\nolimits}

\newcommand{\Sp}{\mathop{{\rm Sp}}\nolimits}
\newcommand{\Sym}{\mathop{{\rm Sym}}\nolimits}

\newcommand{\Heis}{\mathop{{\rm Heis}}\nolimits}

\newcommand{\HSp}{\mathop{{\rm HSp}}\nolimits}

\newcommand{\HCSp}{\mathop{{\rm HCSp}}\nolimits}

\newcommand{\gl}  {\mathop{{\mathfrak{gl} }}\nolimits}

\newcommand{\fsl} {\mathop{{\mathfrak{sl} }}\nolimits}

\newcommand{\su}  {\mathop{{\mathfrak{su} }}\nolimits}
\newcommand{\so}  {\mathop{{\mathfrak{so} }}\nolimits}

\newcommand{\Fix}{\mathop{{\rm Fix}}\nolimits}

\newcommand{\ad}{\mathop{{\rm ad}}\nolimits}
\newcommand{\Ad}{\mathop{{\rm Ad}}\nolimits}
\newcommand{\AdS}{\mathop{{\rm AdS}}\nolimits}

\renewcommand{\Re}{\mathop{{\rm Re}}\nolimits}
\renewcommand{\Im}{\mathop{{\rm Im}}\nolimits}

\newcommand{\tr}{\mathop{{\rm tr}}\nolimits}

\newcommand{\Hom}{\mathop{{\rm Hom}}\nolimits}

\newcommand{\Hol}{\mathop{{\rm Hol}}\nolimits}

\newcommand{\Herm}{\mathop{{\rm Herm}}\nolimits}

\newcommand{\Aut}{\mathop{{\rm Aut}}\nolimits}

\newcommand{\End}{\mathop{{\rm End}}\nolimits}
\newcommand{\id}{\mathop{{\rm id}}\nolimits}
\newcommand{\rk}{\mathop{{\rm rank}}\nolimits}

\renewcommand{\dim}{\mathop{{\rm dim}}\nolimits}

\newcommand{\im}{\mathop{{\rm im}}\nolimits}

\newcommand{\spec}{\mathop{{\rm spec}}\nolimits}

\newcommand{\Inn}{\mathop{{\rm Inn}}\nolimits}

\newcommand{\oline}{\overline}
\newcommand{\la}{\langle}
\newcommand{\ra}{\rangle}
\newcommand{\spann}{{\rm span}}

\newcommand{\ext}{{\rm ext}}

\def\theoremname{Theorem}
\def\propositionname{Proposition}
\def\corollaryname{Corollary}
\def\lemmaname{Lemma}
\def\remarkname{Remark}
\def\conjecturename{Conjecture} 

\def\definitionname{Definition}
\def\exercisename{Exercise}
\def\examplename{Example}
\def\examplesname{Examples}
\def\problemname{Problem}
\def\problemsname{Problems}

\def\@thmcounter#1{\noexpand\arabic{#1}}
\def\@thmcountersep{}
\def\@begintheorem#1#2{\it \trivlist \item[\hskip 
\labelsep{\bf #1\ #2.\quad}]}
\def\@opargbegintheorem#1#2#3{\it \trivlist
      \item[\hskip \labelsep{\bf #1\ #2.\quad{\rm #3}}]}
\makeatother
\newtheorem{theor}{\theoremname}[section]
\newtheorem{propo}[theor]{\propositionname}
\newtheorem{coro}[theor]{\corollaryname}
\newtheorem{lemm}[theor]{\lemmaname}

\newenvironment{thm}{\begin{theor}\it}{\end{theor}}

\newenvironment{prop}{\begin{propo}\it}{\end{propo}}
\newenvironment{proposition}{\begin{propo}\it}{\end{propo}}

\newenvironment{cor}{\begin{coro}\it}{\end{coro}}

\newenvironment{lem}{\begin{lemm}\it}{\end{lemm}}

\newtheorem{rema}[theor]{\remarkname}

\newenvironment{remark}{\begin{rema}\rm}{\end{rema}}
\newenvironment{rem}{\begin{rema}\rm}{\end{rema}}

\newtheorem{stepnow}[theor]{}

\newtheorem{defin}[theor]{\definitionname} 

\newenvironment{definition}{\begin{defin}\rm}{\end{defin}}

\newtheorem{exerc}{\exercisename}[section]

\newtheorem{exa}[theor]{\examplename}

\newenvironment{example}{\begin{exa}\rm}{\end{exa}}

\newtheorem{exas}[theor]{\examplesname}

\newtheorem{conj}[theor]{\conjecturename}

\newtheorem{pro}[theor]{\problemname}

\newtheorem{prs}[theor]{\problemsname}

\newcommand{\pmat}[1]{\begin{pmatrix} #1 \end{pmatrix}}

\usepackage{bbold}

\newcommand{\Wmin}{{W_{\mathrm{min}}}}
\newcommand{\Wmax}{W_{\mathrm{max}}}
\newcommand{\Cmin}{C_{\mathrm{min}}}
\newcommand{\Cmax}{C_{\mathrm{max}}}

\newcommand{\bbone}{\mathbb{1}}

\newcommand{\tran}{\mathsf{T}}

\newcommand{\der}{\mathrm{der}}

\usepackage{enumitem}
\setlist[enumerate,1]{label={\rm (\alph*)}}
\setlist[enumerate,2]{label={\rm (\roman*)}}

\makeatletter
\def\blfootnote{\xdef\@thefnmark{}\@footnotetext}
\makeatother

\begin{document}

\title{Nets of standard subspaces induced by antiunitary representations of admissible Lie groups I}

\author{Daniel Oeh\footnote{Department Mathematik, Friedrich--Alexander--Universit\"at Erlangen--N\"urnberg, Cauerstr. 11, D-91058 Erlangen, Germany, oehd@math.fau.de} \,\footnote{Supported by DFG-grant Ne 413/10-1}} 

\maketitle

\begin{NoHyper}
\end{NoHyper}

\begin{abstract}
  Let \((\pi, \cH)\) be a strongly continuous unitary representation of a 1-connected Lie group \(G\) such that the Lie algebra \(\g\) of \(G\) is generated by the positive cone \(C_\pi := \{x \in \g : -i\partial \pi(x) \geq 0\}\) and an element \(h\) for which the adjoint representation of \(h\) induces a 3-grading of \(\g\). Moreover, suppose that \((\pi, \cH)\) extends to an antiunitary representation of the extended Lie group \(G_\tau := G \rtimes \{\1, \tau_G\}\), where \(\tau_G\) is an involutive automorphism of \(G\) with \(\L(\tau_G) = e^{i\pi\ad h}\).

  In a recent work by Neeb and {\'Olafsson}, a method for constructing nets of standard subspaces of \(\cH\) indexed by open regions of \(G\) has been introduced and applied in the case where \(G\) is semisimple.
  In this paper, we extend this construction to general Lie groups \(G\), provided the above assumptions are satisfied and the center of the ideal \(\g_C = C_\pi - C_\pi \subset \g\) is one-dimensional. The case where the center of \(\g_C\) has more than one dimension will be discussed in a separate paper.
\end{abstract}


\section{Introduction}
\label{sec:intro}

In the formulation of Algebraic Quantum Field Theory (AQFT) by Haag--Kastler (cf.\ \cite{Ha96}), physical systems are described in terms of nets of local observables: To every open region \(\cO\) in a space-time manifold \(X\), one assigns a von Neumann algebra \(\cM(\cO)\) of bounded operators in a Hilbert space \(\cH\).
This assignment should satisfy the \emph{isotony} and \emph{locality} assumptions, namely that \(\cO_1 \subset \cO_2\) implies \(\cM(\cO_1) \subset \cM(\cO_2)\) and that, if \(\cO_1\) and \(\cO_2\) are space-like separated, then \(\cM(\cO_1) \subset \cM(\cO_2)'\), where \(\cM(\cO_2)'\) denotes the commutant of \(\cM(\cO_2)\).
One assumes that there exists a unitary representation \((\pi, \cH)\) of a Lie group \(G\) which acts on \(X\) and satisfies the \emph{covariance} condition \(\pi(g)\cM(\cO)\pi(g^{-1}) = \cM(g.\cO)\).
Moreover, we assume that there exists a unit vector \(\Omega \in \cH\) that is fixed by \(\pi(G)\).
If \(\Omega\) is a cyclic and separating vector for some von Neumann algebra \(\cM(\cO)\), then the Tomita--Takesaki Theorem (cf.\ \cite[Thm.\ 2.5.14]{BR87}) provides a pair \((\Delta_\cO, J_\cO)\) of modular objects, where \(\Delta_\cO\) is a positive selfadjoint operator on \(\cH\) and \(J_\cO\) is an antiunitary involution on \(\cH\).
This pair satisfies the modular relation \(J_\cO \Delta_\cO J_\cO = \Delta_\cO^{-1}\) as well as
\[J_\cO \cM(\cO) J_\cO = \cM(\cO)' \quad \text{and} \quad \Delta_\cO^{it} \cM(\cO) \Delta_\cO^{-it} = \cM(\cO) \quad \text{for all } t \in \R.\]
The operator \(J_\cO\Delta_\cO^{1/2}\) is the closure of the densely defined unbounded antilinear operator
\[S_\cO : \cM(\cO)\Omega \rightarrow \cH, \quad A\Omega \mapsto A^*\Omega.\]
The fixed point space \(\tV_\cO\) of the closure of \(S_\cO\) is a \emph{standard subspace} of \(\cH\), i.e.\ a closed real subspace such that
\[\tV_\cO \cap i\tV_\cO = \{0\} \quad \text{and} \quad \oline{\tV_\cO + i\tV_\cO} = \cH.\]
More generally, one can show that there is a one-to-one correspondence between standard subspaces of \(\cH\) and pairs of operators \((A, J)\), where \(A\) is selfadjoint and \(J\) is an antiunitary involution with \(JAJ = -A\) (cf.\ \cite[Cor.\ 3.5]{Lo08}). For such a pair of operators, the fixed point space of the antilinear operator \(Je^{\frac{1}{2}A}\) is a standard subspace.
We refer to \cite{Lo08} for more information on the structure of standard subspaces.
We thus obtain a net of standard subspaces of \(\cH\) indexed by those open regions \(\cO\) of \(X\) for which \(\Omega\) is a cyclic and separating vector of \(\cM(\cO)\).
Conversely, we can pass from nets of standard subspaces to nets of von Neumann algebras using second quantization (cf.\ \cite{Ar63}, \cite{NO17}).
As before, we can formulate locality, isotony, and covariance properties of such nets if we define the causal complement \(\tV'\) of a standard subspace \(\tV\) as the real orthogonal complement of \(i\tV\) in \(\cH\). 
These properties are then preserved by the second quantization functor.
This allows us to construct nets of local observables by constructing nets of standard subspaces with the locality, isotony, and covariance property.

In AQFT, the space-time on which the net of standard subspaces (resp.\ von Neumann algebras) is defined is usually a homogeneous space.
Typical well-known examples are:
\begin{itemize}
  \item The circle \(\bS^1 = \{z \in \C : |z| = 1\}\), on which the group \(\PSL(2,\R) = \SL(2,\R) / \{-1,1\}\) acts by M\"obius transformations.
    We can identify \(\bS^1\) with \(\PSL(2,\R)/P\), where \(P\) is a closed subgroup that is conjugate to the two-dimensional subgroup of \(\PSL(2,\R)\) generated by translations and dilations (cf.\ \cite[2.1]{Lo08}).
  \item The \(d\)-dimensional anti-de Sitter space \(\AdS^d := \{x \in \R^{d+1} : x_0^2 + x_1^2 - x_2^2 - \ldots - x_d^2 = 1\}\) endowed with the group action of \(\SO_{2,d-1}(\R)_e\). The stabilizer of \(e_0 \in \AdS^d\) can be identified with \(\SO_{1,d-1}^\uparrow(\R)\), so that \(\AdS^d \cong \SO_{2,d-1}(\R)_e / \SO_{1,d-1}^\uparrow(\R)\).
\end{itemize}
In this paper, we construct nets of closed real subspaces \(\cO \mapsto \tV(\cO)\) indexed by open subsets \(\cO\) of a Lie group \(G\) on a Hilbert space \(\cH\) which satisfy isotony, locality, and covariance.
The main advantage of this approach is that we then also obtain a net of closed subspaces on the corresponding homogeneous spaces of \(G\): Given a closed subgroup \(P \subset G\) and the corresponding natural projection \(q_P : G \rightarrow M := G/P\), we assign to every open subset \(\cO_M \subset M\) the subspace \(\tV_M(\cO_M) = \tV(q_P^{-1}(\cO_M))\).
In order to identify standard subspaces in the net on \(M\) with one of the standard subspaces constructed in the net on \(G\), the net on the group should satisfy the invariance property \(\tV(\cO P) = \tV(\cO)\) for open subsets \(\cO \subset G\). We return to the problem of constructing nets on homogeneous spaces in Section \ref{sec:nets-hom-spaces}.

One way to construct standard subspaces is the \emph{Brunetti--Guido--Longo} (BGL) construction (\cite{BGL02}): Let \((G, \tau_G)\) be a symmetric connected Lie group with Lie algebra \((\g,\tau)\) and let \((\pi, \cH)\) be a strongly continuous antiunitary representation of \(G_\tau := G \rtimes \{\1, \tau_G\}\), i.e.\ \(\pi(G)\) consists of unitary operators on \(\cH\) and \(\pi(\tau_G)\) is an antilinear surjective isometry.
We assume that the kernel of \(\pi\) is discrete.
Let \(h \in \g^\tau\) be a fixed point of \(\tau\). Then
\[J := \pi(\tau_G) \quad \text{and} \quad \Delta^{-it/2\pi} := \pi(\exp(th)), \quad t \in \R,\]
determine a standard subspace \(\tV := \Fix(J\Delta^{1/2})\) of \(\cH\). The subsemigroup
\[S_\tV := \{g \in G : \pi(g)\tV \subset \tV\}\]
encodes the order structure on the \(G\)-orbit of \(\tV\) in the set of standard subspaces of \(\cH\). It has been studied extensively in \cite{Ne21} and \cite{Ne20}. Its Lie wedge
\[\L(S_\tV) := \{x \in \g : \exp(\R_+ x) \subset S_\tV\}\]
can be described in terms of the Lie algebraic data \((\pi,\cH)\), \(\tau\), and \(h\) by
\begin{equation}
  \label{eq:lie-wedge-sv}
  \L(S_\tV) = C_-(\pi, \tau, h) \oplus \g_0^\tau(h) \oplus C_+(\pi, \tau, h), \quad C_\pm(\pi, \tau, h) := \g_{\pm 1}^{-\tau}(h) \cap (\pm C_\pi),
\end{equation}
where \(\g_\lambda(h) := \ker(\ad(h) - \lambda\id_\g)\) and \(C_\pi := \{x \in \g : -i\partial \pi(x) \geq 0\}\) is the positive cone of \(\pi\) (cf.\ \cite[Thm.\ 5.4]{Ne21}).
We refer to \cite{Oeh20a} and \cite{Oeh20b} for first steps towards a classification of the Lie algebras generated by \(C_+(\pi,\tau,h)\) and \(C_-(\pi,\tau,h)\).
In particular, they are 3-graded with respect to the adjoint representation of \(h\).
For the purpose of constructing nets of standard subspaces, it is reasonable to assume that \(\L(S_\tV)\) spans \(\g\) and that \(\tau = e^{i\pi\ad(h)}\), i.e.\ that \(\g^\tau = \g_0(h)\) and \(\g^{-\tau} = \g_{-1}(h) + \g_1(h)\).
In this case, the subsemigroup \(S_\tV\) is given by
\begin{equation}
  \label{eq:sv-3grad}
  S_\tV = \exp(C_-(\pi, \tau, h))G_\tV\exp(C_+(\pi, \tau, h)) \quad \text{with} \quad G_\tV = \{g \in G : U(g)\tV = \tV\}
\end{equation}
(cf.\ \cite[Thm.\ 3.4]{Ne20}).

In \cite{NO21}, a new method for constructing covariant nets of real subspaces indexed by open subsets of \(G\) has been developed in this context:
Let \(\cH^{-\infty}\) be the space of antilinear functionals on the space of smooth vectors \(\cH^{\infty}\) of \((\pi, \cH)\) which are continuous with respect to the \(C^\infty\)-topology on \(\cH^\infty\) (cf.\ Appendix \ref{sec:app-nets-smooth}).
We call the elements of \(\cH^{-\infty}\) \emph{distribution vectors of \((\pi, \cH)\)}.
For every open subset \(\cO \subset G\) and a fixed real subspace \(\sE \subset \cH^{-\infty}\) of distribution vectors of \((\pi, \cH)\), let
\[\sH_\sE(\cO) := \oline{\spann_\R \{\pi^{-\infty}(\varphi)\eta : \varphi \in C_c^\infty(\cO, \R), \eta \in \sE\}}.\]
Here, \(\pi^{-\infty}(\varphi)\) is defined on \(\cH^{-\infty}\) by
\[\pi^{-\infty}(\varphi) = \int_G \varphi(g)\pi^{-\infty}(g)\, d\mu_G(g), \quad \text{where} \quad (\pi^{-\infty}(g)\eta)(\xi) := \eta(\pi(g^{-1})\xi)\]
for \(\eta \in \cH^{-\infty}, \xi \in \cH^\infty,\) and \(\mu_G\) denotes a left invariant Haar measure on \(G\).
This assignment leads to a covariant net of closed real subspaces with the isotony property.
Sufficient conditions for \(\sH_\sE(\cO)\) to be cyclic, i.e.\ to span a dense complex subspace of \(\cH\), or even be standard can be found in \cite[Sec.\ 3]{NO21}.
In \cite[Thm.\ 5.3]{NO21}, it is shown that, if \(G\) is simple hermitian and of tube type and if \((\pi, \cH)\) extends to an irreducible antiunitary representation of \(G_\tau\) for which \(C_\pi\) is generating, then there exists a real subspace \(\sE \subset \cH^{-\infty}\) such that the standard subspace \(\tV\) specified by \((\pi, \cH)\) and \(h\) coincides with \(\sH_\sE(S)\), where \(S\) is the interior of the identity component of the subsemigroup \(S_\tV\).
Moreover, the subsemigroup \(S\) can be interpreted as a ``wedge domain'' in the sense that, for every \(g \in G\) and every open subset \(\emptyset \neq \cO \subset G\) with \(\cO \subset gS\), the subspace \(\sH_\sE(\cO)\) is standard.
For general open subsets \(\cO \subset G\), the subspace \(\sH_\sE(\cO)\) is cyclic.
This result also extends to Lie groups with quasihermitian reductive Lie algebras, i.e.\ direct sums of compact and hermitian simple Lie algebras (cf.\ \cite[Rem.\ 5.4]{NO21}).

In this paper, we show that this construction can be extended further to Lie groups \(G\) of the following form:
We assume that the pointed positive cone \(C_\pi\) of the irreducible representation \((\pi, \cH)\) generates a hyperplane ideal \(\g_C\) of \(\g\) such that \(\g = \g_C + \R h\), where \(h\) induces a 3-grading on \(\g\) via the adjoint representation of the form
\[\g = \g_{-1}(h) \oplus \g_0(h) \oplus \g_1(h).\]
We call such an element \(h\) an \emph{Euler element of \(\g\)} and the derivation \(\ad_{\g_C}(h)\) an \emph{Euler derivation of \(\g_C\)}, which is motivated by the relation of these elements to linear Euler vector fields on certain homogeneous spaces (cf.\ \cite{MN21}).
A Lie algebra such as \(\g_C\) which contains a pointed generating invariant closed convex cone is called \emph{admissible} (cf.\ \cite{Ne00}).

The reason for considering extended admissible Lie algebras \(\g \cong \g_C \rtimes \R h\) instead of admissible Lie algebras is the following:
In the case where \(\g_C\) is non-reductive, which is the case that we are interested in in this paper, the Euler derivation \(\ad_{\g_C}(h)\) is in general not inner.
This follows from the description of Euler derivations in \cite{Oeh20b}.

For the construction of nets of standard subspaces beyond the reductive case in \cite{NO21}, we therefore have to consider extended admissible Lie algebras \(\g_C \rtimes \R h\).

The construction of nets of standard subspaces requires us to study the irreducible unitary representations of Lie groups \(G\) as above and the existence of antiunitary extensions.
Since we are going to use the well-known Mackey method to tackle this problem, the orbits of the automorphism group \((e^{th})_{t \in \R}\) on the characters \(\hat Z\) on the center \(Z\) of \(G_C := \la \exp_G(\g_C) \ra \subset G\) will play an important role in classifying unitary representations of \(G\).

In this paper, we focus on the easiest non-reductive case and assume that the center of \(\g_C\) is one-dimensional and that the Euler element \(h\) acts non-trivially on it, so that \(\hat Z\) decomposes into three orbits under the action of \((e^{th})_{t \in \R}\).
In the higher-dimensional case, the representation theory of \(G\) and the structure of the Euler elements \(h\) become more complicated, but we expect that many of the arguments in this paper can be generalized to this case. We discuss this problem in more detail in Section \ref{sec:nets-adm-perspectives} and also in the upcoming paper \cite{Oeh21}.

\subsection*{Contents of the paper}

In Section \ref{sec:nets-std-preliminaries}, we recall the definition of an admissible Lie algebra, resp.\ the definition of an admissible Lie group, and of a unitary highest weight representation.
Admissible Lie algebras can be described in terms of a construction due to K.\ Spindler (cf.\ \cite{Sp88}).
On every admissible Lie algebra that is parametrized in this way, there exists a natural homomorphism into the Jacobi algebra
\[\hsp(\R^{2n}) := \heis(\R^{2n}) \rtimes \sp(2n,\R) \quad \text{for some } n \in \N,\]
which is also admissible.
This allows us to tackle all central questions of this paper for the Jacobi algebra first and then use this homomorphism to discuss the general case.
All the results that are necessary for this construction are recalled in Section \ref{sec:nets-std-preliminaries-adm-hwr}.
Section \ref{sec:euler-elements} recalls how Euler derivations (i.e.\ derivations \(D\) of a Lie algebra \(\g\) that induce a 3-grading on \(\g\) via the adjoint representation) of admissible Lie algebras can be described in terms of Spindler's construction.
In the case of the Jacobi algebra, the corresponding extended Lie algebra \(\hsp(\R^{2n}) \rtimes \R D\) for every non-trivial Euler derivation is isomorphic to the \emph{conformal Jacobi algebra}
\[\hcsp(\R^{2n}) := \heis(\R^{2n}) \rtimes (\sp(2n,\R) \times \R),\]
where \(r \in \R\) acts on \(\heis(\R^{2n})\) by \(r.(v,z) := (rv, 2rz)\) for \((v,z) \in \R^{2n} \times \R\).

One interesting property of the conformal Jacobi algebra \(\hcsp(\R^{2n})\) is that it can be embedded as a maximal parabolic subalgebra into \(\sp(2n+2,\R)\). We show, in a more general setting, that this embedding and the Euler element of \(\hcsp(\R^{2n})\) can be chosen in such a way that the Euler element in the Jacobi algebra is mapped to an Euler element of \(\sp(2n+2,\R)\) (cf. Theorem \ref{thm:h1-emb-euler-elements}).

In Section \ref{sec:pos-energy-rep-ext-adm}, we discuss the irreducible unitary representations of 1-connected Lie groups \(G\) with Lie algebra \(\g = \g_C \rtimes \R D\), where \(D\) is an Euler derivation of the admissible Lie algebra \(\g_C\).
It is well-known that irreducible representations of the 1-connected Lie group \(\tilde\HSp(\R^{2n})\) with Lie algebra \(\hsp(\R^{2n})\) can be factorized into a tensor product of the extended metaplectic representation and an irreducible representation of \(\tilde\Sp(2n,\R)\), which is the 1-connected Lie group with Lie algebra \(\sp(2n,\R)\) (cf.\ \cite{Sa71}).
The main result of Section \ref{sec:pos-energy-rep-ext-adm} is that a similar result holds for the irreducible unitary positive energy representations of \(G\) (cf.\ Theorem \ref{thm:rep-adm-pos-factors}).
To this end, we discuss the unitary representations of the semidirect product \(N := \Heis(\R^{2n}) \rtimes_\alpha \R_+^\times\), where
\[\alpha_r(v,z) := (rv, r^2 z) \quad \text{for} \quad r \in \R_+^\times \text{ and } (v,z) \in \R^{2n} \times \R = \Heis(\R^{2n}).\]
The unitary representation theory of this group is very similar to that of the \(ax+b\)-group in the sense that all unitary strictly positive energy representations are direct sums of irreducible ones (cf.\ Theorem \ref{thm:heis-aff-discrete-decomp}).
Here, we say that a unitary representation has strictly positive energy if the generator of the unitary one-parameter group of the center of the subgroup \(\Heis(\R^{2n})\) has strictly positive spectrum.
We show that only one irreducible representation of \(N\) with strictly positive energy exists (cf.\ Lemma \ref{lem:conf-heis-irrep}).
This irreducible representation has an extension to \(\tilde\HCSp(\R^{2n})\), i.e.\ the universal covering of \(\hcsp(\R^{2n})\), via the metaplectic representation, which we denote by \((\nu, \cH_\nu)\).

In Section \ref{sec:nets-std-adm}, we construct nets of standard subspaces indexed by open regions on the Lie groups \(G\) introduced earlier, using the methods developed in \cite{NO21}.
The main difficulty here is to construct a suitable subspace \(\sE \subset \cH^{-\infty}\) of distribution vectors.
In the case of the representation \((\nu, \cH_\nu)\) and \(G = \tilde\HCSp(\R^{2n})\), we present two different approaches:

The first one is to use the embedding of \(\hcsp(\R^{2n})\) into \(\sp(2n+2,\R)\).
The representation \((\nu, \cH_\nu)\) can be extended to a highest weight representation of \(\tilde\Sp(2n + 2,\R)\) (cf.\ Proposition \ref{prop:rep-extjac-metaplectic-sp-ext}).
This is a consequence of results from \cite{HNO94},\cite{HNO96a}, and \cite{HNO96b}.
We apply the construction in \cite{NO21} to this extended representation and show that the distribution subspace \(\sE_P\) that we obtain for the extended representation is also a suitable distribution subspace for \((\nu, \cH_\nu)\) (cf.\ Theorem \ref{thm:std-dist-parabolic-embedding}). The distribution subspace is invariant under the subgroup \(G\) with Lie algebra \(\g_0(h) + \g_{-1}(h)\).

For the second approach, we construct a different distribution subspace \(\sE_\nu\) for \((\nu, \cH_\nu)\) which is invariant under \(\exp(\sp(2n,\R)_1(h))\) (cf.\ Section \ref{sec:nets-std-H-nu}). Using the Factorization Theorem \ref{thm:rep-adm-pos-factors} and results on tensor products of standard subspaces and distribution vectors in Section \ref{sec:std-tensor-dist}, we then construct suitable distribution subspaces for tensor product representations of general extended admissible Lie groups \(G\) (cf.\ Theorem \ref{thm:nets-adm-general}).

In Section \ref{sec:nets-hom-spaces}, we discuss how both constructions also lead to covariant isotone nets of cyclic subspaces on homogeneous spaces of \(G\).
Here, the invariance properties of the spaces of distribution vectors \(\sE_\nu\) and \(\sE_P\) play an important role in order to determine which open subsets of the homogeneous space correspond to standard subspaces.

While, in the case of the representation \((\nu, \cH_\nu)\) of \(\tilde\HCSp(\R^{2n})\), the standard subspaces \(\sH_{\sE_P}(gS)\) and \(\sH_{\sE_\nu}(gS)\) belonging to wedge domains coincide for all \(g \in G\), it is not clear at the moment whether \(\sH_{\sE_P}(\cO)\) and \(\sH_{\sE_\nu}(\cO)\) coincide for general open subsets \(\cO \subset G\).

\subsection*{Notation and conventions}

\begin{itemize}
  \item For a complex Hilbert space \(\cH\), its scalar product \(\la \cdot, \cdot \ra\) is linear in the second argument and antilinear in the first.
  \item Let \(\g\) be a Lie algebra and let \(x \in \g\). For an \(\ad(x)\)-invariant subspace \(E \subset \g\), we denote by \(E_\lambda(x) := \ker(\ad(x) - \lambda\id_\g) \cap E\) the \(\lambda\)-eigenspace of \(\ad_\g(x)\) in \(E\).
  \item Let \(G\) be a Lie group with Lie algebra \(\g\) and let \(\tau_G \in \Aut(G)\) be an involutive automorphism. Let \(\tau := \L(\tau_G)\). We define
    \[G^{\tau} := \{g \in G : \tau_G(g) = g\} \quad \text{and} \quad \g^{\pm \tau} := \{x \in \g : \tau(x) = \pm x\}.\]
    Furthermore, we define \(G_\tau := G \rtimes \{\1, \tau_G\}\).
\end{itemize}

\section{Admissible Lie groups}
\label{sec:nets-std-preliminaries}

In this section, we introduce the necessary structure theory of admissible Lie algebras and their corresponding Lie groups.

\subsection{Admissible Lie groups and unitary highest weight representations} 
\label{sec:nets-std-preliminaries-adm-hwr}

\begin{definition}
  \label{def:liealg-adm}
  Let \(\g\) be a real finite-dimensional Lie algebra. Then \(\g\) is called \emph{admissible} if there exists a generating (\(e^{\ad \g}\)-)invariant closed convex cone \(C\) with \(H(C) := \{x \in \g : C + x = C\} = \{0\}\), i.e.\ \(C\) contains no affine lines. A Lie group \(G\) is called \emph{admissible} if its Lie algebra is admissible.
\end{definition}

We recall a few definitions in the context of unitary highest weight representations of admissible Lie groups (cf.\ \cite{Ne00}):

\begin{definition}
  \label{def:u-hrw}
  \begin{enumerate}
    \item Let \(G\) be a connected Lie group with Lie algebra \(\g\) and let \(\fk\) be a maximal compactly embedded subalgebra of \(\g\).
      Suppose that there exists a compactly embedded Cartan subalgebra \(\ft\) of \(\g\) contained in \(\fk\).
      Then \(\ft_\C\) is a Cartan subalgebra of \(\g_\C\). For \(\alpha \in \ft_\C^*\), we denote the corresponding root space of \(\alpha\) by \(\g_\C^\alpha\) and denote by
      \[\Delta := \Delta(\g_\C, \ft_\C) := \{\alpha \in \ft_\C^* \setminus \{0\} : \g_\C^\alpha \neq \{0\}\}\]
      the corresponding set of roots. We say that a root \(\alpha \in \Delta\) is \emph{compact} if \(\g_\C^\alpha \subset \fk_\C\). We denote the set of compact roots by \(\Delta_k\) and the set of non-compact roots by \(\Delta_p := \Delta \setminus \Delta_k\). 

      Let \(\Delta^+ \subset \Delta\) be a positive system and \(X_0 \in i\ft\) such that \(\Delta^+ = \{\alpha \in \Delta : \alpha(X_0) > 0\}\).
      Then we say that \(\Delta^+\) is \emph{adapted} if \(\beta(X_0) > \alpha(X_0)\) for all \(\alpha \in \Delta_k\) and \(\beta \in \Delta_p\).

      If an adapted positive system exists, then \(\g\) is called \emph{quasihermitian}, which is equivalent to \(\fz_\g(\fz(\fk)) = \fk\) (cf.\ \cite[Prop.\ VII.2.14]{Ne00}). A simple non-compact quasihermitian Lie algebra is called \emph{hermitian}, in which case \(\fz(\fk)\) is one-dimensional.  Every admissible Lie algebra is quasihermitian (cf.\ \cite[Thm.\ VII.3.10]{Ne00}).

    \item Suppose now that \(\g\) is quasihermitian and that \(\Delta^+\) is adapted. Let \(K := \la \exp \fk \ra \subset G\) be the corresponding maximal compactly embedded subgroup of \(G\). Consider a strongly continuous unitary representation \((\pi, \cH)\) of \(G\). The space of \emph{\(K\)-finite smooth vectors of \((\pi,\cH)\)} is defined as
      \[\cH^K := \{v \in \cH^\infty : \dim \spann(\pi(K)v) < \infty\}.\]
      The representation \((\pi, \cH)\) is called a \emph{unitary highest weight representation} if \(\cH^K\) is a highest weight module of \(\g_\C\) with respect to \(\Delta^+\). Conversely, if \(V\) is a highest weight module of \(\g_\C\) with respect to \(\Delta^+\), then \(V\) is called \emph{unitarizable} if there exists a strongly continuous unitary representation \((\pi, \cH)\) of \(G\) such that \(\cH^K = V\).

    \item Suppose that \(\g\) is admissible. For an adapted positive system \(\Delta^+\), we define the cones
      \[\Cmin := \mathrm{cone } \{i[\oline{x}, x] : x \in \g_\C^\alpha, \alpha \in \Delta_p^+\} \subset \ft \quad \text{and}\]
      \[\Cmax := \{x \in \ft : (\forall \alpha \in \Delta_p^+)\, i\alpha(x) \geq 0\}.\]
      For every pointed generating invariant closed convex cone \(W \subset \g\), there exists an adapted positive system \(\Delta^+\) such that \(\Cmin \subset W \cap \ft \subset \Cmax\). Two pointed generating invariant closed convex cones coincide if and only if their intersections with \(\ft\) coincide. Having fixed the \(\fk\)-adapted positive system \(\Delta^+\), we denote by \(\Wmin \subset \g\) the unique pointed invariant closed convex cone with \(\Cmin = \Wmin \cap \ft\) and by \(\Wmax \subset \g\) the unique generating invariant closed convex cone with \(\Cmax = \Wmax \cap \ft\) (cf.\ \cite[Ch.\ III]{HHL89}).
      If \(\g\) is simple and non-compact, then \(\Wmin\) and \(\Wmax\) are both pointed and generating and unique up to sign (cf. \cite[Prop.\ III.2.18]{HHL89}).
  \end{enumerate}
\end{definition}

Every admissible Lie algebra can be described in terms of \emph{Spindler's construction} \cite{Sp88}: Consider the quadruple \((\fl, V, \fz, \beta)\), where
\begin{enumerate}[label=(S\arabic*)]
  \item \(\fl\) is a real reductive Lie algebra,
  \item \(V\) is a real \(\fl\)-module,
  \item \(\fz\) is a real vector space, and
  \item \(\beta: V \times V \rightarrow \fz\) is a skew-symmetric \(\fl\)-invariant bilinear map, i.e.\
    \begin{equation}
      \label{eq:spindler-beta-invariance}
      \beta(x.v, w) + \beta(v, x.w) = 0 \quad \text{for all } v,w \in V, x \in \fl.
    \end{equation}
\end{enumerate}
Then \(\g(\fl, V, \fz, \beta) := V \times \fz \times \fl\) is a Lie algebra with the bracket
\[[(v, z, x), (w, z', y)] := (x.w - y.v, \beta(v,w), [x,y]).\]

\begin{thm}
  \label{thm:spindler-adm-symp-mod}
  \begin{enumerate}
    \item Let \(\g\) be an admissible Lie algebra. Let \(\fu\) be the maximal nilpotent ideal of \(\g\) and let \(\ft\) be a compactly embedded Cartan subalgebra of \(\g\). Then \([\fu,\fu] \subset \fz(\g)\). Moreover, there exists a \(\ft\)-invariant reductive quasihermitian subalgebra \(\fl \subset \g\) such that \(V := [\ft, \fu]\) is an \(\fl\)-invariant submodule with respect to the adjoint representation of \(\fl\) with the following properties:
\begin{enumerate}
  \item \(V\) and \(\fl\) satisfy \eqref{eq:spindler-beta-invariance} with respect to \(\beta := [\cdot, \cdot]\lvert_{V \times V}\) and we have \(\fz_{\fz(\fl)}(V) = \{0\}\).
  \item \(\ft_\fl := \ft \cap \fl\) is a compactly embedded Cartan subalgebra of \(\fl\) and we have \(\fz_V(\ft_\fl) = \{0\}\).
  \item We have \(\g \cong \g(\fl, V, \fz(\g), \beta)\).
\end{enumerate}
\item Conversely, let \(\g = \g(\fl, V, \fz, \beta)\) with \(\fz_{\fz(\fl)}(V)  = \{0\}\). Suppose that there exists a compactly embedded Cartan subalgebra \(\ft_\fl \subset \fl\) with \(\fz_V(\ft_\fl) = \{0\}\). Then \(\g\) is admissible if and only if \(\fl\) is quasihermitian and there exists \(f \in \fz^*\) and \(x \in \fl\) such that the Hamiltonian function
  \begin{equation}
    \label{eq:spindler-pos-def}
    (f \circ \beta)_x : V \rightarrow \R, \quad v \mapsto (f \circ \beta)(x.v, v),
  \end{equation}
  is positive definite.
  \end{enumerate}
\end{thm}
\begin{proof}
  cf.\ \cite[Thm.\ VII.2.26]{Ne00} and \cite[Thm.\ VIII.2.7]{Ne00}.
\end{proof}

From now on, if \(\g\) is an admissible Lie algebra, we say that \(\g\) is \emph{defined as \(\g = \g(\fl, V, \fz, \beta)\) in terms of Spindler's construction} if
\begin{itemize}
  \item \(\g = \g(\fl, V, \fz, \beta)\) and \((\fl,V, \fz, \beta)\) satisfies (S1)-(S4),
  \item \(\fz_{\fz(\fl)}(V) = \{0\}\), and
  \item there exists a compactly embedded Cartan subalgebra \(\ft_\fl \subset \fl\) with \(\fz_V(\ft_\fl) = \{0\}\).
\end{itemize}
In this case, we have \(\fz = \fz(\g)\) and the condition of Theorem \ref{thm:spindler-adm-symp-mod}(b) is satisfied.
Such a quadruple \((\fl, V, \fz, \beta)\) can be constructed by choosing any subalgebra \(\fl \subset \g\) such that \(\g = \fu \rtimes \fl\), where \(\fu\) is the maximal nilpotent ideal of \(\g\). Then, \(\fl\) is reductive and quasihermitian and, for every compactly embedded Cartan subalgebra \(\ft_\fs \subset [\fl,\fl]\), the subalgebra \(\ft := \fz(\fg) + \fz(\fl) + \ft_s\) is a compactly embedded Cartan subalgebra of \(\g\). The subspace \(V := [\ft_\fl, \fu]\) is an \(\fl\)-submodule and satisfies \([V,V] \subset \fz(\g)\) (cf.\ \cite[Thm.\ 2.24, Cor.\ 2.25]{Oeh20b}).

Suppose that \(\g = \g(\fl, V, \fz, \beta)\) is admissible. Then \eqref{eq:spindler-beta-invariance} implies that the adjoint representation of \(\fl\) restricted to \(V\) induces a homomorphism from \(\fl\) into the Lie algebra
\[\sp(V,\beta) := \{x \in \gl(V) : (\forall v,w \in V)\, \beta(xv, w) + \beta(v, xw) = 0\}.\]
Hence, we call \((V,\beta)\) a \emph{symplectic \(\fl\)-module}.
In view of the positivity condition \eqref{eq:spindler-pos-def}, we also call \((V,\beta)\) a \emph{symplectic \(\fl\)-module of convex type}. Notice that \(\beta\) is automatically non-degenerate in this case.

Under these assumptions, the Lie algebra \(\sp(V,\beta)\) is admissible and reductive and the \emph{generalized Jacobi algebra}
\[\hsp(V,\beta) := \g(\sp(V,\beta), V, \fz, \beta),\]
where \(\sp(V,\beta)\) acts naturally on \(V\), is also admissible (cf.\ \cite[Thm.\ 2.22]{Oeh20b}). If \(\dim \fz = 1\), i.e.\ \(\beta\) is a symplectic form, then \(\hsp(V,\beta)\) is called the \emph{Jacobi algebra}.
Let \(f \in \fz^*\) such that \((f \circ \beta)_x\) is positive definite for some \(x \in \fl\) and set \(\Omega := f \circ \beta\).
Then the homomorphism on \(\fl\) corresponding to the representation on \(V\), which we denote by \(\rho: \fl \rightarrow \sp(V,\beta)\), extends naturally to a homomorphism
\begin{equation}
  \label{eq:adm-jacobi-homom}
  \gamma_f : \g \rightarrow \hsp(V, \Omega), \quad (v,z,x) \mapsto (v, f(z), \rho(x)).
\end{equation}
On \(\hsp(V,\Omega)\), we have a pointed generating invariant closed convex cone
\begin{equation}
  \label{eq:hsp-plus-cone}
  \hsp(V,\Omega)_+ := \{(v,z,x) \in \hsp(V,\Omega) : (\forall w \in V)\, \Omega(xw, w) + \Omega(v, w) + z \geq 0\}.
\end{equation}
We denote its preimage in \(\g\) under \(\gamma_f\) by \(W_f\), which is a generating invariant closed convex cone with \(H(W_f) \subset \ker(\gamma_f)\) (cf.\ \cite[Thm.\ VIII.2.7]{Ne00}).

\begin{definition}
  \label{def:h1-homom}
  A necessary condition for the existence of elements \(x \in \fl\) and \(f \in \fz^*\) for which \eqref{eq:spindler-pos-def} is positive definite is the \((H_1)\)-condition on the homomorphism \(\fl \rightarrow \sp(V,f \circ \beta)\) induced by the adjoint representation, which is defined as follows: In every reductive admissible Lie algebra \(\g\), there exists an element \(H_0\), called \(H\)-element, with the property that \(\ker(\ad H_0)\) is a maximal compactly embedded subalgebra and \(\ad H_0\) is semisimple on \(\g_\C\) with \(\spec(\ad H_0) = \{0, i, -i\}\).
  In this case, we say that \((\g, H_0)\) is \emph{of hermitian type} (cf.\ \cite{Sa80}).
  A homomorphism \(\kappa : (\g, H_0) \rightarrow (\g', H_0')\) is called an \emph{\((H_1)\)-homomorphism} if \(\kappa \circ \ad H_0 = \ad H_0' \circ \kappa\) and it is called an \emph{\((H_2)\)-homomorphism} if \(\kappa(H_0) = H_0'\).

  One can show that, if \eqref{eq:spindler-pos-def} holds, then the corresponding homomorphism \(\fl \rightarrow \sp(V, f \circ \beta)\) is an \((H_1)\)-homomorphism for some \(H\)-elements of \(\fl\) and \(\sp(V,f \circ \beta)\) (cf.\ \cite[Thm.\ IV.6]{Ne94}).
\end{definition}

\subsection{Euler elements and parabolic subalgebras of admissible Lie algebras}
\label{sec:euler-elements}

\begin{definition}
  \label{def:euler-element}
  Let \(\g\) be a real or complex Lie algebra.
  \begin{enumerate}
\item We call \(h \in \g\) an \emph{Euler element} if \(\ad h\) is diagonalizable with \(\spec(\ad h) \subset \{-1, 0, 1\}\).
\item A derivation \(D \in \der(\g)\) is called an \emph{Euler derivation} if \((0,D)\) is an Euler element of \(\g \rtimes \R D\).
  \end{enumerate}
\end{definition}

We briefly recall the classification of Euler derivations in admissible Lie algebras.
Suppose first that \(\g_C\) is a semisimple admissible Lie algebra.
Since all derivations of \(\g_C\) are inner, we only have to describe all Euler elements of \(\g_C\).
The Lie algebra \(\g_C\) can be decomposed as \(\g_C = \bigoplus_{k=0}^n \g_k\), where \(\g_0\) is compact and \(\g_k\) is hermitian simple for \(k=1,\ldots,n\).
An element \(h = \sum_{k=0}^n h_k\) with \(h_k \in \g_k\) is an Euler element if and only if \(h_0 = 0\) and \(h_k\) is an Euler element of \(\g_k\) for \(k=1,\ldots,n\).
This reduces the classification of Euler elements of \(\g_C\) to the classification of Euler elements of simple non-compact Lie algebras, which has been carried out in \cite[Thm.\ 3.10]{MN21} for general real simple Lie algebras.
In the case where \(\g_C\) is hermitian simple, it says that non-trivial Euler elements are unique up to conjugation by inner automorphisms and they exist if and only if \(\g_C\) is of tube type.

\begin{lem}
  \label{lem:euler-cone-intersection-reductive}
  Let \(\fl\) be a quasihermitian reductive Lie algebra and let \(W \subset \fl\) be a pointed invariant closed convex cone. If \(h \in \fl\) is an Euler element of \(\fl\) and \(\Wmin\) is a minimal pointed invariant closed convex cone with \(\Wmin \subset W\), then \(\fl_{\pm 1}(h) \cap W = \fl_{\pm 1}(h) \cap \Wmin\).
\end{lem}
\begin{proof}
  We decompose \(\fl\) into ideals \(\fl = \fl_0 \oplus \bigoplus_{k=1}^n \fl_k\), where \(\fl_0\) is compact and \(\fl_k\) hermitian simple for \(k=1,\ldots,n\).
  Then \(h\) can be written as \(h = h_0 + \sum_{k=1}^n h_k\) with \(h_0 \in \fl_0\) and \(h_k \in \fl_k\). The classification of Euler elements for simple Lie algebras (cf.\ \cite[Thm.\ 3.10]{MN21}) then shows that \(h_0 \in \fz(\fl)\) and \(h_k \neq 0\) only if \(\fl_k\) is of tube type.
  Moreover, \(\Wmin\) is contained in \(\fl_1 \oplus \ldots \oplus \fl_n\).
  We may thus assume that \(\fl\) is semisimple and that all simple ideals are hermitian and of tube type. 
  The minimal cone \(\Wmin\) is adapted to the decomposition of \(\fl\) into simple ideals in the sense that \(\Wmin = \sum_{k=1}^n \Wmin \cap \fl_k\), and \(\Wmin \cap \fl_k\) is one of the up to sign unique minimal pointed generating invariant closed convex cones of \(\fl_k\) (cf.\ Definition \ref{def:u-hrw}(c)).
  In particular, \(\Wmin\) and therefore \(W\) are generating.

  Let \(\Wmax \subset \fl\) be a maximal pointed generating invariant closed convex cone with \(\Wmin \subset W \subset \Wmax\). Then the cone \(\Wmax\) can also be written as \(\Wmax = \sum_{k=1}^n \Wmax \cap \fl_k\).
  By \cite[Lem.\ 2.28]{Oeh20a}, we have \(\Wmin \cap (\fl_k)_{\pm 1}(h_k) = \Wmax \cap (\fl_k)_{\pm 1}(h_k)\) for all \(k=1,\ldots,n\). In particular, this shows that
  \begin{align*}
\Wmin \cap \fl_{\pm 1}(h) &\subset W \cap \fl_{\pm 1}(h) \subset \Wmax \cap \fl_{\pm 1}(h) = \sum_{k=1}^n \Wmax \cap (\fl_k)_{\pm 1}(h_k) = \sum_{k=1}^n \Wmin \cap (\fl_k)_{\pm 1}(h_k)
  \end{align*}
  and therefore \(\Wmin \cap \fl_{\pm 1}(h) = W \cap \fl_{\pm 1}(h)\).
\end{proof}

Consider now an admissible non-reductive Lie algebra \(\g_C\).
We only recall the classification of Euler elements here in the case where \(\dim \fz(\g_C) = 1\) (cf.\ \cite[Cor.\ 3.15]{Oeh20b}) and refer to \cite[Thm. 3.14]{Oeh20b} for the general case: Let \(D \in \der(\g_C)\) be a non-trivial Euler derivation.
We assume that \(D\) is non-trivial on the radical of \(\g_C\). Otherwise, we could decompose \(\g_C\) into \(D\)-invariant ideals \(\g' \oplus \fs'\), where \(\fs'\) is semisimple and \(D\) vanishes on \(\g'\).
Then \(\g_C\) is of the form \(\g_C = \g(\fl, V, \fz, \beta)\) in terms of Spindler's construction and \(D\) acts on \(\g_C\) by
\begin{equation}
  \label{eq:euler-adm}
  D(v,z,x) = (\tfrac{\lambda}{2}v + h_\fl.v, \lambda z, [h_\fl,x]), \quad (v,z,x) \in V \times \fz \times \fl = \g_C,
\end{equation}
where \(\lambda \in \{-1, 1\}\) and \(h_\fl\) is a non-trivial Euler element of \(\fl\).
In particular, we see that \(2h_\fl\) acts as an antisymplectic involution on \(V\).
Let \(V_{\pm 1}\) denote the \(\pm 1\)-eigenspaces of \(2h_\fl\) on \(V\).
Both signs of \(\lambda\) in the definition of \(D\) lead to isomorphic extended Lie algebras \(\g_C \rtimes \R D\), so that we may without loss of generality assume that \(\lambda = 1\).
In this case, the 3-grading induced by \(D\) is given by
\[\g_{C,-1} = \fl_{-1}(h_\fl), \quad \g_{C,0} = V_{-1} \oplus\fl_0(h_\fl), \quad \g_{C,1} = V_1 \oplus \fz \oplus \fl_1(h_\fl).\]
If \(\rho: \fl \rightarrow \sp(V,\beta)\) denotes the homomorphism corresponding to the representation of \(\fl\), then this implies that \(\rho(h_\fl)\) is an Euler element of \(\sp(V,\Omega)\) (cf.\ \cite[Ex.\ 2.4]{Oeh20b}).
Let \(D_V := D - \ad_{\g_C} h_\fl\) and set \(\g := \g_C \rtimes \R D_V\). Then \(D \in \g\) is an Euler element of \(\g\).

In the case where \(\fl = \sp(V,\beta)\), i.e.\ where \(\g_C = \hsp(V,\beta)\) is the Jacobi algebra, we call \(\g\) the \emph{conformal Jacobi algebra} and denote it by \(\hcsp(V,\beta)\). The reason for this definition is that \(\hcsp(V,\beta) \cong \heis(V,\beta) \rtimes \csp(V,\beta)\), where \(\csp(V,\beta) = \sp(V,\beta) \oplus \R\) is the conformal symplectic Lie algebra.

The above discussion leads to the following result below.
\begin{lem}
  \label{lem:adm-euler-hsp-homom}
  Let \(\g_C = \g(\fl,V,\fz,\beta)\) be an admissible non-reductive Lie algebra with \(\dim \fz = 1\) and let \(D \in \der(\g_C)\) be an Euler derivation that is defined in terms of an Euler element \(0 \neq h_\fl \in \fl\) of \(\fl\) as in \eqref{eq:euler-adm}. Let \(f \in \fz^*\) such that, for \(\Omega := f \circ \beta\), \((V, \Omega)\) is a symplectic \(\fl\)-module of convex type.
  Let \(\rho: \fl \rightarrow \sp(V,\Omega)\) be the homomorphism corresponding to the representation of \(\fl\) and let \(\g := \g_C \rtimes \R D_V\), where \(D_V := D - \ad_{\g_C}h_\fl\).
  Then the natural homomorphism \(\gamma_f: \g_C \rightarrow \hsp(V,\Omega)\) (cf.\ \eqref{eq:adm-jacobi-homom}) extends to a homomorphism
  \begin{equation}
    \label{eq:adm-jacobi-homom-ext}
    \tilde\gamma_f : \g \rightarrow \hcsp(V,\Omega), \quad \tilde\gamma_f\lvert_{\g_C} := \gamma_f,\quad \tilde\gamma_f(D_V) := D_V,
  \end{equation}
  which maps \(h_\fl + D_V\) to an Euler element of \(\hcsp(V,\Omega)\).
\end{lem}

\begin{remark}
  Notice that, in Lemma \ref{lem:adm-euler-hsp-homom}, we only assume the existence of the non-trivial Euler derivation, but not that the homomorphism \eqref{eq:adm-jacobi-homom-ext} injective.

  In general, injectivity of \(\gamma_f\) can be achieved as follows:
  Since \(\fz\) is one-dimensional by assumption and \(f\) is non-zero, the kernel of \(\gamma_f\) is given by the ideal \(\fl' := \{x \in \fl : x.V = \{0\}\}\) of \(\g_C\).
  Let \(\fl_V \subset \fl\) be an ideal of \(\fl\) complementing \(\fl'\).
  Then \(\g_C = \g(\fl_V, V, \fz, \beta) \oplus \fl'\) is a direct sum of admissible Lie algebras, and the restriction of \(\gamma_f\) to \(\g(\fl_V,V,\fz,\beta)\) is injective.
\end{remark}

We can embed the conformal Jacobi algebra into a symplectic Lie algebra in such a way that the Euler element of the Jacobi algebra constructed above is an Euler element of the larger Lie algebra. In fact, this can be done in a more general setting, as we will show in this section.

For the remainder of this section, let \(\g\) be a hermitian simple Lie algebra of tube type. Fix an \(H\)-element \(H_0 \in \g\).
We define
\[H := \pmat{1 & 0 \\ 0 & -1}, ~ X := \pmat{0 & 1 \\ 0 & 0}, ~ Y := \pmat{0 & 0 \\ 1 & 0}, ~ \text{and} ~ U := \pmat{0 & 1 \\ -1 & 0}.\]
Let  \(\kappa : (\fsl(2,\R), \frac{1}{2}U) \rightarrow (\g, H_0)\) be an \((H_1)\)-homomorphism. We define \(A_\kappa := \kappa(A)\) for \(A \in \{H, X, Y, U\} \subset \fsl(2,\R)\).
Then \(H_\kappa\) induces a 5-grading on \(\g\), i.e.\ we have
\[\g = \g_{-2}(H_\kappa) \oplus \g_{-1}(H_\kappa) \oplus \g_0(H_\kappa) \oplus \g_1(H_\kappa) \oplus \g_2(H_\kappa).\]
(cf.\ \cite[Ch.\ III, Lem.\ 1.2]{Sa80}).
We recall from \cite[p.\ 110f.]{Sa80} that \(\kappa\) can be chosen in such a way that \(\dim \g_{\pm 2}(H_\kappa) = 1\).
We call such a homomorphism an \emph{\((H_1)\)-homomorphism of multiplicity \(1\)}.
From now on, unless stated otherwise, we assume that \(\kappa\) has this property.
We denote by \(\theta\) the Cartan involution of \(\g\) determined by \(\g^\theta = \ker(\ad H_0)\). Notice that we have \(\theta(H_\kappa) = -H_\kappa\) and \(\theta(X_\kappa) = -Y_\kappa\). 

\begin{rem}
  \label{rem:sl2-triples-orbits}
  The elements \(X_\kappa\) and \(Y_\kappa\) are nilpotent elements of \(\g\), i.e.\ their adjoint representations are nilpotent endomorphisms. More precisely, they are nilpotent elements of convex type, i.e.\ their adjoint orbits \(\cO_{X_\kappa}\) and \(\cO_{Y_\kappa}\) are contained in the up to sign unique minimal pointed generating invariant cone \(\pm\Wmin\) (cf.\ Definition \ref{def:u-hrw}).
  Suppose that the sign of \(\Wmin\) is chosen so that \(X_\kappa \in \Wmin\). Then \(Y_\kappa \in -\Wmin\) because \(\Wmin\) is \(\theta\)-invariant and \(\theta(X_\kappa) = -Y_\kappa\).
  We identify \(\g^*\) with \(\g\) via the Cartan--Killing form of \(\g\).
  Then, since \(\Wmin \subset \Wmax = -\Wmin^*\) and the dual cone \(\cO_{X_\kappa}^*\), resp.\ \(\cO_{Y_\kappa}^*\), is pointed and generating, we actually have \(\cO_{X_\kappa}^* = -\cO_{Y_\kappa}^* = -\Wmax\).
\end{rem}

\begin{proposition}
  \label{prop:h1-multone-parabolic-subalg}
  The Lie algebra
\[\fb_\kappa := \g_0(H_\kappa) + \g_1(H_\kappa) + \g_2(H_\kappa)\]
is a maximal parabolic subalgebra of \(\g\).
Moreover, we have \(\g_0(H_\kappa) = \R H_\kappa \oplus \g_\kappa\), where \(\g_\kappa := \fz_\g(\im \kappa)\), and \((\g_\kappa, H_0 - \frac{1}{2}U_\kappa)\) is a Lie algebra of hermitian type.
More precisely, \(\g_\kappa\) is either abelian or reductive with exactly one hermitian simple ideal of tube type and of real rank \(\rk_\R \g - 1\). 
On \(V_\kappa := \g_1(H_\kappa)\), there exists a skew-symmetric bilinear map
\[\beta : V_\kappa \times V_\kappa \rightarrow \fz_\kappa := \g_2(H_\kappa) = \R X_\kappa = \fz(\fb_\kappa) \cong \R\] 
such that \((V_\kappa,\beta)\) is a symplectic \(\g_\kappa\)-module of convex type with respect to the adjoint action. 
In particular, we have
\[\fb_\kappa = \fj_\kappa \rtimes \R H_\kappa \quad \text{with} \quad \fj_\kappa := \g(\g_\kappa, V_\kappa, \fz_\kappa, \beta)\]
in terms of Spindler's construction, and \(\fj_\kappa\) is admissible.
\end{proposition}
\begin{proof}
  \(\fb_\kappa\) is maximal parabolic by \cite[Prop.\ IV.22]{HNO94}.
  By \cite[p.\ 112]{Sa80}, the subspace \(V_\kappa\) is a faithful irreducible \(\g_0(H_\kappa)\)-module with respect to the adjoint action and it satisfies \([V_\kappa,V_\kappa] = \g_2(H_\kappa)\).
  By \cite[Ch.\ III Prop.\ 1.3]{Sa80}, the subalgebra \((\g_\kappa, H_0 - \frac{1}{2}U_\kappa)\) is of hermitian type, and by \cite[Ch.\ III Thm. 3.3]{Sa80}, there exists a symplectic form \(\beta : V_\kappa \times V_\kappa \rightarrow \g_2(H_\kappa)\) such that the adjoint action of \(\g_\kappa\) on \(V_\kappa\) induces an \((H_2)\)-homomorphism \((\g_\kappa, H_0) \rightarrow (\sp(V_\kappa, \beta), I_\kappa)\) for some \(H\)-element \(I_\kappa \in \sp(V_\kappa,\beta)\).
  Combining this with the faithfulness and irreducibility of the module \(V_\kappa\), this shows that \((V_\kappa,\beta)\) is a symplectic \(\g_\kappa\)-module of convex type by \cite[Thm.\ IV.6]{Ne94}.
  Now Theorem \ref{thm:spindler-adm-symp-mod} implies that \(\fj_\kappa = \g(\g_\kappa, V_\kappa, \fz_\kappa, \beta)\) is admissible, and we have \(\fb_\kappa = \fj_\kappa \rtimes \R H_\kappa\) by construction.

  That \(\g_\kappa\) is either abelian or reductive with exactly one simple hermitian ideal of tube type follows from \cite[p.\ 114--118]{Sa80} (see also Table \ref{table:h1-multone-class}).
\end{proof}

\begin{table}[H]
  \centering
  \begin{tabular}{|c|c|}
    \hline
    \(\g\) & \(\g_\kappa\) \\
    \hline
    \(\su(p,p)\, (p \geq 1)\) & \(\fu(p-1,p-1)\)\\
    \hline
    \(\so^*(2n)\, (n \text{ even})\) & \(\so^*(2n - 4) \oplus \fsl(1,\H)\)\\
    \hline
    \(\sp(2n,\R)\) & \(\sp(2n - 2, \R)\)\\
    \hline
    \(\so(2,n)\, (n \geq 3)\) & \(\fsl(2,\R) \oplus \so(n - 2)\)\\
    \hline
    \(\fe_{7(-25)}\) & \(\so(2,10)\)\\
    \hline
  \end{tabular}
  \caption{Lie subalgebras \(\g_\kappa\) for all hermitian simple Lie algebras of tube type and all \((H_1)\)-homomorphisms \(\kappa: \fsl(2,\R) \rightarrow \g\) of multiplicity \(1\). This table is based on \cite[p.\ 114--118]{Sa80}, where \(\g_\kappa\) is classified for all hermitian simple Lie algebras and all multiplicities of \(\kappa\).}
    \label{table:h1-multone-class}
\end{table}

\begin{prop}
  \label{prop:adm-parabolic-emb-cone}
  In the context of {\rm Proposition \ref{prop:h1-multone-parabolic-subalg}}, let \(f := \la Y_\kappa, \cdot \ra \in \fz_\kappa^*\), where \(\la \cdot, \cdot \ra\) denotes the Cartan--Killing form of \(\g\). Let \(W_f \subset \fj_\kappa\) be the preimage of the cone \(\hsp(V_\kappa,\beta)_+\) under the canonical homomorphism \(\gamma_f\) (cf.\ \eqref{eq:adm-jacobi-homom}). Then \(W_f \in \{\pm \Wmax \cap \fj_\kappa\}\), where \(\Wmax \subset \g\) is a maximal pointed generating invariant closed convex cone in \(\g\).
\end{prop}
\begin{proof}
  Consider the positive definite bilinear form \((\cdot, \cdot) := - \la \theta \cdot, \cdot\ra\). Then \(f = (X_\kappa, \cdot)\). Since this formula also makes sense on \(\g\), we regard \(f\) as an element of \(\g^*\) and denote its restriction to \(\fj_\kappa\) by \(f_{\fj_\kappa}\). Notice also that \(\g_\kappa + V_\kappa \subset \ker(f)\) because this space is \((\cdot, \cdot)\)-orthogonal to \(\fz_\kappa\) by \cite[Prop.\ V.2]{HNO94}.

  Let \(G := \Inn(\g)\) and set \(J := \Inn_\g(\fj_\kappa)\).
  By \cite[Prop.\ VIII.2.4]{Ne00}, we have \(W_{f_{\fj_\kappa}} = \cO^*_{f_{\fj_\kappa}}\).
  If we identify \(\fj_\kappa\) with its dual via \(\la \cdot, \cdot\ra\), we obtain
  \[W_{f_{\fj_\kappa}} = \cO^*_{f_{\fj_\kappa}} = \{Y \in \fj_\kappa : \la J(Y_\kappa), Y \ra \subset \R_{\geq 0}\}.\]
  In the following, we identify \(\g^*\) with \(\g\) via the Cartan--Killing form.
  Fix a maximal pointed generating invariant cone \(\Wmax\) in \(\g\) so that \(Y_\kappa \in -\Wmax\) and suppose that \(Y \in \Wmax \cap \fj_\kappa\).
  Then \(Y \in \Wmax = \cO_{Y_\kappa}^*\) (cf.\ Remark \ref{rem:sl2-triples-orbits}). Since
  \[\cO_{Y_\kappa}^* = \{Y \in \g : \la G(Y_\kappa), Y \ra \subset \R_{\geq 0}\} \subset \{Y \in \g : \la J(Y_\kappa), Y \ra \subset \R_{\geq 0}\},\]
  we thus have \(Y \in W_{f_{\fj_\kappa}}\).

  Conversely, let \(Y \in W_{f_{\fj_\kappa}}\). Then we have
  \[\la J(Y_\kappa), Y \ra \subset \R_{\geq 0}.\]
  Since the \(J\)-orbit \(J(Y_\kappa) \subset \g\) is dense in \(\cO_{Y_\kappa} = G(Y_\kappa)\) by \cite[Prop.\ V.9]{HNO94}, we actually have
  \[\la G(Y_\kappa), Y \ra \subset \R_{\geq 0}\]
  and thus \(Y \in \cO_{Y_\kappa}^* = \Wmax\).
\end{proof}

\begin{lem}
  \label{lem:quasiherm-euler-h1-sl2}
  Let \(\g\) be a quasihermitian reductive Lie algebra with \(H\)-element \(H_0\) and let \(h \in [H_0,\g]\) be an Euler element of \(\g\). Then there exists an \((H_1)\)-homomorphism \(\kappa : (\fsl(2,\R), \frac{1}{2}U) \rightarrow (\g, H_0)\) such that \(\ad(\kappa(\frac{1}{2}H)) = \ad(h)\).
\end{lem}
\begin{proof}
  The case \(h = 0\) is trivial. If \(\g\) is hermitian simple and of tube type, then the claim follows from \cite[Rem.\ 2.14]{Oeh20a}.

  Therefore, suppose that \(h \neq 0\) and let \(\g = \fz(\g) \oplus \g_0 \oplus \bigoplus_{k=1}^n \g_k\), where \(\g_0\) is compact and \(\g_k\) is hermitian simple for \(1 \leq k \leq n\).
  Along this decomposition, we have \(H_0 = H_z + \sum_{k=1}^n H_{0,k}\), where \(H_z \in \fz(\g)\) and \(H_{0,k} \in \g_k\) is an \(H\)-element of \(\g_k\) (cf.\ \cite[Rem.\ II.2]{HNO94}).
  We recall that non-trivial Euler elements on simple quasihermitian Lie algebras exist if and only if they are hermitian simple and of tube type.
  For each \(1 \leq k \leq n\), there exists a (possibly trivial) \((H_1)\)-homomorphism
  \[\varphi_k : (\fsl(2,\R), \tfrac{1}{2}) \rightarrow (\g_k, H_{0,k}), \quad \text{with} \quad \ad(\varphi_k(\tfrac{1}{2}H)) = \ad(h_k).\]
  Since the inclusion mapping \(\iota_k : (\g_k, H_{0,k}) \rightarrow (\g, H_0)\) is an \((H_1)\)-homomorphism for each \(1 \leq k \leq n\), we obtain an \((H_1)\)-homomorphism \(\kappa := \sum_{k=1}^n \iota_k \circ \varphi_k\) with the desired properties.
\end{proof}

\begin{thm}
  \label{thm:h1-emb-euler-elements}
  Suppose that \(\g_\kappa\) is hermitian simple and of tube type and let \linebreak \(h_\kappa \in \g_\kappa^{-\theta} = [H_0,\g_\kappa]\) be a non-trivial Euler element of \(\g_\kappa\). Then \(h := \frac{1}{2}H_\kappa + h_\kappa \in \fb_\kappa\) is an Euler element of \(\g\) and \(\fb_\kappa\) and we have \(\g_1(h) = (\fb_\kappa)_1(h)\) and \(h \in \g^{-\theta}\).
\end{thm}
\begin{proof}
  Recall from Proposition \ref{prop:h1-multone-parabolic-subalg} that \(H_0 - \frac{1}{2}U_\kappa\) is an \(H\)-element of \(\g_\kappa\).
  Hence, by Lemma \ref{lem:quasiherm-euler-h1-sl2}, there exists an \((H_2)\)-homomorphism
  \[\kappa_0 : (\fsl(2,\R), \tfrac{1}{2}U) \rightarrow (\g_\kappa, H_0 - \tfrac{1}{2}U_\kappa) \quad \text{with} \quad \kappa_0(\tfrac{1}{2}H) = h_\kappa.\]
  In particular, \((2h_\kappa, \kappa_0(X), \kappa_0(Y))\) is an \(\fsl_2\)-triple.
  Using this data, we endow \(\g_{\kappa,1}(h_\kappa)\) with the structure of a simple euclidean Jordan algebra with unit element \(\kappa_0(X)\) via the KKT-construction (cf.\ Appendix \ref{sec:appendix-KKT}).

  Since \(\g_\kappa = \fz_\g(\im \kappa)\) commutes with \(U_\kappa\), the inclusion
  \[\iota: (\g_\kappa, H_0 - \tfrac{1}{2}U_\kappa) \hookrightarrow (\g, H_0)\]
  is an \((H_1)\)-homomorphism. Hence, the commutative sum
  \[\tilde\kappa := (\iota \circ \kappa_0) + \kappa : (\fsl(2,\R), \tfrac{1}{2}U) \rightarrow (\g, H_0), \quad x \mapsto \iota(\kappa_0(x)) + \kappa(x),\]
  is an \((H_1)\)-homomorphism with \(\tilde\kappa(\frac{1}{2}H) = h_\kappa + \frac{1}{2}H_\kappa = h.\) In particular, \(h\) induces a 5-grading on \(\g\) with \(\spec(\ad h) \subset \{0, \pm \frac{1}{2}, \pm 1\}\). We claim that \(h\) is actually an Euler element of \(\g\).

  To this end, we endow \(\g_1(h)\) with the structure of a simple euclidean Jordan algebra with unit element \(\tilde\kappa(X) = X_\kappa + \kappa_0(X)\).
  Then \(\g_{\kappa,1}(h_\kappa) \subset \g_1(h)\) is a Jordan subalgebra of \(\g_1(h)\) because \(\g_\kappa\) commutes with the image of \(\kappa\). 
  Recall now from Appendix \ref{sec:appendix-KKT} that the tube type algebra \(\g\) corresponds uniquely to a simple euclidean Jordan algebra, say \(E\), via the KKT-construction, and that \(\g_1(h) \cong E^{(\ell)}\) for some \(1 \leq \ell \leq r := \rk_\R \g\).
  The same observation applies to the tube type subalgebra \(\g_\kappa\), and a case-by-case analysis using Table \ref{table:h1-multone-class} and Table \ref{table:kkt} shows that \(\g_{\kappa,1}(h_\kappa) \cong E^{(r-1)}\).
  This implies \(\ell \in \{r-1, r\}\).
  Since \(\g_1(h)\) contains \(X_\kappa \in \im \kappa\), it is strictly larger than \(\g_{\kappa,1}(h_\kappa)\), so that \(\ell = r\) and therefore \(\g_1(h) \cong E^{(r)} = E\).
  This shows that \(h\) is an Euler element of \(\g\).

  In particular, \(h \in \fb_\kappa\) is an Euler element of \(\fb_\kappa\).
  It remains to observe that \((\fb_\kappa)_1(h) = \g_1(h)\).
  To this end, we decompose \(\g_1(h)\) into eigenspaces of \(\ad(H_\kappa)\) and see that \(\g_1(h) \cap \g_\lambda(H_\kappa)\) is non-trivial if and only if \(\lambda \in \{0,1,2\}\) because both \(h_\kappa\) and \(\frac{1}{2}H_\kappa\) induce 5-gradings on \(\g\) and commute. Hence we have
  \begin{equation}
    \label{eq:h1-emb-euler-eigenspace}
    \g_1(h) = \g_1(h_\kappa + \tfrac{1}{2}H_\kappa) = \R X_\kappa \oplus \g_1(h_\kappa) \oplus (\g_{\frac{1}{2}}(h_\kappa) \cap \g_{\frac{1}{2}}(\tfrac{1}{2}H_\kappa)) \subset \fb_\kappa. \qedhere
  \end{equation}
\end{proof}

\begin{remark}(The Jordan algebra structure on \((\fb_\kappa)_1(h)\))
  \label{rem:h1-emb-euler-jordan}

  (1) In the notation of the proof of Theorem \ref{thm:h1-emb-euler-elements}, the Jordan algebra structure on \(E := \g_1(h)\) is given by
  \[a \cdot b = \frac{1}{2}[[a, Y_\kappa + \kappa_0(Y)], b] \quad \text{for } a,b \in E.\]
  (cf.\ Appendix \ref{sec:appendix-jordan-algebras}). Since \(\kappa_0(Y) \in \g_\kappa\), which commutes with the image of \(\kappa\), it is easy to see that \(X_\kappa\) is an idempotent of \(E\). We have \(\g_1(h_\kappa) \subset \g_0(H_\kappa)\) because otherwise \(h\) would not be an Euler element of \(\g\). Hence, \(\g_1(h_\kappa) \subset E_0(X_\kappa)\). Furthermore, if \(x \in \g_1(H_\kappa)\), then
  \[\frac{1}{2}[[X_\kappa, Y_\kappa + \kappa_0(Y)], x] = \frac{1}{2}[H_\kappa, x] = \frac{1}{2}x,\]
  which shows \(\g_1(H_\kappa) \cap \g_{\frac{1}{2}}(h_\kappa) \subset E_{\frac{1}{2}}(h_\kappa)\). This shows that \(E_1(X_\kappa) = \R X_\kappa\), i.e.\ \(X_\kappa\) is a primitive idempotent. From \eqref{eq:h1-emb-euler-eigenspace}, we see that its Peirce decomposition is given by
  \[E_1(X_\kappa) = \R X_\kappa, \quad E_0(X_\kappa) = \g_1(h_\kappa), \quad E_{\frac{1}{2}}(X_\kappa) = \g_{\frac{1}{2}}(h_\kappa) \cap \g_1(H_\kappa).\]

  (2) Suppose that \(\g_\kappa\) is hermitian simple. Then \cite[Lem.\ 3.4]{Oeh20b} shows that \(\tau_{V_\kappa} := 2\ad_{V_\kappa}(h_\kappa)\) is an antisymplectic involution on \(V_\kappa = \g_1(H_\kappa)\). Hence, \(E_{\frac{1}{2}}(X_\kappa)\) is given by the 1-eigenspace \(V_{\kappa, 1}\) of \(\tau_{V_\kappa}\) in this case.
\end{remark}

\begin{example}
  \label{ex:jacobi-alg-max-parabolic}
  Consider the simple hermitian tube type Lie algebra \(\g = \sp(2n + 2, \R)\), where \(n \in \N\), and let \(\kappa : (\fsl(2,\R), \frac{1}{2}U) \rightarrow (\g, H_0)\) be an \((H_1)\)-homomorphism of multiplicity \(1\).
It is straightforward to derive from Table \ref{table:h1-multone-class} and Proposition \ref{prop:h1-multone-parabolic-subalg} that
\[\g_\kappa \cong \sp(2n,\R), \quad \fj_\kappa = V_\kappa \times \fz_\kappa \times \g_\kappa \cong \hsp(\R^{2n}), \quad \text{and} \quad \fb_\kappa \cong \hcsp(\R^{2n}).\]
Choose an Euler element \(h_\kappa \in \g_\kappa^{-\theta}\). As explained in Remark \ref{rem:h1-emb-euler-jordan}, the endomorphism \(\tau_{V_\kappa} := 2\ad_{V_\kappa}(h_\kappa)\) is an antisymplectic involution on \(V_\kappa\) (cf.\ \cite[Lem.\ 3.4]{Oeh20b}). As a result, we obtain an isotropic decomposition
\[V_\kappa = V_{\kappa,1} \oplus V_{\kappa,-1} \quad \text{with} \quad V_{\kappa, \pm 1} := \ker(\tau_{V_\kappa} \mp \id_{V_\kappa}).\]
The linear endomorphism \(I_0 := 2\ad_{V_\kappa}(H_0 - \frac{1}{2}U_\kappa)\) interchanges \(V_{\kappa, 1}\) and \(V_{\kappa, -1}\) and is diagonalizable over \((V_{\kappa})_\C \cong \C^{2n}\) with eigenvalues \(\pm i\).
\end{example}

\section{Positive energy representations of extended admissible Lie groups}
\label{sec:pos-energy-rep-ext-adm}

Throughout this section, we fix the following data: Let \((V,\Omega)\) be a symplectic vector space and let \(\hcsp(V,\Omega) = \hsp(V,\Omega) \rtimes \R \id_V\) be the corresponding conformal Jacobi algebra.
The element \(\id_V\) denotes the outer derivation of the Jacobi algebra \(\hsp(V,\Omega)\) which acts by
\[\id_V (v,z,x) := (v,2z,0) \quad \text{for } (v,z,x) \in V \times \fz \times \sp(V,\Omega) = \hsp(V,\Omega).\]
Here, \(\fz \cong \R\) denotes the center of \(\hsp(V,\Omega)\).
In the notation of Example \ref{ex:jacobi-alg-max-parabolic}, this derivation equals \(\ad(H_\kappa)\) for a suitable \(\fsl_2\)-embedding \(\kappa\).
By the same example, we may assume that \(\hcsp(V,\Omega)\) is embedded in a symplectic Lie algebra \(\sp(\tilde V, \tilde \Omega)\) as a maximal parabolic subalgebra with \(\rk_\R \sp(\tilde V, \tilde \Omega) = \rk_\R \sp(V,\Omega) + 1\) with respect to some \(H\)-element \(H_0 \in \sp(\tilde V, \tilde \Omega)\) and an \((H_1)\)-homomorphism \(\kappa: \fsl(2,\R) \rightarrow \sp(\tilde V, \tilde \Omega)\). We omit the \(\kappa\)-subscripts from now on for the sake of readability.

Recall that the choice of the \(H\)-element \(H_0\) determines a Cartan involution \(\theta := e^{\pi\ad H_0}\) with \(\sp(\tilde V, \tilde \Omega)^\theta = \ker(\ad H_0)\), which restricts to a Cartan involution on the subalgebra \(\sp(V,\Omega)\). 

We identify \(\heis(V,\Omega)\) with the subalgebra \(\heis(V,\Omega) \times \{0\}\) of \(\hsp(V,\Omega)\) and, similarly, \(\hsp(V,\Omega)\) as a subalgebra of \(\hcsp(V,\Omega)\).

On the group level, we have the connected Lie group \(\Heis(V,\Omega) = V \times \fz\) with Lie algebra \(\heis(V,\Omega)\). 
Moreover, we set
\[\HSp(V,\Omega) := \Heis(V,\Omega) \rtimes \Sp(V,\Omega) \quad \text{and} \quad \HCSp(V,\Omega) := \HSp(V,\Omega) \rtimes_\alpha \R_+^\times,\]
where \(\alpha_r(v,z,x) := (rv, r^2 z, x)\) for \(r \in \R_+^\times\).
In order to construct an Euler element in \(\hcsp(V,\Omega)\), we fix a non-trivial Euler element \(h_s \in \sp(V,\Omega)^{-\theta}\). Then \(\tau_V := 2\ad_V(h_s)\) is an antisymplectic involution on \((V,\Omega)\).
We define \(\fs_\lambda := \sp(V,\Omega)_\lambda(h_s)\) for \(\lambda \in \{-1,0, 1\}\) and \(V_{\pm 1} = \ker(\tau_V \mp \id_V)\).
By Theorem \ref{thm:h1-emb-euler-elements} and Example \ref{ex:jacobi-alg-max-parabolic}, the element \(h := \frac{1}{2}\id_V + h_s\) is an Euler derivation of \(\hsp(V,\Omega)\) and \(\sp(\tilde V, \tilde \Omega)\) with
\[\hsp(V,\Omega)_0(h) = V_{-1} \oplus \fs_{0}, \quad \hsp(V,\Omega)_{-1}(h) = \fs_{-1}, \quad \text{and} \quad \hsp(V,\Omega)_1(h) = \fz \oplus V_1 \oplus \fs_1.\]
As described in Example \ref{ex:jacobi-alg-max-parabolic}, we obtain a complex structure \(I \in \sp(V,\Omega)^\theta\) that interchanges \(V_1\) and \(V_{-1}\) and satisfies \(\Omega(v,Iv) > 0\) for all \(v \in V \setminus \{0\}\). On \(V_\C\), we denote the \(\pm i\)-eigenspace of \(I\) by \(V^{\pm}\).
Then the map
\[\Psi : V \rightarrow V^+, \quad v \mapsto v_+ := \tfrac{1}{2}(v - iIv)\]
is a linear isomorphism of real vector spaces with inverse given by \(\Psi^{-1}(v) := v + \oline{v}\), where \(\oline{v}\) denotes the complex conjugate of \(v \in V_\C\).
Note that \(V^+ = \Psi(V_1) \oplus i\Psi(V_1)\), so that we can regard \(V_1\) as a real form of \(V^+\).
We extend \(\Omega\) to a complex bilinear form on \(V_\C\). Then
\begin{equation*}
  \la \cdot, \cdot \ra : V^+ \times V^+ \rightarrow \C, \quad \la v,w \ra := i \Omega(\oline{v}, w),
\end{equation*}
defines a scalar product on \(V^+\) (cf.\ \cite[p.\ 85]{HNO96b}). More precisely, we have
\[\la v_+, v_+ \ra = \tfrac{1}{2}\Omega(v, Iv) \quad \text{and} \quad \la p_+, q_+ \ra = \tfrac{i}{2}\Omega(p,q) \quad \text{for } v \in V, p \in V_1, q \in V_{-1}.\]
Moreover, we endow \(V_1\) with a euclidean Lebesgue measure.

Finally, we endow \(\Heis(V,\Omega) = V_1 \times V_{-1} \times \R\) with the multiplication
\[(p,q,z)(p',q',z') := (p + p', q + q', z + z' + \tfrac{1}{2}(\Omega(p,q') + \Omega(q,p'))).\]
In some cases, it will be more convenient to work with the polarized Heisenberg group \(\Heis_p(V,\Omega) = V_1 \times V_{-1} \times \R\) with the multiplication
\[(p,q,z)(p',q',z') := (p + p', q + q', z + z' + \Omega(p,q')).\]
The map
\begin{equation}
  \label{eq:heis-pol-iso}
  \Heis(V,\Omega) \rightarrow \Heis_p(V,\Omega), \quad (p,q,z) \mapsto (p,q, z + \tfrac{1}{2}\Omega(p,q)),
\end{equation}
is a group isomorphism from \(\Heis(V,\Omega)\) to \(\Heis_p(V,\Omega)\).

\subsection{Unitary representations of the conformal Jacobi group}
\label{sec:urep-conf-jacobi}

Throughout this section, let
\[N := \Heis(V,\Omega) \rtimes_\alpha \R_+^\times, \quad T := Z(\Heis(V,\Omega)) \cong \R, \quad \text{and} \quad D := \R_+^\times.\]
Notice that \(T \rtimes_\alpha D = TD\) is isomorphic to \(\Aff(\R)_e\), i.e.\ the identity component of the \(ax+b\)-group.

\begin{definition}
  (1) A strongly continuous unitary representation \((\pi, \cH)\) of a connected Lie group \(G\) with \(\L(G) = \hsp(V,\Omega)\) is called a \emph{positive energy representation} if \(\hsp(V,\Omega)_+ \subset C_\pi\).

  (2) Let \((\pi, \cH)\) be a strongly continuous unitary representation of \(N\).
  We say that \((\pi, \cH)\) is a \emph{positive} (resp.\ \emph{negative}) \emph{energy representation} of \(N\) if the spectrum of the infinitesimal generator of \(\pi\lvert_T\) is positive (resp.\ negative).
  If the spectrum of the infinitesimal generator is strictly positive (resp.\ strictly negative), then we call \((\pi, \cH)\) a \emph{strictly positive} (resp.\ \emph{strictly negative}) energy representation.
\end{definition}

We note that, if \((\pi, \cH)\) is a positive energy representation of \(N\), then the restriction of \(\pi\) to the subgroup \(TD \cong \Aff(\R)_e = \R \rtimes \R_+^\times\) is a positive energy representation in the sense that its restriction to the translation group \(\R \times \{0\}\) is a unitary one-parameter group with a positive infinitesimal generator.

\begin{lem}
  \label{lem:urep-conf-jacobi-energy-decomp}
  Let \(G = N \rtimes L\) be a connected Lie group such that \(L\) commutes with \(T\). Let \((\pi, \cH)\) be a strongly continuous unitary representation of \(G\). Then \((\pi, \cH)\) decomposes into a direct sum \(\cH = \cH_0 \oplus \cH_- \oplus \cH_+\) of subrepresentations, where \((\pi\lvert_T, \cH_0)\) is the trivial representation and \((\pi\lvert_N, \cH_\pm)\) is of strictly positive (resp.\ strictly negative) energy.
\end{lem}
\begin{proof}
  Recall that \(TD \cong \Aff(\R)_e\). We apply \cite[Thm.\ 2.8]{Lo08} to decompose \(\cH\) into \(TD\)-invariant subspaces \(\cH_0 \oplus \cH_+ \oplus \cH_-\) such that \((\pi\lvert_T, \cH_0)\) is the trivial representation and \((\pi\lvert_{TD}, \cH_\pm)\) is of strictly positive (resp.\ strictly negative) energy. The subspaces \(\cH_0, \cH_+,\) and \(\cH_-\) are the spectral subspaces of the infinitesimal generator of \(\pi\lvert_T\) corresponding to the subsets \(\{0\}, (0,\infty),\) and \((-\infty, 0)\). Since \(V \subset \Heis(V,\Omega)\) and \(L\) commute with \(T\), these subspaces are also invariant under \(V\) and \(L\), so that \(\cH_0\) and \(\cH_\pm\) are also \(G\)-invariant.
\end{proof}

\begin{thm}
  \label{thm:heis-aff-discrete-decomp}
  Let \((\pi, \cH)\) be a strongly continuous unitary representation of \(N\) of strictly positive (negative) energy. Then \((\pi, \cH)\) is equivalent to a direct sum of irreducible representations of \(N\) of strictly positive (negative) energy.
\end{thm}
\begin{proof}
  Since \(N\) is separable and \((\pi, \cH)\) can be decomposed into a direct sum of cyclic subrepresentations, we may assume that \(\cH\) is separable. It is a well-known consequence of Zorn's Lemma that \((\pi, \cH)\) is equivalent to a direct sum \((\pi_d, \cH_d) \oplus (\pi_c, \cH_c)\), where \((\pi_d, \cH_d)\) is a direct sum of irreducible subrepresentations and \((\pi_c, \cH_c)\) does not contain any irreducible subrepresentation. Hence, it suffices to show that every strongly continuous unitary representation of \(N\) of strictly positive (resp.\ negative) energy contains an irreducible subrepresentation.
  
  We assume that \((\pi, \cH)\) is a representation of strictly positive energy.
  The negative energy case is treated analogously.
  Then the restriction of \(\pi\) to the affine subgroup \(TD\) is of strictly positive energy, hence is equivalent to a multiple of the unique irreducible representation of strictly positive energy (cf.\ \cite{GN47}).
  We may thus assume that \(\cH = L^2(\R_+^\times, \frac{d\lambda}{\lambda}; \cK) \cong L^2(\R_+^\times, \frac{d\lambda}{\lambda}) \hotimes \cK\), where \(\cK\) is a Hilbert space and the restriction \(\pi_{TD} := \pi\lvert_{TD}\) to \(TD = \R \rtimes_\alpha \R_+^\times\) is given by 
  \begin{equation}
    \label{eq:aff-pos-en-action}
    (\pi_{TD}(z,r)f)(\lambda) = e^{i\lambda^2z}f(r\lambda) \quad \text{for } (z,r) \in \R \times \R_+^\times, \lambda \in \R_+^\times, f \in \cH.
  \end{equation}
  In order to identify the representation of the subgroup \(V\), we observe that \(\pi(V)\) commutes with \(\pi(T)\), hence also with the von Neumann algebra \(\pi(T)''\) generated by \(\pi(T)\), which coincides with the algebra of diagonal operators \(M_f\), \(f \in L^\infty(\R_+^\times, \frac{d\lambda}{\lambda})\), which act on \(\cH\) by multiplication with bounded measurable functions \(f\).
  This implies that, for every \(v \in V\), the operator \(\pi(v)\) is \emph{decomposable} in the sense that there exists a bounded operator-valued measurable function
  \[F_v: \R_+^\times \rightarrow B(\cK) \quad \text{such that} \quad (\pi(v)f)(\lambda) = F_v(\lambda)f(\lambda) \text{ for } f \in \cH, \lambda \in \R_+^\times\]
  (cf.\ \cite[Thm.\ 7.10]{Ta79}). If we identify \(\cH\) with the direct integral \(\int_{\R_+^\times}^\oplus \cK \frac{d\lambda}{\lambda}\), then this argument shows that \(\pi(g)\) is decomposable for every \(g \in \Heis(V,\Omega)\). After changing \(F_v\) on a \(d\lambda\)-measurable subset of measure zero if necessary, we may assume that \(\pi\lvert_{\Heis(V,\Omega)}\) is a direct integral of unitary representations, i.e.\ we have
  \[F_v(\lambda)F_{v'}(\lambda) = e^{\frac{i}{2}\lambda^2\Omega(v,v')}F_{v + v'}(\lambda) \quad \text{for all } \lambda \in \R_+^\times, v,v' \in V\]
  (cf.\ \cite{Mau51}). For \(\lambda = 1\), we thus obtain a representation \((\rho, \cK)\) of \(\Heis(V,\Omega)\) with \(\rho(v) := F_v(1)\) for \(v \in V\) and \(\rho(z) := e^{iz}\id_\cK\) for \(z \in Z(\Heis(V,\Omega))\). By the Stone--von Neumann Theorem, the representation \((\rho, \cK)\) is a direct sum of irreducible subrepresentations. Let \(\cK_0 \subset \cK\) be a non-zero closed subspace on which \(\rho\) acts irreducibly. We claim that
  \[\cH_0 := \{f \in \cH : f(\lambda) \in \cK_0 \text{ almost everywhere}\} \cong L^2(\R_+^\times, \tfrac{d\lambda}{\lambda}) \hotimes \cK_0\]
  is an irreducible subrepresentation of \((\pi, \cH)\). It follows from \eqref{eq:aff-pos-en-action} that \(\cH_0\) is \(TD\)-invariant. Let \(f \in \cH, v \in V, r \in D,\) and \(\lambda\in \R_+^\times\). Then
  \[F_{rv}(\lambda)f(\lambda) = (\pi(rv)f)(\lambda) = (\pi(r)\pi(v)\pi(\tfrac{1}{r})f)(\lambda) = F_v(r\lambda)(\pi(\tfrac{1}{r})f)(r\lambda) = F_v(r\lambda)f(\lambda)\]
  shows that
  \[(\pi(v)f)(\lambda) = F_v(\lambda)f(\lambda) = F_{\lambda v}(1)f(\lambda) = \rho(\lambda v)f(\lambda).\]
  In particular, \(\cH_0\) is also \(V\)-invariant and therefore \(N\)-invariant.
  In order to show that \((\pi, \cH_0)\) is irreducible, we consider an orthogonal projection \(P_{\cH_0}\) commuting with \(\pi\). Since \(P_{\cH_0}\) commutes with \(\pi(T)\), we may assume that \(P_{\cH_0}\) is represented by a measurable function \(P_{\cK} : \lambda \mapsto P_{\cK}(\lambda)\) on \(\R_+^\times\) with values in the set of orthogonal projections on \(\cK\). Then \eqref{eq:aff-pos-en-action} implies that \(P_{\cK}\) is constant almost everywhere, and the fact that \(P_{\cK}(\lambda)\) is an orthogonal projection on \(\cK_0\) commuting with \(\rho(\Heis(V,\Omega))\) implies that \(P_{\cK}(\lambda)\) is trivial for every \(\lambda \in \R_+^\times\). Hence, \(P_{\cH_0}\) is also trivial, which finishes the proof. 
\end{proof}

\begin{rem}
  \label{rem:mackey-method-induced-rep}
  In order to classify the irreducible unitary representations of the conformal Jacobi group and, more generally, the extended Lie groups \(G = G_C \rtimes \R_+^\times\), where \(G_C = \Heis(V,\Omega) \rtimes L\) is admissible, we use the well-known Mackey method (cf.\ \cite{Mac49},\cite{Mac52},\cite[Thm.\ 6.12]{Va85}).
  If \(A\) and \(G\) are Lie groups, where \(A\) is abelian and \(\alpha: G \rightarrow \Aut(A)\) is a group homomorphism, then the equivalence classes of irreducible unitary representations of the semidirect product \(A \rtimes_\alpha G\) can be described as follows: Suppose that \(\Sigma \subset \hat A = \Hom(A,\T)\) is a measurable subset with the property that it contains exactly one element of each \(G\)-orbit. We denote by \(\hat\alpha: G \rightarrow \Aut(\hat A)\) the induced action of \(G\) on \(\hat A\).
  Then every \(\chi \in \Sigma\) and every irreducible unitary representation \((\rho, \cK)\) of the stabilizer \(G_\chi\) of \(\chi\) induce an irreducible unitary representation \((\pi_{\chi, \rho}, \cH_{\chi, \rho})\) of \(A \rtimes_\alpha G\).
  Moreover, two representations \(\pi_{\chi, \rho}\) and \(\pi_{\chi', \rho'}\) are equivalent if and only if \(\chi = \chi'\) and \(\rho\) and \(\rho'\) are equivalent (cf.\ \cite[Thm.\ 6.24]{Va85}).
  In order to realize an induced representation \(\pi_{\chi, \rho}\), we fix a measure \(\mu\) from the unique equivalence class of \(G\)-invariant measures on \(G/G_\chi\), which we identify with \(G.\chi \subset \hat A\).
  Furthermore, we fix a section \(s : G/G_\chi \rightarrow G\) of the quotient map \(q_\chi : G \rightarrow G/G_\chi\) such that \(s(q(\1)) = \1\). For every \(g \in G\), we obtain a function
  \begin{equation}
    \label{eq:mackey-method-sec-rep}
    \varphi_\rho(g,x) := \rho(s(x)^{-1}gs(g^{-1}.x)), \quad g \in G, x \in G/G_\chi.
  \end{equation}
  We then define \(\cH_{\chi, \rho} := L^2(G/G_\chi, \mu; \cK)\) and set
  \begin{equation}
    \label{eq:mackey-induced-rep}
    (\pi_{\chi, \rho}(a,g)f)(\tilde\chi) := \tilde\chi(a) \varphi_\rho(g,\tilde\chi) \sqrt{\delta_g}f(g^{-1}.\tilde\chi), \quad f \in L^2(G/G_\chi, \mu; \cK), (a,g) \in A \rtimes_\alpha G.
  \end{equation}
  The term \(\delta_g\) is the Radon--Nikodym derivative of the measure \(g_*\mu := \mu \circ \hat\alpha_g^{-1}\) with respect to \(\mu\).
\end{rem}

\begin{lem}
  \label{lem:conf-heis-irrep}
  There exist exactly two equivalence classes of strongly continuous irreducible unitary representations of \(N = \Heis(V,\Omega) \rtimes_\alpha \R_+^\times\) which are non-trivial on \(T = Z(\Heis(V,\Omega))\) and one equivalence class of strongly continuous irreducible strictly positive energy representations \((\pi, \cH)\). Its restriction to the subgroup \(\Heis(V,\Omega)\) is equivalent to the direct integral of representations
  \[\int_{\R_+^\times}^\oplus (\pi_\lambda, \cH_\lambda)\,\frac{d\lambda}{\lambda}, \]
  where \((\pi_\lambda, \cH_\lambda)\) is a strongly continuous irreducible unitary representation of \(\Heis(V,\Omega)\) with \(\pi_\lambda (z) = e^{i\lambda^2 z}\id_{\cH_\lambda}\) for \(z \in T\).
\end{lem}
\begin{proof}
  We identify \(N\) with \(\Heis_p(V,\Omega) \rtimes_\alpha \R_+^\times\) using \eqref{eq:heis-pol-iso}. We can then write \(N\) as \(N = A \rtimes H\), where \(A := V_{-1} \times \R\) and \(H := V_{1} \rtimes \R_+^\times\), and \(H\) acts on \(A\) by \((p, r).(q,z) := (rq, r^2 z + r\Omega(p,q))\). All characters on \(A\) are of the form \(\chi_{(p,z)}(q,z') = e^{i zz'} e^{i \Omega(q,p)}\) with \((p,z) \in V_1 \times \R\), so that we can identify \(\hat A\) with \(V_1 \times \R\).
  For \(p' \in V_1\) and \(r > 0\), we have
  \[((p',r).\chi_{(p,z)})(q,z') = \chi_{(p,z)}((p',r)^{-1}.(q,z')) = \chi_{(p,z)}(\tfrac{1}{r}q,\tfrac{1}{r^2}z' - \tfrac{1}{r^2}\Omega(p',q)) = \chi_{(\frac{1}{r}(p + \frac{z}{r}p'), \tfrac{1}{r^2}z)}(q,z').\]
  Thus, \(\hat A\) decomposes into the \(H\)-orbits
  \[\cO_\pm = V_{1} \times \R_\pm^\times \quad \text{and} \quad \cO_{\R_+^\times x, 0} = \R_+^\times x \times \{0\}, \quad x \in V_{1}.\]
  By Remark \ref{rem:mackey-method-induced-rep}, every irreducible unitary representations of \(N\) is equivalent to a representation \((\pi_{\chi, \rho}, \cH_{\chi, \rho})\) as in \eqref{eq:mackey-induced-rep}, where \(\chi \in \hat{A}\) is a character and \((\rho,\cK)\) is an irreducible unitary representation of \(H_\chi\).
  If \(\chi \in \cO_{\R_+^\times,0}\) for some \(x \in V_1\), then \eqref{eq:mackey-induced-rep} shows that \(\pi_{\chi,\rho}\) is trivial on \(T \cong \R\) because every character in \(\cO_{\R_+^\times x,0} = H.\chi\) vanishes on \(T\).

  On the other hand, if \(\chi \in \cO_\pm\), then every induced representation \(\pi_{\chi,\rho}\) is non-trivial on \(T\).
  In this case, the stabilizer group \(H_\chi\) is trivial, so that \(\cK \cong \C\) and \(\rho\) is trivial.
  This implies that there are, up to equivalence, exactly two strongly continuous irreducible unitary representations of \(N\) which are non-trivial on the center. The induced representations \(\pi_\pm\) of \(N\) corresponding to \(\cO_\pm\) can be realized on \(\cH := L^2(\R_+^\times \times V_1, \frac{d\lambda}{\lambda} \otimes dx)\) by
  \begin{equation}
    \label{eq:conf-heis-irrep-expl-pol-1}
    (\pi_\pm(p,q,z)f)(\lambda,x) := e^{\pm i\lambda^2 z} e^{\pm i\lambda \Omega(q, x)} f(\lambda, x - \lambda p), \quad (p,q,z) \in V_1 \times V_{-1} \times T, \text{ and}
  \end{equation}
  \begin{equation}
    \label{eq:conf-heis-irrep-expl-pol-2}
    (\pi_\pm(r)f)(\lambda, x) = f(r\lambda, x), \quad r \in \R_+^\times \cong \exp(\R \id_V)
  \end{equation}
(cf.\ \cite[Ex.\ 3.7]{Ne20}). In particular, we see that only \(\pi_+\) satisfies the positive energy condition. The last part of the claim follows from the construction of \(\pi_+\).
\end{proof}

\begin{definition}
  \label{def:urep-nu}
  (1) Let \((\pi, \cH)\) be the unitary positive energy representation of \(\Heis(V,\Omega) \rtimes_\alpha \R_+^\times\) on \(\cH = L^2(\R_+^\times \times V_1, \frac{d\lambda}{\lambda} \otimes dx) \cong L^2(\R_+^\times, \frac{d\lambda}{\lambda}) \hotimes L^2(V_1, dx)\) that we obtain from Lemma \ref{lem:conf-heis-irrep}. After composing with the isomorphism \eqref{eq:heis-pol-iso}, we see that it is given by
  \begin{equation}
    \label{eq:conf-heis-irrep-expl-1}
    (\pi(p,q,z)f)(\lambda,x) := e^{i\lambda^2 (z + \frac{1}{2}\Omega(p,q))} e^{i\lambda \Omega(q, x)} f(\lambda, x - \lambda p), \quad (p,q,z) \in V_1 \times V_{-1} \times T, \text{ and}
  \end{equation}
  \begin{equation}
    \label{eq:conf-heis-irrep-expl-2}
    (\pi(r)f)(\lambda, x) = f(r\lambda, x), \quad r \in \R_+^\times \cong \exp(\R \id_V).
  \end{equation}
  The metaplectic representation of the universal covering group \(\tilde\Sp(V,\Omega)\) of \(\Sp(V,\Omega)\) on \(L^2(V_1, dx)\) extends to a representation on \(\cH\) by acting trivially on \(L^2(\R_+^\times, \frac{d\lambda}{\lambda})\), and it extends \((\pi, \cH)\) to an irreducible unitary representations of \(\tilde\HCSp(V,\Omega)\).
  From now on, we denote the extension of \((\pi, \cH)\) to \(\tilde\HCSp(V,\Omega)\) by \((\nu, \cH_\nu)\).

  (2) Let \(\tau_G\) be the involutive automorphism of \(G := \tilde\HCSp(V,\Omega)\) with \(\L(\tau_G) = e^{i\pi \ad h}\), where \(h \in \hcsp(V,\Omega)\) is the Euler element that we defined in the introduction of Section \ref{sec:pos-energy-rep-ext-adm}. Then the representation \((\nu, \cH_\nu)\) of \(G\) extends to an antiunitary representation of the semidirect product \(G_\tau\) as follows: If we realize \((\nu, \cH_\nu)\) on \(L^2(\R_+^\times \times V_1, \frac{d\lambda}{\lambda} \otimes dx)\) as in \eqref{eq:conf-heis-irrep-expl-1} and \eqref{eq:conf-heis-irrep-expl-2}, then \((J_\nu f)(\lambda, x) := \oline{f(\lambda, -x)}\) for \(f \in \cH_\nu\) satisfies \(\nu(\tau_G(g)) = J_\nu\nu(g)J_\nu\) for all \(g \in G\) (cf.\ \cite[Ex.\ 3.7]{Ne20}).
  We denote the antiunitary extension of \((\nu, \cH_\nu)\) to \(G_\tau\) by \((\nu_\tau, \cH_\nu)\).
\end{definition}

\begin{rem}
  \label{rem:rep-nu-pos-cone-bounded-cone}
  The positive cone \(C_\nu \subset \hsp(V,\Omega)\) of the representation \((\nu,\cH_\nu)\) is the cone \(\hsp(V,\Omega)_+\) (cf. \eqref{eq:hsp-plus-cone}) by (cf.\ \cite[Ex.\ 3.7]{Ne20}).

  The cone \(W_\nu := \{x \in \hsp(V,\Omega) : \inf(\spec(-i\partial\nu(x))) > -\infty\}\) of semibounded elements of \((\nu, \cH_\nu)\) coincides with \(C_\nu\).
  This can be seen by considering, for a compactly embedded Cartan subalgebra \(\ft_\fs \subset \sp(V,\Omega)\) and the compactly embedded Cartan subalgebra \(\ft := \{0\}\times \fz \times \ft_\fs\), the intersections of \(C_\nu\) and \(W_\nu\) with \(\ft\):
  For \((0,z,x_s) \in W_\nu \cap \ft\), the element \(x_s\) is contained in the up to sign unique pointed generating invariant closed convex cone of \(\sp(V,\Omega)\), so that \(x_s \in C_\nu\).
  From \eqref{eq:conf-heis-irrep-expl-1}, we see that \(-i\partial\nu(z)\) is bounded from below if and only if \(z \geq 0\).
  This shows that \(C_\nu \cap \ft = W_\nu \cap \ft\) and therefore \(C_\nu = W_\nu\) by \cite[Prop.\ VII.3.31]{Ne00}.
\end{rem}

\begin{cor}
  \label{cor:heis-aff-pos-en-mult}
  Every strictly positive (negative) energy representation \((\pi, \cH)\) of \(N\) is a multiple of the unique irreducible strictly positive (negative) energy representation.
\end{cor}
\begin{proof}
  This is a consequence of Theorem \ref{thm:heis-aff-discrete-decomp} and Lemma \ref{lem:conf-heis-irrep}.
\end{proof}

\begin{rem}
  \label{rem:urep-conf-jacobi-aff-group}
  The restriction of \(\nu\) to the affine subgroup \(TD \cong \Aff(\R)_e\) is equivalent to a strictly positive energy representation of \(\Aff(\R)_e\) on \(L^2(\R_+^\times, \frac{dx}{x})\), which can be seen via the identification \(\cH \cong L^2(\R_+^\times, \frac{d\lambda}{\lambda}) \hotimes L^2(V_1, dx)\). The restriction of \(\nu \circ \tau_G\) to \(TD\) yields an irreducible strictly negative energy representation.
\end{rem}

\begin{lem}
  \label{lem:rep-adm-irrep-ss-extension}
  Let \(G = N \rtimes S\) be a semidirect product of \(N\) with a connected perfect Lie group \(S\). Let \((\pi_1, \cH)\) and \((\pi_2, \cH)\) be strongly continuous unitary representations of \(G\) such that \(\pi_1\lvert_N = \pi_2\lvert_N\) and \(\pi_1\lvert_N\) is irreducible. Then \(\pi_1 = \pi_2\).
\end{lem}
\begin{proof}
  We have \(\pi_1(N)' = \C \bbone\) by Schur's Lemma. Moreover, the map
  \[\chi : S \rightarrow \pi_1(N)' = \C \bbone, \quad s \mapsto \pi_1(s)\pi_2(s)^{-1}, \]
  is a strongly continuous group representation because
  \begin{align*}
    \chi(ss') &= \pi_1(ss')\pi_2(ss')^{-1} = \pi_1(s)\pi_1(s')\pi_2(s')^{-1}\pi_2(s)^{-1} = \pi_1(s)\chi(s')\pi_2(s)^{-1} = \pi_1(s)\pi_2(s)^{-1}\chi(s') \\
              &= \chi(s)\chi(s')
  \end{align*}
  for all \(s,s' \in S\). Since \(\L(S)\) is perfect and the image of \(\chi\) is abelian, the derived representation \(\dd\chi\) of the Lie algebra \(\L(S)\) is trivial. Hence, \(\chi\) is constant and therefore \(\pi_1(s) = \pi_2(s)\) for all \(s \in S\).
\end{proof}

\begin{cor}
  \label{cor:rep-adm-irrep-pos-ss-ext}
  Let \(G = \Heis(V,\Omega) \rtimes (\tilde\Sp(V,\Omega) \times \R_+^\times)\) and let \((\pi, \cH)\) be a strongly continuous unitary representation of \(G\) such that \((\pi\lvert_N, \cH)\) is irreducible and of strictly positive energy. Then \((\pi, \cH) \cong (\nu, \cH_\nu)\).
\end{cor}
\begin{proof}
  The positive energy condition on \(\pi\) combined with Lemma \ref{lem:conf-heis-irrep} implies that \((\pi\lvert_N, \cH) \cong (\nu\lvert_N, \cH)\). By Lemma \ref{lem:rep-adm-irrep-ss-extension}, the extension of the irreducible representation \(\nu\lvert_N\) to \(G\) is unique, which proves the claim.
\end{proof}

In the proposition below, we show that \((\nu, \cH_\nu)\) can be extended to a unitary highest weight representation of \(\tilde\Sp(\tilde V, \tilde \Omega)\) up to coverings. Before stating the result, we have to introduce the following data:
Recall that we have fixed on \(\sp(\tilde V, \tilde \Omega)\) a Cartan involution \(\theta\) at the beginning of Section \ref{sec:pos-energy-rep-ext-adm} such that the Euler element \(h \in \sp(\tilde V, \tilde \Omega)\) satisfies \(\theta(h) = -h\) (see also Theorem \ref{thm:h1-emb-euler-elements}).
This implies that the involution \(\tau_h := e^{i\pi\ad h} \in \Aut(\sp(\tilde V, \tilde \Omega))\) commutes with \(\theta\).
Let \(\sp(\tilde V, \tilde \Omega) = \fk \oplus \fp\) be the Cartan decomposition with respect to \(\theta\).
We fix a maximal compactly embedded \(\tau_h\)-invariant Cartan subalgebra \(\ft \subset \fk\) and a maximal abelian subspace \(\fa \subset \fp\) with \(h \in \fa\).
By Moore's Theorem (cf.\ \cite[Thm.\ 2]{Mo64}), the restricted root system \(\Sigma \subset \fa^*\) is of type \((C_{r+1})\), i.e.\ we have
\[\Sigma \cong \{\pm(\varepsilon_k \pm \varepsilon_\ell) : 1 \leq k < \ell \leq r+1\} \cup \{\pm 2\varepsilon_k : 1 \leq k \leq r+1\},\]
where \(r + 1 := \rk_\R \sp(\tilde V, \tilde \Omega) = \frac{1}{2}\dim \tilde V\).
We fix a \(\fk\)-adapted positive system \(\Delta^+ \subset \Delta := \Delta(\sp(\tilde V, \tilde \Omega)_\C, \ft_\C)\) (cf.\ Definition \ref{def:u-hrw}) and a system \(\Pi := \{\gamma_1,\ldots,\gamma_{r+1}\} \subset \Delta_p^+\) of \emph{strictly orthogonal roots}, i.e.\ we have \(\gamma_k \pm \gamma_\ell \not\in \Delta\) for \(k,\ell = 1,\ldots,r+1\) and \(k \neq \ell\).
Using a Cayley transform, we can identify \(\fa\) as a subset of \(i\ft\) in such a way that \(\gamma_j\lvert_\fa = 2\varepsilon_j\).
Having fixed this data, we obtain the following result:

\begin{prop}
  \label{prop:rep-extjac-metaplectic-sp-ext}
  Let \(B := \la \exp \fb_\kappa \ra \subset S := \tilde\Sp(\tilde V, \tilde \Omega)\) be the integral subgroup with Lie algebra \(\fb_\kappa \cong \hcsp(V,\Omega)\).
  Let \(\tau_S \in \Aut(S)\) be the involutive automorphism with \(\L(\tau_S) = \tau := e^{i\pi\ad(h)}\) and let \(\tau_{\tilde B}\) be the involutive automorphism of \(\tilde B := \tilde\HCSp(V,\Omega)\) with \(\L(\tau_{\tilde B}) = e^{i\pi\ad(h)}\). Let
  \[q : \tilde B_\tau \rightarrow B_\tau, \quad q\lvert_{\tilde B} := q_B, \quad q(\tau_{\tilde B}) := \tau_S,\]
  where \(q_B : \tilde B \rightarrow B\) is the corresponding covering map of \(B\).

  There exists a strongly continuous antiunitary representation \((\pi_S, \cH_S)\) of \(S_\tau\) such that \newline\((\pi_S \circ q, \cH_S)\) is equivalent to \((\nu_\tau, \cH_\nu)\).
  Its restriction to \(S\) is unitary highest weight representation of highest weight \(\lambda \in i\fz(\fk)^*\) with \(\lambda\lvert_\fa = -\frac{1}{2}\sum_{j=1}^{r+1} \varepsilon_j\).
\end{prop}
\begin{proof}
  The unitary highest weight representation of \(S\) of highest weight \(\lambda\) can be realized as follows:
  Let
  \[\cS := \{(z,v) \in \C \times V^+ : \Im(z) + 2i\Omega(v,\oline{v}) > 0\},\]
  where \(I\) is the complex structure on \(V\) which we defined at the beginning of Section \ref{sec:pos-energy-rep-ext-adm} and \(V^+ \subset V_\C\) is its \(i\)-eigenspace.
  Then \(\cS\) is a Siegel domain of the second kind (cf.\ \cite[Def.\ IV.1]{HNO96b}).
  For \(c := \frac{1}{2}\), we consider the positive definite kernel
  \[\cK^c((z,v),(z',v')) := \left(4\pi\left(\tfrac{z'-\oline{z}}{2i} - 2i\Omega(\oline{v}, v')\right)\right)^{-c}  \quad \text{for } (z,v),(z',v') \in \cS.\]
  We denote the corresponding reproducing kernel Hilbert space in \(\Hol(\cS)\) by \(\cH_{c}\).
  We apply \cite[Prop.\ IV.8]{HNO96b} in order to obtain a strongly continuous unitary representation \((\pi, \cH_c)\) of \(\tilde\HCSp(V,\Omega)\).
  The restriction of \(\pi\) to the subgroup \(\Heis(V,\Omega) \rtimes \R_+^\times\) is irreducible (cf.\ \cite[p.\ 172]{HNO96b}) and non-trivial on the center.
  We may without loss of generality assume that \((\pi, \cH_{c})\) is a positive energy representation (cf.\ Remark \ref{rem:urep-conf-jacobi-aff-group}).
  Hence, \((\pi, \cH_{c})\) is equivalent to \((\nu, \cH_{\nu})\) by Corollary \ref{cor:rep-adm-irrep-pos-ss-ext}.
  Now \cite[Thm.\ VI.3]{HNO96b} and \cite[Prop.\ VI.7]{HNO96b} show that there exists a strongly continuous unitary highest weight representation \((\pi_{S,0}, \cH_{c})\) of \(\tilde\Sp(\tilde V, \tilde \Omega)\) such that \(\pi_{S,0} \circ q_B = \pi\).
  The highest weight \(\lambda \in i\ft^*\) of \((\pi_{S,0}, \cH_{c})\) is of scalar type, i.e.\ it vanishes on \(\ft \cap [\fk,\fk]\), so that we may regard it as an element of \(i\fz(\fk)^*\). It is uniquely determined by \(\lambda(\check \gamma_1) = -c\), which implies \(\lambda(\check \gamma_\ell) = -c\) for \(1 \leq \ell \leq r+1\) (cf.\ \cite[Lem.\ III.3(i)]{HNO96b}) and therefore
  \[\lambda\lvert_\fa = \sum_{j=1}^{r+1} \lambda(\check \gamma_j)\varepsilon_j = -c \sum_{j=1}^{r+1} \varepsilon_j = -\frac{1}{2} \sum_{j=1}^{r+1} \varepsilon_j.\]

  In view of the uniqueness of antiunitary extensions to \(\tilde B_\tau\) (cf.\ \cite[Thm.\ 2.11]{NO17}), it only remains to show that an extension of \((\pi_{S,0}, \cH_c)\) to an antiunitary representation \((\pi_S, \cH_c)\) of \(S_\tau\) exists because we have already shown that \((\pi_{S,0} \circ q_B, \cH_c) \cong (\nu, \cH_\nu)\).
  Since \(h\) induces a 3-grading on \(\sp(\tilde V, \tilde \Omega)\), we have \(\fz(\fk) \subset \sp(\tilde V, \tilde\Omega)^{-\tau_h}\) by \cite[Lem.\ 2.15]{Oeh20a} and \cite[Prop.\ 2.20]{Oeh20a}.
  Moreover, the compactly embedded Cartan subalgebra \(\ft\) decomposes into \(\ft = \fz(\fk) \oplus (\ft \cap [\fk,\fk])\), and both summands are \(\tau_h\)-invariant.
  Since \(\lambda\) vanishes on \(\ft \cap [\fk,\fk]\), we have \(\lambda \circ \tau = -\lambda\).
  The representation \((\pi_{S,0} \circ \tau, \cH_c)\) is a unitary highest weight representation with highest weight \(\lambda \circ \tau\) and the dual representation \((\pi_{S,0}^*, \cH_c^*)\) is a unitary highest weight representation with highest weight \(-\lambda\) (cf.\ \cite[Prop.\ IX.1.13]{Ne00}).
  The previous discussion now shows that \((\pi_{S,0} \circ \tau,\cH_c) \cong (\pi_{S,0}^*,\cH_c^*)\) because the highest weights coincide (cf.\ \cite[Thm.\ X.4.2]{Ne00}).
  By \cite[Lem.\ 2.16]{NO17}, this implies that an antiunitary extension of \((\pi_{S,0}, \cH_c)\) exists.
\end{proof}

\subsection{A factorization theorem}
\label{sec:jacobi-rep-factorization}

Throughout this section, let \(\g_C = \g(\fl, V, \fz, \Omega)\) be an admissible Lie algebra that is defined in terms of Spindler's construction, where \(\fz = \fz(\g_C) = \R\) and \((V,\Omega)\) is a symplectic \(\fl\)-module of convex type. Let \(\g = \g_C \rtimes \R \id_V\) and denote by \(\rho : \fl \rightarrow \sp(V,\Omega)\) the homomorphism that corresponds to the representation of \(\fl\) on \(V\). Let
\begin{equation}
  \label{eq:rep-adm-pos-factors-pullback}
  \gamma : \g \rightarrow \hcsp(V,\Omega), \quad (v, z, x, r \id_V) \mapsto (v, z, \rho(x), r \id_V).
\end{equation}
Let \(L\) be a 1-connected Lie group with Lie algebra \(\fl\) and let \(G := \Heis(V,\Omega) \rtimes (L \times \R_+^\times)\).

Moreover, let \(0 \neq h \in \g\) be an Euler element of \(\g\) with \(\ad h\) of the form \eqref{eq:euler-adm} and let \(\tau_G \in \Aut(G)\) be the involutive automorphism with \(\L(\tau_G) = e^{i\pi \ad h}\).
Let \(G_\tau := G \rtimes \{\1,\tau_G\}\) and \(L_\tau := L \rtimes \{\1, \tau_G\lvert_L\}\).

The following theorem provides a convenient description of all strongly continuous unitary representations of \(G\) which satisfy the positive energy condition on the subgroup \(N = \Heis(V,\Omega) \rtimes_\alpha \R_+^\times\):

\begin{thm}
  \label{thm:rep-adm-pos-factors}{\rm (Factorization Theorem)}
  Let \((\pi, \cH)\) be a strongly continuous unitary representation of \(G\) with discrete kernel such that \((\pi\lvert_N, \cH)\) is a strictly positive energy representation.
  Then there exists a strongly continuous unitary representation \((\pi_L, \cH_L)\) of \(L\) such that \(\cH \cong \cH_L \hotimes \cH_\nu\) and \(\pi \cong \pi_L \boxtimes \nu_\gamma\) (cf.\ {\rm Appendix \ref{sec:app-tensor-uni-rep}}), where \(\nu_\gamma\) is the pullback of the representation \((\nu, \cH_\nu)\) of \(\tilde\HCSp(V,\Omega)\) via the group homomorphism whose derivative is \(\gamma\).
  Moreover, the following holds:
  \begin{enumerate}
    \item \(N\) acts trivially on \(\cH_L\).
    \item We have \(\pi(G)' = \pi_L(L)' \otimes \id_{\cH_\nu}\). In particular, \((\pi, \cH)\) is irreducible if and only if \((\pi_L, \cH_L)\) is irreducible.
    \item The positive cone \(C_\pi \subset \g\) of \((\pi, \cH)\) is contained in the positive cone \(C_{\nu_\gamma}\) of \((\nu_\gamma, \cH_\nu)\).
    \item Suppose that the positive cone \(C_\pi \subset \g\) of \((\pi, \cH)\) generates \(\g_C\) and that \((\pi, \cH)\) is irreducible.
      Let \(\ft_\fl \subset \fl\) be a compactly embedded Cartan subalgebra of \(\fl\) and \(\Delta^+\) be an adapted positive system with respect to the compactly embedded Cartan subalgebra \(\ft := \{0\} \times \fz \times \ft_\fl\) of \(\g_C\) such that \(\Cmin \subset C_\pi \cap \ft \subset \Cmax\) (cf.\ {\rm Definition \ref{def:u-hrw}}).
      Then \((\pi_L, \cH_L)\) is a unitary highest weight representation with respect to \(\Delta^+\).
    \item If \(\fl\) is semisimple, then \((\pi\lvert_N, \cH)\) is irreducible if and only if \((\pi, \cH) \cong (\nu_\gamma, \cH_\nu)\).
  \end{enumerate}
\end{thm}
\begin{proof}
  The first part of the claim follows by an adaptation of the proof of the Metaplectic Factorization Theorem (cf.\ \cite{Sa71} or \cite[Thm.\ X.3.7]{Ne00}): By Corollary \ref{cor:heis-aff-pos-en-mult}, the representation \((\pi\lvert_N, \cH)\) is a multiple of \((\nu, \cH_\nu)\). Hence, we may assume that \(\cH = \cH_L \otimes \cH_\nu\) and that \(\pi(n) = \id_{\cH_L} \otimes \nu(n)\) for \(n \in N\). For \(g \in L\), let \(\pi'(g) := \pi(g)(\id_{\cH_L} \otimes \nu_\gamma(g)^{-1})\). Then we have
  \begin{align*}
    \pi'(g)\pi(n) &= \pi(g)(\id_{\cH_L} \otimes \nu_\gamma(g)^{-1})(\id_{\cH_L} \otimes \nu(n)) = \pi(g)(\id_{\cH_L} \otimes \nu_\gamma(g)^{-1}\nu(n)\nu_\gamma(g))(\id_{\cH_L} \otimes \nu_\gamma(g)^{-1}) \\
                  &= \pi(g)\pi(g^{-1}ng)(\id_{\cH_L} \otimes \nu_\gamma(g)^{-1}) = \pi(n)\pi'(g).
  \end{align*}
  In particular, \(\pi(g)' \in \pi(N)' = B(\cH_L) \otimes \id_{\cH_\nu}\) and therefore we have \(\pi'(g) = \pi_L(g) \otimes \id_{\cH_\nu}\) for some \(\pi_L(g) \in \U(\cH_L)\). From
  \begin{align*}
    \pi_L(g)\pi_L(g') \otimes \id_{\cH_\nu} &= \pi'(g)\pi'(g') = \pi(g)(\id_{\cH_L} \otimes \nu_\gamma(g^{-1}))\pi'(g') = \pi(g)\pi'(g')(\id_{\cH_L} \otimes \nu_\gamma(g^{-1}))\\
                                            &= \pi(g)\pi(g') (\id_{\cH_L} \otimes \nu_\gamma(gg')^{-1})
  \end{align*}
  for all \(g,g' \in L\), we see that \((\pi_L, \cH_L)\) is a continuous unitary representation of \(L\), which we trivially extend to a representation of \(G\).
  This proves (a).

  (b) Since \(\pi(N)' = B(\cH_L) \otimes \id_{\cH_\nu}\), we have \(\pi_L(L)' \otimes \id_{\cH_\nu} = \pi(G)'\) . Hence, \((\pi_L, \cH_L)\) is irreducible if and only if \((\pi, \cH)\) is irreducible by Schur's Lemma.

  (c) Let \(x = (v,z,y) \in C_\pi \subset V \times \fz \times \fl = \g_C\).
  For unit vectors \(v \in \cH_L^\infty\) and \(w \in \cH_\nu^\infty\), we then have
  \[\la i\partial\nu_\gamma(x)w, w \ra = \la i\partial\pi(x)(v \otimes w), v \otimes w \ra - \la i\partial\pi_L(x)v, v\ra \leq -\la i\partial\pi_L(x)v, v \ra.\]
  This shows that \(-i\partial\nu_\gamma(x)\) is bounded from below, so that \(\gamma(x) \in C_{\nu} = \hsp(V,\Omega)_+\) by Remark \ref{rem:rep-nu-pos-cone-bounded-cone} and therefore \(x \in C_{\nu_\gamma}\).

  (d) The representation \((\pi_L, \cH_L)\) is irreducible by (b).
  We choose an element \(x \in (C_\pi \cap \ft)^o \subset \Cmax^o\).
  For unit vectors \(v \in \cH_L^\infty\) and \(w \in \cH_\nu^\infty\), we then have
  \[\la i\partial\pi_L(x)v, v \ra = \la i\partial\pi(x)(v \otimes w), v \otimes w \ra - \la i\partial\nu_\gamma(x)w, w\ra \leq -\la i\partial\nu_\gamma(x)w, w \ra.\]
  Hence, \(-i\partial\pi_L(x)\) is bounded from below. By \cite[Cor.\ X.2.10]{Ne00}, this implies that \((\pi_L, \cH_L)\) is a highest weight representation with respect to \(\Delta^+\).

  (e) Suppose that \(\fl\) is semisimple. Then \(L\) is a perfect Lie group. If \((\pi\lvert_N, \cH)\) is irreducible, then the positive energy assumption on \((\pi, \cH)\) and Lemma \ref{lem:conf-heis-irrep} imply that \((\pi\lvert_N, \cH)\) is equivalent to the irreducible representation \((\nu\lvert_N, \cH_\nu)\). By Lemma \ref{lem:rep-adm-irrep-ss-extension}, the extension of \((\pi\lvert_N, \cH)\) to \(G\) is unique. Therefore, we have \((\pi, \cH) \cong (\nu_\gamma, \cH_\nu)\), which proves (e).
\end{proof}

\begin{prop}
  \label{prop:rep-adm-pos-antiuniext}
  In the context of {\rm Theorem \ref{thm:rep-adm-pos-factors}}, the following are equivalent:
  \begin{enumerate}
    \item The representation \((\pi_L, \cH_L)\) extends to an antiunitary representation of \(L_\tau\).
    \item The representation \((\pi, \cH)\) extends to an antiunitary representation of \(G_\tau\).
  \end{enumerate}
  In this case, the operator \(\pi(\tau_G)\) can be written as a tensor product \(J_L \otimes J_\nu\) of antiunitary involutions on \(\cH_L\), respectively \(\cH_\nu\).
\end{prop}
\begin{proof}
  Recall from Lemma \ref{lem:adm-euler-hsp-homom} that the natural homomorphism \(\gamma: \g \rightarrow \hcsp(V,\Omega)\) maps \(h\) onto an Euler element of \(\hcsp(V,\Omega)\).
  Thus, by Definition \ref{def:urep-nu}, the representation \((\nu_\gamma, \cH_\nu)\) of \(G\) extends to an antiunitary representation of \(G_\tau\), i.e.\ there exists an antiunitary involution \(J_\nu\) on \(\cH_\nu\) such that \(\nu_\gamma(\tau_G(g)) = J_\nu \nu_\gamma(g) J_\nu\) for all \(g \in G\).
  Hence, (a) \(\Rightarrow\) (b) is clear because the tensor product of two antiunitary involutions is a well-defined antiunitary involution.

  Conversely, suppose that there exists an antiunitary involution \(J\) on \(\cH\) with \(\pi(\tau_G(g)) = J\pi(g)J\) for all \(g \in G\).
  We fix an antiunitary involution \(I_L\) on \(\cH_L\). Then we have
  \[(I_L \otimes J_\nu) (\id_{\cH_L} \otimes \nu_\gamma(g)) (I_L \otimes J_\nu) = (\id_{\cH_L} \otimes \nu_\gamma(\tau_G(g))) = J (\id_{\cH_L} \otimes \nu_\gamma(g)) J\]
  for all \(g \in N = \Heis(V,\Omega) \rtimes_\alpha \R_+^\times\). This shows that \(\pi\lvert_N\) can be extended to an antiunitary representation of \(N \rtimes \{\1, \tau_G\lvert_N\}\) by both \(J\) and \(I_L \otimes J_\nu\).
  By \cite[Thm.\ 2.11]{NO17}, there exists a unitary operator \(\Phi \in \pi(N)'\) such that \(\Phi(I_L \otimes J_\nu) = J \Phi\).
  Since \(\pi(N)' = B(\cH_L) \otimes \id_{\cH_\nu}\), we have \(\Phi = \Psi \otimes \id_{\cH_\nu}\) for some \(\Psi \in \U(\cH_L)\), so that \(J = (\Psi I_L \Psi^{-1}) \otimes J_\nu\).
  If we set \(J_L := \Psi I_L \Psi^{-1}\), then we obtain for all \(g \in L\) that
  \[\pi(\tau_G(g)) = J\pi(g)J = (J_L\pi_L(g)J_L) \otimes \nu_\gamma(\tau_G(g))).\]
  Hence, we have
  \[\pi_L(\tau_G(g)) \otimes \id_{\cH_\nu} = \pi(\tau_G(g))(\id_{\cH_L} \otimes \nu_\gamma(\tau_G(g))^{-1}) = (J_L\pi_L(g)J_L) \otimes \id_{\cH_\nu}.\]
  By identifying \(\cH_L\) with \(\cH_L \otimes v \subset \cH_L \hotimes \cH_\nu\) for some unit vector \(v \in \cH_\nu\), the claim follows.
\end{proof}

The following proposition discusses the commutant of antiunitary representations \((\pi, \cH)\) of \(G_\tau\) with discrete kernel for which the restriction to \(N\) is a representation with strictly positive energy.
We recall from \cite[Thm.\ 2.11]{NO17} that, if \((\pi, \cH)\) is irreducible, then the commutant \(\pi(G_\tau)' \subset B(\cH)\) is either isomorphic to \(\R\), \(\C\), or \(\H\). The restriction of \(\pi\) to \(G\) is irreducible if and only if \(\pi(G_\tau)' \cong \R\). In the other two cases, the restriction to \(G\) is a direct sum of two irreducible representations which do not extend to an antiunitary representation of \(G_\tau\) and which are equivalent if and only if \(\pi(G_\tau)' \cong \C\).

\begin{prop}
  \label{prop:adm-arep-commutant}
  Let \((\pi, \cH)\) be a strongly continuous antiunitary representation of \(G_\tau\) with discrete kernel such that \((\pi\lvert_N, \cH)\) is a strictly positive energy representation. Let \((\pi\lvert_G, \cH) \cong (\pi_L \boxtimes \nu_\gamma, \cH_L \hotimes \cH_\nu)\) be the factorization of \(\pi\lvert_G\) from {\rm Theorem \ref{thm:rep-adm-pos-factors}}. Then we have \(\pi(G_\tau)' \cong \pi_L(L_\tau)'\). In particular, \((\pi, \cH)\) is irreducible if and only if the antiunitary extension of \((\pi_L, \cH_L)\) to \(L_\tau\) (cf.\ {\rm Proposition \ref{prop:rep-adm-pos-antiuniext}}) is irreducible.
\end{prop}
\begin{proof}
  We assume that \((\pi\lvert_G, \cH)\) is realized on \(\cH_L \hotimes \cH_\nu\) as in Theorem \ref{thm:rep-adm-pos-factors}.
  Then, by the same theorem, we have \(\pi(G)' = \pi_L(L)' \otimes \id_{\cH_\nu}\).
  By Proposition \ref{prop:rep-adm-pos-antiuniext}, we have \(\pi(\tau_G) = J_L \otimes J_\nu\), and the antiunitary involution \(J_L\) satisfies \(J_L\pi_L(g)J_L = \pi_L(\tau_G(g))\) for all \(g \in L\).
  This implies that \(\pi_L(L_\tau)' \otimes \id_{\cH_\nu} \subset \pi(G_\tau)'\). On the other hand, if \(A \otimes \id_{\cH_\nu} \in \pi(G_\tau)' \subset \pi(G)'\), then we have \((AJ_L) \otimes J_\nu = (J_L A) \otimes J_\nu\), which implies \(AJ_L = J_LA\) and therefore \(A \in \pi_L(L_\tau)'\).
  This proves \(\pi(G_\tau)' = \pi_L(L_\tau)' \otimes \id_{\cH_\nu}\).

  The antiunitary representation \((\pi, \cH)\), resp.\ the antiunitary extension of \((\pi_L, \cH_L)\) to \(L_\tau\), is irreducible if and only if \(\pi(G_\tau)' \cong \pi_L(L_\tau)'\) is isomorphic to \(\R, \C,\) or \(\H\) (cf.\ \cite[Thm.\ 2.11]{NO17}), which finishes the proof.
\end{proof}

\section{Nets of standard subspaces on extended admissible Lie groups}
\label{sec:nets-std-adm}

In this section, we will construct nets of standard subspaces on admissible Lie groups using the methods developed in \cite{NO21}.
In this introduction, we recall the central definitions and results needed for the construction.

We first consider the following setting:
\begin{enumerate}[label=(A\arabic*)]
  \item Let \((G, \tau_G)\) be a connected symmetric Lie group with Lie algebra \(\g\) and let \(\tau := \L(\tau_G)\).
  \item Let \(0 \neq h \in \g^\tau\) be an Euler element of \(\g\).
  \item Let \((\pi, \cH)\) be a strongly continuous antiunitary representation of \(G_\tau = G \rtimes \{\1, \tau_G\}\) with discrete kernel.
We assume that the positive cone \(C_\pi \subset \g\) generates a hyperplane ideal \(\g_C \subset \g\).
  \item We assume that \(\g = \g_C + \R h\). The Euler element \(h\) induces a 3-grading on \(\g\) that we denote by \(\g = \g_{-1} \oplus \g_0 \oplus \g_1\). Obviously, we have \(\g_{\pm 1} \subset \g_C\).
    Moreover, we assume that \(\tau = e^{i\pi\ad(h)}\).
\end{enumerate}

Under the above assumptions, we obtain a standard subspace \(\tV_{(h,\tau_G)}\) of \(\cH\) specified by
\[\Delta^{-it/2\pi} := \pi(\exp(th)) \quad \text{and} \quad J := \pi(\tau_G), \quad t \in \R.\]
Moreover, we set \(G_{\pm 1} := \exp_G(\g_{\pm 1})\).
\begin{definition}
  \label{def:ext-dist-vectors}
  (1) (cf.\ \cite[Def.\ 3.2]{NO21}) Let \(\tV := \tV_{(h,\tau_G)}\) as above. We define
  \[\tV^\infty := \tV \cap \cH^\infty \quad \text{and} \quad \tV^{-\infty} := \{\alpha \in \cH^{-\infty} : \Im \alpha((\tV')^\infty) = \{0\}\}.\]
  (2) (cf.\ \cite[Def.\ 3.6]{NO21}) We define \(\cH^{-\infty}_\ext\) as the set of distribution vectors \(\eta \in \cH^{-\infty}\) for which the orbit map
  \begin{equation}
    \label{eq:ext-dist-vectors-orbit-map}
    \sigma^\eta : \R \rightarrow \cH^{-\infty}, \quad \sigma^\eta(t) := \pi^{-\infty}(\exp(th))\eta = \eta \circ \pi(\exp(-th))
  \end{equation}
  (cf.\ Appendix \ref{sec:app-nets-smooth}), extends to a weak-\(*\)-continuous map on \(\oline\cS_\pi\) which is weak-\(*\)-holomorphic on \(\cS_\pi\), where
  \[\cS_\pi := \{z \in \C : 0 < \Im z < \pi\}.\]
  We denote by \(\cH_{\ext, J}^{-\infty}\) the set of those distribution vectors \(\eta \in \cH_\ext^{-\infty}\) for which the extension of \(\sigma^\eta\) satisfies \(\sigma^\eta(\pi i) = J\eta\).
  By \cite[Lem.\ 3.12]{NO21}, we have \(\cH^{-\infty}_{\ext,J} \subset \tV^{-\infty}\) if \(\L(S_\tV)\) generates \(\g\).

  (3) Fix a left invariant Haar measure on \(G\). For a real subspace \(\sE \subset \cH^{-\infty}\) and an open subset \(\cO \subset G\), we define
  \[\sH_{\sE}(\cO) := \oline{\spann_\R (\pi^{-\infty}(C_c^\infty(\cO,\R))\sE)} \quad \text{and} \quad \cH_{\sE}(\cO) := \oline{\spann_\C (\pi^{-\infty}(C_c^\infty(\cO))\sE)}.\]
  We say that \(\sE\) is \emph{\(G\)-cyclic} if \(\cH_\sE(G) = \cH\).
\end{definition}

The following theorem has been proven in \cite[Thm.\ 5.3]{NO21}, where it is formulated for the case of irreducible antiunitary representations of hermitian simple Lie groups of tube type:
\begin{thm}
  \label{thm:netstand-constr}
  Let \(\sE \subset \cH^{-\infty}\) be a real subspace that satisfies the following conditions:
  \begin{enumerate}[label={\rm (C\arabic*)}]
    \item \(\sE\) is contained in \(\tV^{-\infty}\).
    \item \(\sE\) is invariant under \(A := (\exp(th))_{t \in \R}\).
    \item \(\sE\) is \(G\)-cyclic.
  \end{enumerate}
  Then, for every non-empty open subset \(\cO \subset G\), the real subspace \(\tV(\cO) := \sH_\sE(\cO)\) of \(\cH\) is \emph{cyclic}, i.e.\ \(\tV(\cO) + i\tV(\cO)\) is dense in \(\cH\), and the following assertions hold:
  \begin{enumerate}
    \item For two non-empty open subsets \(\cO_1 \subset \cO_2 \subset G\), we have \(\tV(\cO_1) \subset \tV(\cO_2)\).
    \item For every \(g \in G\) and every non-empty open subset \(\cO \subset G\), we have \(\tV(g\cO) = \pi(g)\tV(\cO)\).
    \item Let \(C_\pm := (\pm C_\pi) \cap \g_{\pm 1}(h)\) and \(S := G^{\tau} \exp(C_+^o \oplus C_-^o)\). Then \(\tV_{(h,\tau_G)} = \sH_\sE(S)\). For every \(g \in G\), the subspace \(\tV(gS)\) is standard and the corresponding pair of modular objects \((\Delta_{\tV(gS)}, J_{\tV(gS)})\) is determined by
      \[\Delta_{\tV(gS)}^{-it/2\pi} = \pi(\exp(t\Ad(g)h)) \quad \text{for } t \in \R \quad \text{and} \quad J_{\tV(gS)} = \pi(g\tau_G(g)^{-1})J.\]
    \item For every non-empty open subset \(\cO \subset G\) that satisfies \(\cO \subset gS\) for some \(g \in G\), the subspace \(\tV(\cO)\) is standard.
    \item \(\tV(gS)' = \tV(gS^{-1}) = \tV(g\tau_G(S))\) for every \(g \in G\).
  \end{enumerate}
\end{thm}
\begin{proof}
  The fact that \(\sE\) is \(G\)-cyclic and \(A\)-invariant implies that \(\sH_{\sE}(\cO)\) is cyclic for every non-empty open subset \(\cO \subset G\) (cf.\ \cite[Thm.\ 3.5]{NO21}). The proof of the remaining statements is identical to the proof of \cite[Thm.\ 5.3]{NO21}.
\end{proof}

\begin{rem}
  \label{rem:netstand-constr-reductions}
  (1) The condition that \(\sE\) is contained in \(\tV^{-\infty}\) and is cyclic can also be checked on a closed subgroup \(H \subset G\).
  To see this, we first note that the space \(\cH^\infty\) of smooth vectors of \((\pi, \cH)\) is contained in the space \(\cH_H^\infty\) of smooth vectors of \((\pi\lvert_H, \cH)\).
  The inclusion \(\cH^\infty \hookrightarrow \cH_H^\infty\) is continuous with respect to the respective \(C^\infty\)-topologies, so that, by taking the adjoint map, we obtain an inclusion \(\cH_H^{-\infty} \hookrightarrow \cH^{-\infty}\) of the corresponding spaces of distributions. Hence, if \(\alpha \in \cH^{-\infty}_H\), we have \(\Im(\alpha(\tV' \cap \cH^\infty)) \subset \Im(\alpha(\tV' \cap \cH^\infty_H))\).

  In order to check the cyclicity condition, we first fix a \(\delta\)-sequence \((\delta_n)_{n \in \N}\) in \(C_c^\infty(G)\) and a left invariant Haar measure \(\mu_H\) on \(H\).
  For every \(f \in L^1(H,\mu_H)\), the measure \(f\mu_H\) is a finite complex Borel measure on \(G\).
  This defines an inclusion \(L^1(H, \mu_H) \hookrightarrow M(G)\) (cf.\ Appendix \ref{sec:app-nets-smooth}).
  If \(\varphi \in C_c^\infty(H)\), then \(\delta_n * (\varphi\mu_H) \in M(G)\) corresponds to an element in \(C_c^\infty(G)\), which we denote by \(\delta_n * \varphi\). We have
  \[\lim_{n \rightarrow \infty}\pi(\delta_n) = \id_\cH \quad \text{and} \quad \lim_{n \rightarrow \infty}\pi(\delta_n * \varphi) = \lim_{n \rightarrow \infty} \pi(\delta_n)\pi(\varphi) = \pi(\varphi)\]
  in the strong operator topology. If \(\eta \in \cH_H^{-\infty}\) is a distribution vector on \(\cH_H^{\infty}\), then we have \(\eta \in \cH^{-\infty}\) by the argument above and \(\pi^{-\infty}(\varphi)\eta \in \cH \subset \cH^{-\infty}\) and therefore
  \[\lim_{n \rightarrow \infty} \pi^{-\infty}(\delta_n * \varphi)\eta = \lim_{n \rightarrow \infty}\pi(\delta_n)\pi^{-\infty}(\varphi)\eta = \pi^{-\infty}(\varphi)\eta.\]
  This shows that \(\pi^{-\infty}(\varphi)\eta \in \sH_{\R\eta}(G)\).

  (2) Suppose that \(G\) is 1-connected. Let \(\g(C_\pi, \tau, h)\) be the ideal of \(\g\) that is generated by \(C_\pi \cap \g_1\) and \(C_\pi \cap \g_{-1}\). Then \(\tilde \g := \g(C_\pi,\tau, h) + \R h\) is \(\tau\)-invariant, so that the closed subgroup \(\tilde G\) of \(G\) corresponding to the subalgebra \(\tilde \g\) satisfies the conditions (A1)-(A4) if the representation \((\pi, \cH)\) is restricted to \(\tilde G\). Combining this observation with (1), we see that it suffices to check the conditions of Theorem \ref{thm:netstand-constr} on the closed subgroup \(\tilde G\).

  (3) The isotony of the net constructed in Theorem \ref{thm:netstand-constr}, resp.\ in \cite{NO21}, follows directly from the definition of \(\sH_\sE(\cO)\) for \(\cO \subset G\) and does not actually require any assumption on the positive cone of \((\pi, \cH)\) or the real subspace \(\sE \subset \cH^{-\infty}\).

  However, in order to ensure that the net contains non-trivial inclusions of standard subspaces, it is reasonable to assume that the endomorphism semigroup \(S_\tV\) of \(\tV\) is ``large'' in the sense that it has a non-empty interior in \(G\) and is not a group.
  The explicit formulas \eqref{eq:lie-wedge-sv} and \eqref{eq:sv-3grad} for \(\L(S_\tV)\), resp.\ \(S_\tV\), from the introduction then show that the element \(h\) with which \(\tV = \tV_{(h,\tau_G)}\) was constructed is an Euler element of \(\g\) and that \(C_\pi\) is non-trivial.
\end{rem}

\begin{rem}
  \label{rem:netstand-constr-cyclic-irred}
  Suppose, in the context of Theorem \ref{thm:netstand-constr}, that \((\pi, \cH)\) is irreducible.
  If \(\sE \neq \{0\}\) and if \(\sE_\C := \sE + i\sE\) is \(\pi^{-\infty}(\tau_G)\)-invariant, then \(\sE\) is automatically \(G\)-cyclic:
  Using a \(\delta\)-sequence in \(C_c^\infty(G)\), it is easy to see that \(\sH_{\sE}(G) \neq \{0\}\).
  Furthermore, we have \(\pi(g)\sH_{\sE}(G) = \sH_{\sE}(gG) = \sH_{\sE}(G)\) by \cite[Lem.\ 2.11]{NO21}, which shows that \(\cH_\sE(G)\) is a non-trivial \(G\)-invariant closed subspace of \(\cH\).

  To see that \(\cH_\sE(G)\) is also invariant under the action of \(\tau_G\), we recall that \(\pi(\tau_G(g)) = J\pi(g)J\) for all \(g \in G\). Moreover, if \(\mu_G\) is a left invariant Haar measure on \(G\), then \(\mu_G \circ \tau_G\) is also left invariant, so that \(\mu_G \circ \tau_G = \mu_G\). Hence, we have for every \(\varphi \in C_c^\infty(G)\) and \(v \in \cH\) that
  \[\pi(\tau_G)\pi(\varphi)v = \int_G \varphi(g)\pi(\tau_G(g))Jv \,d\mu_G(g) = \pi(\varphi \circ \tau_G)Jv.\]
  and therefore \(\pi(\tau_G)\pi(\varphi)\pi(\tau_G) = \pi(\varphi \circ \tau_G)\). In particular, we have
  \[\pi^{-\infty}(\tau_G)\pi^{-\infty}(\varphi) = \pi^{-\infty}(\varphi \circ \tau_G)\pi^{-\infty}(\tau_G),\]
  so that the \(\pi^{-\infty}(\tau_G)\)-invariance of \(\sE_\C\) implies that \(\cH_{\sE}(G)\) is also \(\pi(\tau_G)\)-invariant.
  Since \((\pi, \cH)\) is irreducible, we thus have \(\cH_\sE(G) = \cH\).
\end{rem}

\subsection{Tensor products of standard subspaces and distribution vectors}
\label{sec:std-tensor-dist}

In view of the Factorization Theorem \ref{thm:rep-adm-pos-factors}, it is natural to construct a suitable subspace \(\sE\) of distribution vectors satisfying the conditions (C1)-(C3) of Theorem \ref{thm:netstand-constr} in two steps: First, we study subspaces \(\sE_\nu \subset \cH_\nu^{-\infty}\) for the representation \((\nu, \cH_\nu)\) of \(\tilde\HCSp(V,\Omega)\). Next, we construct a subspace \(\sE_L \subset \cH_L^{-\infty}\) of the representation \((\pi_L, \cH_L)\) using the results from \cite{NO21} and verify that the tensor product \(\sE_L \otimes \sE_\nu\) satisfies (C1)-(C3).
This requires us to take a closer look at tensor products of standard subspaces and tensor products of spaces of distribution vectors.

\begin{definition}
  \label{def:std-subspace-tensorprod}
  Let \(\tV_k \subset \cH_k\) be standard subspaces of complex Hilbert spaces \(\cH_k\) for \(k=1,2\) and denote the corresponding pair of modular objects by \((\Delta_k, J_k)\).
  Then we obtain a continuous antiunitary representation \((\rho, \cH_1 \hotimes \cH_2)\) of \(\R^\times\) via
  \[\rho(e^t) = \Delta^{-it/2\pi}_1 \otimes \Delta^{-it/2\pi}_2 \quad \text{for } t \in \R  \quad \text{and} \quad \rho(-1) := J_1 \otimes J_2.\]
  The representation \(\rho\) determines a pair \((\Delta, J)\) of modular objects on \(\cH_1 \hotimes \cH_2\) given by \(\Delta := e^{2\pi i\partial\rho(1)}\) and \(J := \rho(-1)\) (cf.\ \cite[Lem.\ 2.21]{NO17}), which in turn yields a standard subspace \(\tV := \Fix(J\Delta^{1/2})\) of \(\cH_1 \hotimes \cH_2\).
  We call \(\tV =: \tV_1 \hotimes \tV_2\) the \emph{tensor product of \(\tV_1\) and \(\tV_2\)}.
\end{definition}

\begin{rem}
  \label{def:std-subspace-tensorprod-dense}
  For standard subspaces \(\tV_1 \subset \cH_1\) and \(\tV_2 \subset \cH_2\), the real algebraic tensor product \(\tV_1 \otimes \tV_2\) of \(\tV_1\) and \(\tV_2\) is contained in \(\tV_1 \hotimes \tV_2\): To see this, consider \(v_k \in \tV_k\) (\(k=1,2\)).
  Then the orbit maps \(\alpha_k : t \mapsto \Delta_k^{-it/2\pi}v_k\) on \(\R\) have a continuous extension to the strip \(\oline{\cS_\pi}\) which is holomorphic on the interior with \(\alpha_k(\pi i) = J_kv_k\).
  Hence, the tensor product \(t \mapsto \alpha_1(t) \otimes \alpha_2(t)\) on \(\R\) has a continuous extension to \(\oline{\cS_\pi}\) which is holomorphic on the interior and satisfies \(\alpha(\pi i) = J_1v_1 \otimes J_2 v_2\).
  Since this holomorphic extension property characterizes elements of the standard subspace (cf.\ \cite[Prop.\ 2.1]{NOO21}), we have \(v_1 \otimes v_2 \in \tV_1 \hotimes \tV_2\).

  To see that \(\tV_1 \otimes \tV_2\) is dense in \(\tV_1 \hotimes \tV_2\), we first note that \((\tV_1 \otimes \tV_2)'' = \oline{\tV_1 \otimes \tV_2}\) is a standard subspace of \(\cH_1 \hotimes \cH_2\) because it is contained in a separating subspace and because \((\tV_1 + i\tV_1) \otimes (\tV_2 + i\tV_2)\) is dense in \(\cH_1 \hotimes \cH_2\).
  Moreover, it is obvious from the definition of \(\tV_1 \hotimes \tV_2\) that \(\Delta_{\tV_1 \hotimes \tV_2}^{it}\oline{\tV_1 \otimes \tV_2} = \oline{\tV_1 \otimes \tV_2}\) for all \(t \in \R\).
  Hence, we actually have \(\tV_1 \hotimes \tV_2 = \oline{\tV_1 \otimes \tV_2}\) by \cite[Prop.\ 3.10]{Lo08}.
\end{rem}

\begin{lem}
  \label{lem:std-tensor-dists}
  For \(k=1,2\), let \(G_k\) be a connected Lie group and let \(\tau_{G_k}\) be an involutive automorphism of \(G_k\). Let \((\pi_k, \cH_k)\) be a strongly continuous antiunitary representation of \(G_{k,\tau} = G_k \rtimes \{\1, \tau_{G_k}\}\).
  We endow \(\cH := \cH_1 \hotimes \cH_2\) with the tensor product representation \(\pi := \pi_1 \otimes \pi_2\) of \(G_1 \times G_2\) (cf. {\rm Appendix \ref{sec:app-tensor-uni-rep}}).
  Let \(h_k \in \L(G_k)\) be a fixed point of \(\tau_k := \L(\tau_{G_k})\) and let \(\tV_k \subset \cH_k\) be the standard subspace specified by \(J_{\tV_k} = \pi_k(\tau_{G_k})\) and \(\Delta_{\tV_k}^{-it/2\pi} = \pi_k(\exp(th_k))\) for \(t \in \R\).
  Suppose furthermore that \(\1_{G_k} \in \oline{S_{\tV_k}^o}\).
  For \(\tV := \tV_1 \hotimes \tV_2\), the following holds:
  \begin{enumerate}
  \item \(\tV_1^\infty \otimes \tV_2^\infty\) is dense in \(\tV^\infty\) with respect to the \(C^\infty\)-topology on \(\cH^\infty\).
\item For every pair \((\eta_1, \eta_2) \in \tV_1^{-\infty} \times \tV_2^{-\infty}\) of distribution vectors, the tensor product distribution vector \(\eta_1 \hotimes \eta_2 \in \cH^{-\infty}\) (cf.\ {\rm Lemma \ref{lem:smoothvec-tensor-dist-ext}}) is contained in \(\tV^{-\infty}\).
  \end{enumerate}
\end{lem}
\begin{proof}
  (a) Since \(\1_{G_k} \in \oline{S_{\tV_k}^o}\), the interior of \(S_{\tV_k}\) is dense in \(S_{\tV_k}\) (cf.\ \cite[Prop.\ V.0.15]{HHL89}).
  Let \(v \in \tV^\infty\) and let \((\delta_{k,n})_{n \in \N}\) be a \(\delta\)-sequence of compactly supported functions on \(S_{\tV_k}^o\) for \(k=1,2\).
  Then \(\pi(\delta_{1,n} \otimes \delta_{2,n})v\) converges to \(v\) in the \(C^\infty\)-topology.
  Hence, we may assume that \(v = \pi(\varphi_1 \otimes \varphi_2)w\) for some \(\varphi_k \in C_c^\infty(S_{\tV_k}^o)\) and \(w \in \tV^\infty\).
  For every sequence \((w_n)_{n \in \N}\) in \(\tV_1 \otimes \tV_2\) converging to \(w\) in \(\cH\) (cf.\ Remark \ref{def:std-subspace-tensorprod-dense}), the sequence \((\pi(\varphi_1 \otimes \varphi_2)w_n)_{n \in \N}\) is contained in \(\tV_1^\infty \otimes \tV_2^\infty\) and converges to \(v\) in the \(C^\infty\)-topology, which proves the claim.

  (b) Let \((\eta_1, \eta_2) \in \tV_1^{-\infty} \times \tV_2^{-\infty}\). Since \(\eta_1 \hotimes \eta_2\) is continuous on \(\cH^\infty\) by Lemma \ref{lem:smoothvec-tensor-dist-ext}, the claim follows from (a) and the fact that \(\Im (\eta_1 \otimes \eta_2)\) vanishes on the \(C^\infty\)-dense subspace \((\tV_1')^\infty \otimes (\tV_2')^\infty\).
\end{proof}

\begin{lem}
  \label{lem:std-boxtensor-dists}
  Let \((G,\tau_G)\) be a connected symmetric Lie group.
  For \(k=1,2\), let \((\pi_k, \cH_k)\) be a strongly continuous antiunitary representation of \(G_{\tau} = G \rtimes \{\1, \tau_{G}\}\).
  We endow \(\cH := \cH_1 \hotimes \cH_2\) with the tensor product representation \(\pi := \pi_1 \boxtimes \pi_2\) of \(G\) (cf.\ {\rm Appendix \ref{sec:app-tensor-uni-rep}}).
  Let \(h \in \L(G)\) be a fixed point of \(\tau = \L(\tau_G)\) and let \(\tV_k \subset \cH_k\) be the standard subspace specified by \(J_{\tV_k} = \pi_k(\tau_{G_k})\) and \(\Delta_{\tV_k}^{-it/2\pi} = \pi_k(\exp(th))\) for \(t \in \R\). Let \(\tV := \tV_1 \hotimes \tV_2\)
  Suppose furthermore that \(\1 \in G\) is contained in the closure of the interior of the semigroup \(S := S_{\tV_1} \cap S_{\tV_2} \subset G\).
  Then the following holds:
  \begin{enumerate}
  \item \(\tV_1^\infty \otimes \tV_2^\infty\) is dense in \(\tV^\infty\) with respect to the \(C^\infty\)-topology on \(\cH^\infty\).
  \item For every pair \((\eta_1, \eta_2) \in \tV_1^{-\infty} \times \tV_2^{-\infty}\) for which the tensor product \(\eta_1 \otimes \eta_2\) extends to a continuous antilinear functional \(\eta_1 \hotimes \eta_2\) on \(\cH^\infty\), we have \(\eta_1 \hotimes \eta_2 \in \tV^{-\infty}\).
  \end{enumerate}
\end{lem}
\begin{proof}
  (a) By assumption, we have \(\1 \in \oline{S_{\tV_k}^o}\) for \(k=1,2\).
  By applying Lemma \ref{lem:std-tensor-dists} to the tensor product representation \((\pi_1 \otimes \pi_2, \cH)\) of \(G \times G\) and \(h_k := h\) for \(k=1,2\), we see that \(\tV_1^\infty \otimes \tV_2^\infty\) is dense in \(\tV \cap \cH^\infty_{\pi_1 \otimes \pi_2}\).
  Since \(\cH^\infty_{\pi_1 \otimes \pi_2} \subset \cH^\infty = \cH^\infty_{\pi}\), it remains to show that \(\tV \cap \cH^\infty_{\pi_1 \otimes \pi_2}\) is dense in \(\tV^\infty = \tV \cap \cH^\infty\).

  Let \(v \in \tV^\infty\) and choose a \(\delta\)-sequence \((\delta_n)_{n \in \N}\) in \(C_c^\infty(S^o)\).
  Then the sequence
  \[((\pi_1 \otimes \pi_2)(\delta_n \otimes \delta_n)v)_{n \in \N} \subset \cH^\infty_{\pi_1 \otimes \pi_2} \cap \tV\]
  converges to \(v\) in the \(C^\infty\)-topology of \(\cH^\infty\), which shows that \(\tV \cap \cH^\infty_{\pi_1 \otimes \pi_2}\) is dense in \(\tV^\infty\).
  This proves (a).

  (b) follows from (a) and the fact that \(\Im(\eta_1 \otimes \eta_2)\) vanishes on the \(C^\infty\)-dense subspace \((\tV_1')^\infty \otimes (\tV'_2)^\infty\).
\end{proof}

\begin{rem}
  \label{rem:tensor-dist-semigrp-cond}
  Let \(G\) be a connected Lie group with Lie algebra \(\g\) and let \(S \subset G\) be a closed subsemigroup. If the Lie wedge
  \[\L(S) := \{x \in \g : \exp(\R_+ x) \subset S\}\]
  of \(S\) generates \(\g\) in the sense of Lie algebras, then \(\1 \in G\) is contained in the closure of the interior of \(S\) (cf.\ \cite[Thm.\ 2.1]{HR91}).

  In the setting of (A1)-(A4) from the beginning of Section \ref{sec:nets-std-adm}, the Lie wedge of the endomorphism semigroup \(S_{\tV}\) of the standard subspace \(\tV := \tV_{(h,\tau_G)}\) is given by
  \begin{equation}
    \label{eq:sv-lie-wedge-3grad}
    \L(S_{\tV}) = C_-(\pi,h) \oplus \g_0(h) \oplus C_+(\pi,h) \quad \text{with} \quad C_{\pm}(\pi,h) := \g_{\pm 1}(h) \cap \pm C_\pi
  \end{equation}
  (\cite[Cor.\ 5.5]{Ne21}). In particular, \(\L(S_\tV)\) generates \(\g\) if \(C_{\pm}(\pi,h) - C_{\pm}(\pi,h) = \g_{\pm 1}(h)\).

\end{rem}

\subsection{Nets of standard subspaces on \texorpdfstring{\((\nu, \cH_\nu)\)}{(nu, H nu)}}
\label{sec:nets-std-H-nu}

We preserve the notation from the introduction of Section \ref{sec:pos-energy-rep-ext-adm} and Subsection \ref{sec:urep-conf-jacobi}.
In this section, we show that there exists a subspace of distribution vectors which satisfies conditions (C1)-(C3) of Theorem \ref{thm:netstand-constr} in the case of the unique irreducible strictly positive energy representation \((\nu, \cH_\nu)\) of the universal covering of the conformal Jacobi group \(G := \tilde\HCSp(V,\Omega)\) for which the restriction to \(N = \Heis(V,\Omega) \rtimes_\alpha \R_+^\times\) is irreducible.

This construction generalizes to the pullback representations \((\nu_\gamma, \cH_\nu)\) of extended admissible Lie groups that occur in the context of the Factorization Theorem \ref{thm:rep-adm-pos-factors}.
As we will show in Section \ref{sec:nets-adm-general} below, this construction is also extendable to tensor product representations on Hilbert spaces of the form \(\cH_L \hotimes \cH_\nu\), where the subgroup \(N\) acts trivially on \(\cH_L\).

Throughout this section, we assume that \((\nu, \cH_\nu)\) is realized on \(\cH_\nu = L^2(\R_+^\times \times V_1, \frac{d\lambda}{\lambda} \otimes dx)\) as in Definition \ref{def:urep-nu}.
Since we have fixed an Euler element \(h = \frac{1}{2}\id_V + h_s\) of \(\hcsp(V,\Omega)\), we can identify \(\HCSp(V,\Omega)\) with \(\HCSp(\R^{2n})\), where \(n := \dim(V_1)\), in such a way that \(h_s\) corresponds to the element
\[\frac{1}{2}\pmat{\bbone_n & 0 \\ 0 & -\bbone_n}.\]

\begin{lem}
  \label{lem:H-nu-dist-smooth}
  Let \(\rho := \nu\lvert_{V_{-1} \times T}\) be the restriction of \((\nu, \cH_\nu)\) (cf.\ {\rm Definition \ref{def:urep-nu}}) to the subgroup \(V_{-1} \times T\). For \(s \in \C\) with \(\Re(s) > \frac{n}{2}\), where \(n = \dim(V_1)\), the integrals
  \begin{equation}
    \label{eq:H-nu-dist-smooth}
    \eta_s(f) := \int_{\R_+^\times}\int_{V_1} \oline{f(\lambda,x)}\lambda^s dx \frac{d\lambda}{\lambda}, \quad f \in \cH_{\nu,\rho}^\infty := (\cH_\nu)^\infty_\rho,
  \end{equation}
  exist. Moreover, \(\eta_s\) defines a continuous antilinear functional on the space \(\cH_{\nu,\rho}^\infty = (\cH_\nu)_\rho^{\infty}\) of smooth vectors of the representation \((\rho,\cH_\nu)\) and, in particular, on \(\cH_\nu^\infty \subset (\cH_\nu)_\rho^\infty\).
\end{lem}
\begin{proof}
  We recall from \eqref{eq:conf-heis-irrep-expl-1} that
  \[(\rho(q,z)f)(\lambda,x) = (\nu(0,q,z)f)(\lambda,x) = e^{i\lambda^2 z}e^{i\lambda\Omega(q,x)}f(\lambda, x) \quad \text{for } (q,z) \in V_{-1} \times T, f \in \cH_\nu.\]
  In particular, the universal enveloping algebra \(\cU(\ft)\) of \(\ft := \L(T)\) acts on \(\cH_{\nu,\rho}^\infty\) by multiplication with quadratic polynomials in the variable \(\lambda\) and \(\cU(V_{-1})\) acts by multiplication with functions of the form \((\lambda,x) \mapsto p(\lambda x_1,\ldots, \lambda x_n)\), where \(p \in \R[x_1,\ldots,x_n]\) is a polynomial.

  Fix \(m \in \N\) with \(m > \frac{1}{2}\Re(s)\) and let
  \[P(\lambda,x) := (1 + \lambda^2 + \lambda^2\|x\|^2)^m, \quad \lambda \in \R_+^\times, x \in V_1.\]
  Then the integral
  \[\int_{\R_+^\times}\int_{V_1} \left|\frac{\lambda^s}{P(\lambda,x)}\right|^2\, dx \frac{d\lambda}{\lambda} = \int_{\R_+^\times}\int_{V_1} \frac{\lambda^{2 \Re(s)}}{P(\lambda,x)^2}\, dx \frac{d\lambda}{\lambda}\]
  is finite.
  To see this, it suffices to show that the integral
  \[\cI := \int_0^\infty \int_{\R^n} \frac{\lambda^{2\Re(s) - 1}}{(1 + \lambda^2 + \lambda^2 \|x\|^2)^{2m}} \, dx d\lambda = \int_0^\infty \int_{\R^n} \frac{\lambda^{2\Re(s) - (n+1)}}{(1 + \lambda^2 + \|x\|^2)^{2m}} \, dx d\lambda\]
  is finite.
  The second equality in the above equation follows by rescaling \(x\).
  First, we note that, for a constant \(c \geq 1\), we have
  \begin{align*}
    \int_{-\infty}^\infty \frac{1}{(c + x^2)^{2m}} \, dx &\leq 2 \int_{0}^\infty \frac{1}{c^{2m} + x^{4m}} dx\\
                                                     &= 2 \left( \int_{0}^1 \frac{1}{c^{2m} + x^{4m}} dx + \int_1^\infty \frac{1}{1 + x^{2}} \, dx\right)\\
                                                     &\leq \frac{2}{c^{2m}} + \frac{\pi}{2}.
  \end{align*}
  We apply this argument recursively to \(\cI\) using Fubini's Theorem and obtain
  \[\cI \leq 2^n \int_0^\infty \frac{\lambda^{2\Re(s) - (n+1)}}{(1 + \lambda^2)^{2m}}\, d\lambda + C\]
  for some constant \(C > 0\). The integral on the right hand side exists because \(2\Re(s) - n > 0\) and \(2m > \Re(s)-\frac{n}{2}\) by assumption. This proves \(\cI < \infty\).

  For every smooth vector \(f \in \cH_{\rho,\nu}^\infty\), the product \(Pf\) is contained in \(\cH_\nu\).
  As a result, the integral \eqref{eq:H-nu-dist-smooth} exists and we have
  \[\eta_s(f) = \int_{\R_+^\times}\int_{V_1} \oline{P(\lambda,x)f(\lambda,x)} \frac{\lambda^s}{P(\lambda,x)} \, dx \frac{d\lambda}{\lambda} = \left\la Pf, \frac{\lambda^s}{P(\lambda,x)}\right\ra_{\cH_\nu}.\]
  This shows that \(\eta_s \in \spann(\dd\rho^{-\infty}(\cU(V_{-1} \times \ft))\cH_\nu) \subset \cH_{\nu,\rho}^{-\infty}\) (cf.\ \cite[Lem.\ A.2]{NO21}). In particular, \(\eta_s\) is continuous.
\end{proof}

\begin{rem}
  If \(V = \{0\}\), then \(N \cong \R \rtimes_\alpha \R_+^\times\) is isomorphic to the identity component of the \(ax+b\)-group and \((\nu_\tau, \cH_\nu)\) is the irreducible antiunitary representation of strictly positive energy. In this case, the distribution vector discussed in Lemma \ref{lem:H-nu-dist-smooth} is equivalent to the distribution vector discussed in \cite[Prop.\ 3.17]{NO21}.
\end{rem}

\begin{prop}
  \label{prop:H-nu-dist-invariance}
  Let \(s \in \C\) with \(\Re(s) > \frac{n}{2}\). Let \(\fs := \sp(V,\Omega)\) and \(\fs_k := \fs_k(h_s)\) for \(k=-1,0,1\).
  Then the following holds:
  \begin{enumerate}
    \item Let \[y \in \fs_0 \cong \left\{\pmat{A & 0 \\ 0 & -A^\tran} : A \in M_n(\R)\right\} \subset \sp(2n,\R).\] Then \(\nu^{-\infty}(\exp(y))\eta_s = |\det(\exp(y)\lvert_{V_1})|^{-1/2}\eta_s\) for all \(t \in \R\).
    \item \(\nu^{-\infty}(\exp(th))\eta_s = e^{\frac{t}{2}(s-\frac{n}{2})}\eta_s\), where \(n = \dim(V_1)\).
    \item The distribution vector \(\eta_s\) is invariant under the subgroup \(V_1 \times \tilde\Sp(V,\Omega)_1\).
    \item \(\nu^{-\infty}(\tau_G)\eta_s = -\eta_{\oline{s}}\).
  \end{enumerate}
\end{prop}
\begin{proof}
  As described in the introduction to Section \ref{sec:nets-std-H-nu}, we can identify \(\HCSp(V,\Omega)\) with \(\HCSp(\R^{2n})\) and \(V_{\pm 1}\) with \(\R^n\) by choosing a symplectic basis, so that we may assume that
  \[\cH_\nu = L^2(\R_+^\times \times \R^n, \tfrac{d\lambda}{\lambda} \otimes dx)\]
  and that \(\Omega\) is the standard symplectic form on \(\R^{2n}\).
  The restriction of the representation \((\nu, \cH_\nu)\) to \(\Heis(\R^{2n})\) can then be written as
  \begin{equation}
    \label{eq:H-nu-r2n}
  (\nu(p,q,z)f)(\lambda,x) = e^{i\lambda^2(z - \frac{1}{2}\la q,p\ra)}e^{i\lambda \la q, x\ra}f(\lambda, x - \lambda p), \quad p,q,x \in \R^n, z \in \R, \lambda > 0,
  \end{equation}
  where \(\la \cdot, \cdot \ra\) denotes the euclidean scalar product on \(\R^{n}\) with \(\Omega((0,q),(p,0)) = \la q,p\ra\).
  Let \(f \in \cH_\nu^\infty\).

  (a) Let \(y \in \fs_0\). Using the explicit description of the metaplectic representation for \(\exp(\fs_0)\) (\cite[p. 191]{Fo89}), we then see that
  \begin{align*}
    (\nu^{-\infty}(\exp(y))\eta_s)(f) &= \int_{\R_+^\times \times \R^n} |\det(\exp(y)\lvert_{V_1})|^{1/2}\oline{f(\lambda, \exp(y)x)}\lambda^s\, \frac{d\lambda}{\lambda} \otimes dx \\
                                      &= \int_{\R_+^\times \times \R^n} |\det(\exp(y)\lvert_{V_1})|^{-1/2}\oline{f(\lambda,x)}\lambda^s \,\frac{d\lambda}{\lambda} \otimes dx \\
                                      &= |\det(\exp(y)\lvert_{V_1})|^{-1/2}\eta_s(f).
  \end{align*}

  (b) Let \(t \in \R\). Then
  \begin{align*}
    (\nu^{-\infty}(\exp(t \id_V))\eta_s)(f) &= \int_{\R_+^\times \times \R^n} \oline{f(e^{-t}\lambda, x)}\lambda^s \, \frac{d\lambda}{\lambda} \otimes dx = \int_{\R_+^\times \times \R^n} \oline{f(\lambda, x)}(e^{t}\lambda)^s \, \frac{d\lambda}{\lambda} \otimes dx \\
                                            &= e^{ts}\eta_s(f).
  \end{align*}
  Combining this with (a), we obtain \[\nu^{-\infty}(\exp(th))\eta_s = \nu^{-\infty}(\exp(th_s))\nu^{-\infty}(\exp(\tfrac{t}{2}\id_V))\eta_s = e^{\frac{t}{2}(s-\frac{n}{2})}\eta_s.\]

  (c) Let \(p \in V_1\). Then we have
  \[(\nu^{-\infty}(p,0,0)\eta_s)(f) = \int_{\R_+^\times \times \R^n} \oline{f(\lambda, x + \lambda p)}\lambda^s \, \frac{d\lambda}{\lambda} \otimes dx = \int_{\R_+^\times \times \R^n} \oline{f(\lambda, x)}\lambda^s \, \frac{d\lambda}{\lambda} \otimes dx = \eta_s(f).\]
  It remains to show the invariance under the abelian subgroup \(\tilde\Sp(\R^{2n})_1\). We first note that
  \[\fs_{-1}(h_s) = \left\{\pmat{0 & 0 \\ C & 0} : C^\tran = C\right\} \quad \text{and} \quad \fs_1(h_s) = \left\{\pmat{0 & B \\ 0 & 0} : B^\tran = B\right\}.\]
  Next, we recall from \cite[p.\ 179]{Fo89} that the element \(I := \pmat{0 & -\bbone \\ \bbone & 0} \in \Sp(\R^{2n})\) is represented under \(\nu\), up to a scalar factor, by the inverse Fourier transform
  \[(\cF_2^{-1} f)(\lambda, x) := \frac{1}{(2\pi)^{n/2}}\int_{\R^n} e^{i \la x,y\ra} f(\lambda, y)\, dy, \quad f \in \cH_\nu,\]
  since \(\nu(p,q)\cF_2 = \cF_2\nu(q,-p)\) on \(\cH_\nu\) for all \(p,q \in \R^n\). Every element \(\tilde B \in \Sp(\R^{2n})_1\) is of the form \(I{\tilde C}I^{-1}\) for some \(\tilde C = \pmat{\bbone & 0 \\ C & \bbone} \in \Sp(\R^{2n})_{-1}\) with \(C^\tran = C\).
  By \cite[p.\ 191]{Fo89}, we have,
  \[(\nu(\tilde C)f)(\lambda, x) = e^{i\la Cx, x\ra}f(\lambda,x), \quad f \in \cH_\nu.\]
  As a result, we obtain
  \begin{align*}
    (\nu^{-\infty}(\tilde B)\eta_s)(f) &= \eta_s(\nu(I\tilde C^{-1}I^{-1})f) = \frac{1}{(2\pi)^{n/2}}\int_{\R_+^\times \times \R^n} \int_{\R^n} \oline{e^{i\la x,y\ra} e^{-i\la C y, y\ra}(\cF_2 f)(\lambda, y)} \lambda^s dy\frac{d\lambda}{\lambda} \otimes dx \\
                                       &= \frac{1}{(2\pi)^{n/2}}\int_{\R^n} \int_{\R_+^\times \times \R^n} \oline{e^{i\la x - Cy,y\ra} (\cF_2 f)(\lambda, y)} \lambda^s \frac{d\lambda}{\lambda} \otimes dx dy \\
                                       &= \frac{1}{(2\pi)^{n/2}}\int_{\R^n} \int_{\R_+^\times \times \R^n} \oline{e^{i\la x ,y\ra} (\cF_2 f)(\lambda, y)} \lambda^s \frac{d\lambda}{\lambda} \otimes dx dy \\
                                       &= \eta_s(\nu(II^{-1})f) = \eta_s(f).
  \end{align*}

  (d) follows from
  \[(\nu^{-\infty}(\tau_G)\eta_s)(f) = \oline{\eta_s(\nu(\tau_G)f)} = \oline{\int_{\R_+^\times \times \R^n} f(\lambda,-x)\lambda^s\frac{d\lambda}{\lambda} \otimes dx} = -\eta_{\oline{s}}(f). \qedhere\]
\end{proof}

\begin{thm}
  \label{thm:std-dist-H-nu}
  For \(s > \frac{n}{2} \), let \(\tilde\eta_s := e^{i\frac{\pi}{2}(1 - \frac{s}{2} + \frac{n}{4})}\eta_s\), where \(n = \dim(V_1)\) and \(\eta_s\) is defined as in \eqref{eq:H-nu-dist-smooth}. The real vector space
  \begin{equation}
    \label{eq:std-dist-H-nu}
    \sE_\nu := \spann_\R \{\tilde\eta_s : s > \tfrac{n}{2}\}
  \end{equation}
  is contained in \(\cH^{-\infty}_{\nu,\ext,J_\nu}\) and is cyclic under the action of the group \(N = \Heis(V,\Omega) \rtimes_\alpha \R_+^\times\). In particular, \(\sE_\nu\) satisfies the conditions {\rm (C1)-(C3)} of {\rm Theorem \ref{thm:netstand-constr}} with respect to the standard subspace \(\tV\) specified by \((h,\tau_G)\).
\end{thm}
\begin{proof}
  We recall from Proposition \ref{prop:H-nu-dist-invariance} that the orbit map \(\sigma^{\eta_s}\) (cf.\ \eqref{eq:ext-dist-vectors-orbit-map}) satisfies \(\sigma^{\eta_s}(t) = e^{\frac{t}{2}(s-\frac{n}{2})}\eta_s\) for \(t \in \R\). In particular, \(\sigma^{\eta_s}\) extends via \(\sigma^{\eta_s}(z) := e^{\frac{z}{2}(s-\frac{n}{2})}\eta_s\) to a weak-\(*\)-continuous map on \(\oline{\cS_\pi}\) which is weak-\(*\)-holomorphic on \(\cS_\pi\). Since we have also shown in Proposition \ref{prop:H-nu-dist-invariance} that
  \[\nu^{-\infty}(\tau_G)\eta_s = J_\nu\eta_s = -\eta_s,\]
  it follows that \(\sE_\nu \subset \cH_{\nu,\ext,J_\nu}^{-\infty}\).
  In order to show that \(\sE_\nu \subset \tV^{-\infty}\), it therefore suffices to show that \(\cH_{\nu, \ext,J_\nu}^{-\infty} \subset \tV^{-\infty}\).
  By \cite[Lem.\ 3.12]{NO21}, this is the case if \(\L(S_{\tV})\) (cf.\ \eqref{eq:sv-lie-wedge-3grad}) spans \(\hcsp(V,\Omega)\).
  Since the positive cone of \((\nu,\cH_\nu)\) is given by \(\hsp(V,\Omega)_+\) (cf.\ Remark \ref{rem:rep-nu-pos-cone-bounded-cone}), it is a consequence of \cite[Thm.\ 3.17]{Oeh20b} that \(\L(S_\tV)\) generates \(\hcsp(V,\Omega)\).
  Therefore, we have \(\sE_\nu \subset \tV^{-\infty}\).

  That \(\sE_\nu\) is \(N\)-cyclic follows from Remark \ref{rem:netstand-constr-cyclic-irred} because \(\sE_\nu \neq \{0\}\) and because the restriction of \((\nu, \cH_\nu)\) to \(N\) is irreducible.
\end{proof}

In the remainder of this subsection, we discuss a second construction of a real subspace \(\sE_P \subset \cH_\nu^{-\infty}\) that satisfies (C1)-(C3).
To this end, we make use of the fact that \(\hcsp(V,\Omega)\) can be embedded into \(\sp(\tilde V, \tilde \Omega)\), where \(\dim \tilde V = \dim V + 2\), in such a way that the Euler element \(h \in \hcsp(V,\Omega)\) is also an Euler element of \(\sp(\tilde V, \tilde \Omega)\) and \(\sp(\tilde V, \tilde \Omega)_1(h) = \hcsp(V,\Omega)_1(h)\) (cf.\ Theorem \ref{thm:h1-emb-euler-elements}).
Recall that the representation \((\nu_\tau, \cH_\nu)\) extends to an antiunitary representation \((\tilde \nu_\tau, \cH_\nu)\) of \(\tilde\Sp(\tilde V, \tilde \Omega)_\tau\) up to coverings (cf.\ Proposition \ref{prop:rep-extjac-metaplectic-sp-ext}).
The positive cone of this representation generates \(\sp(\tilde V, \tilde \Omega)\) because it contains the positive cone of \((\nu, \cH_\nu)\) (cf.\ Proposition \ref{prop:adm-parabolic-emb-cone}).
Hence, we can apply \cite[Thm.\ 5.3]{NO21} to the representation \((\tilde \nu_\tau, \cH_\nu)\) and obtain a real subspace \(\sE_P \subset \tV_{(h,\tau)}^{-\infty}\) that satisfies (C1)-(C3).
The subspace \(\sE_P\) actually consists of distribution vectors of the closed subgroup \(\Sp(\tilde V, \tilde \Omega)_1\) and is cyclic under the action of this group (cf.\ \cite[Prop. 4.13]{NO21}).
Hence, \(\sE_P\) also satisfies (C1)-(C3) for the representation \((\nu_\tau, \cH_\nu)\) (cf.\ Remark \ref{rem:netstand-constr-reductions}).
By \cite[Prop. 5.2]{NO21}, the space \(\sE_P\) is invariant under the subgroup \(\Sp(\tilde V, \tilde \Omega)_{-1}\Sp(\tilde V, \tilde \Omega)^{\tau_S}\). In particular, it is invariant under the subgroups \(V_{-1} \times \fz\) and \(\tilde\Sp(V, \Omega)_{-1}\tilde\Sp(V,\Omega)^{\tau}\) of \(\tilde\HCSp(V,\Omega)\). To summarize, we obtain the following result:

\begin{thm}
  \label{thm:std-dist-parabolic-embedding}
  The real subspace \(\sE_P\) constructed above consists of distribution vectors for the restriction of \((\nu, \cH_\nu)\) to \(G_1 = V_1 \times T \times \tilde\Sp(V,\Omega)_1\), where \(T = Z(\Heis(V,\Omega)) \cong \R\), and satisfies {\rm (C1)-(C3)}. Moreover, \(\sE_P\) is invariant under the subgroup \(\tilde\Sp(V,\Omega)_{-1}\tilde\HCSp(V,\Omega)^\tau\).
\end{thm}

\begin{rem}
  The subspace \(\sE_P\) is realized as follows: As described in Remark \ref{rem:h1-emb-euler-jordan}, the subspace \(E := \sp(\tilde V, \tilde \Omega)_1(h)\) can be endowed with the structure of a simple euclidean Jordan algebra. The antiunitary representation \((\tilde \nu_\tau, \cH_\nu)\) of \(\Sp(\tilde V, \tilde \Omega)\) can then be realized on \(\cH_\nu := L^2(E, \mu_\alpha)\), where \(\mu_\alpha\) is a Riesz measure on \(E\) with a parameter \(\alpha\) (cf.\ Appendix \ref{sec:app-riesz-measures}). The action of \(\exp(E) \cong E\) is then given by \((\tilde\nu(x)f)(y) := e^{-\frac{i}{2}\la y, x\ra}f(y)\) for \(x,y \in E\). The distribution vectors in \(\sE_P\) are multiples of 
  \[\eta(f) := \int_E \oline{f(x)}\, d\mu_{\alpha}(x), \quad f \in \cH_\nu^{\infty}\]
  (cf.\ \cite[Sec.\ 4.4]{NO21}). In our case, we have \(\alpha = \frac{1}{2}\) by \cite[Thm.\ 5.8]{HN01} because the highest weight \(\lambda\) of \(\tilde \nu\) satisfies \(\lambda\lvert_\fa = -\frac{1}{2}\sum_{j=1}^{r+1} \varepsilon_j\).
\end{rem}

\begin{rem}
  \label{rem:std-dist-parabolic-emb-general-case}
  We recall that Theorem \ref{thm:h1-emb-euler-elements}, which proves the existence of a common Euler element of \(\sp(\tilde V, \tilde \Omega)\) and \(\hcsp(V, \Omega) \cong \fb_\kappa\), also holds for every other maximal parabolic subalgebra of the form \(\fb_\kappa\) (cf.\ Proposition \ref{prop:h1-multone-parabolic-subalg}) for which the subalgebra \(\g_\kappa\) is simple hermitian and of tube type (cf.\ Table \ref{table:h1-multone-class}).
  The results from \cite{HNO96b} which are used in the proof of Proposition \ref{prop:rep-extjac-metaplectic-sp-ext} also hold for general maximal parabolic subalgebras \(\fb_\kappa\) discussed in Section \ref{sec:euler-elements}.
  Hence, we expect that the construction from Theorem \ref{thm:std-dist-H-nu} can also be generalized to other maximal parabolic subalgebras \(\fb_\kappa\) for which \(\g_\kappa\) is hermitian simple and of tube type. Apart from \(\hcsp(\R^{2n})\), this is the case for the admissible Lie algebra of the form \(\fb_\kappa = \g(\g_\kappa, V, \R, \beta) \rtimes \R H_\kappa\) with \(\g_\kappa \cong \so(2,10)\). A description of the corresponding symplectic \(\g_\kappa\)-module structure on \((V,\beta)\) can be found in \cite[p.\ 118]{Sa80}.
\end{rem}

\subsection{The general case}
\label{sec:nets-adm-general}

We now consider a 1-connected Lie group \(G_C\) with an admissible Lie algebra \(\g_C\) which is of the form \(\g_C = \g(\fl, V, \fz, \beta) = V \times \fz \times \fl\) in terms of Spindler's construction (cf. Section \ref{sec:nets-std-preliminaries-adm-hwr}). We may assume that \(G_C = \Heis(V,\Omega) \rtimes L\), where \(L\) is 1-connected with \(\L(L) = \fl\).

In this paper, we will only discuss the case where \(\fz = \fz(\g_C) = \R\), so that \(\Omega := \beta\) is a symplectic form on \(V\) (see also Section \ref{sec:nets-adm-perspectives} for a perspective on the general case).
We assume that there exist non-zero Euler elements of \(\fl\) and fix such an element \(h_\fl \in \fl\). Then \(h_\fl\) acts as an antisymplectic involution on \(V\), so that \(h := h_\fl + \frac{1}{2}\id_V\) is an Euler derivation of \(\g_C\) (cf.\ Section \ref{sec:euler-elements}).
The involution on \(\g_C\) induced by \(h\) is denoted by \(\tau := e^{i\pi \ad h}\).
The corresponding extended Lie algebra is defined by \(\g := \g_C + \R \id_V\) and its 1-connected Lie group is denoted by \(G = G_C \rtimes_\alpha \R_+^\times\).
We denote by \(\tau_G\) the involutive automorphism of \(G\) with \(\L(\tau_G) = \tau\).

  Let \((\pi, \cH)\) be an irreducible strongly continuous antiunitary representation of \(G_\tau = G \rtimes \{\1, \tau_G\}\) with discrete kernel such that the positive cone \(C_\pi\) generates \(\g_C\).
  We denote by \(\tV\) the standard subspace of \(\cH\) specified by
  \[J := \pi(\tau_G) \quad \text{ and } \quad \Delta^{-it/2\pi} = \pi(\exp(th)) \quad \text{for } t \in \R.\]
  By \cite[Thm.\ 2.11]{NO17}, the restricted unitary representation \((\pi\lvert_G, \cH)\) is either irreducible or the direct sum \((\pi_+ \oplus \pi_-, \cH_+ \oplus \cH_-)\) of two irreducible representations.
  In the latter case, the antiunitary involution \(J\) interchanges \(\cH_+\) and \(\cH_-\).
  Since \(\tau(C_\pi) = -C_\pi\) and \(J\cH_\pm = \cH_\mp\), we have \(C_{\pi_+} = C_{\pi_-} = C_\pi\).
  By applying Lemma \ref{lem:urep-conf-jacobi-energy-decomp}, we see that \((\pi_\pm\lvert_N, \cH_\pm)\) is a representation of either strictly positive or strictly negative energy because \(G\) acts irreducibly on \(\cH_\pm\).
  In particular, \((\pi\lvert_N, \cH)\) is either of strictly positive or strictly negative energy.
  We assume throughout this section that \((\pi\lvert_N, \cH)\) is of strictly positive energy.

  By the Factorization Theorem \ref{thm:rep-adm-pos-factors} and Proposition \ref{prop:rep-adm-pos-antiuniext}, we may assume that \((\pi, \cH)\) is of the form \((\pi_L \boxtimes \nu_\gamma, \cH_L \hotimes \cH_\nu)\), where the normal subgroup \(N = \Heis(V,\Omega) \rtimes_\alpha \R_+^\times\) acts trivially on \(\cH_L\), and \(\nu_\gamma\) is the unique antiunitary extension of the pullback via \eqref{eq:rep-adm-pos-factors-pullback} of the unique irreducible unitary representation of \(\tilde\HCSp(V,\Omega)\) for which the restriction to \(N\) is an irreducible positive energy representation.
  The antiunitary representation \((\pi_L, \cH_L)\) is irreducible by Proposition \ref{prop:adm-arep-commutant}.

  By Proposition \ref{prop:rep-adm-pos-antiuniext}, the involution \(J\) can be written as a tensor product
  \[J = J_L \otimes J_\nu := \pi_L(\tau_G) \otimes \nu(\tau_G),\]
  and the one parameter subgroup \((\Delta^{-it/2\pi})_{t \in \R}\) as
  \[\Delta^{-it/2\pi} = \pi_L(\exp(th_\fl)) \otimes \nu_\gamma(\exp(th))\quad \text{ for } t \in \R.\]
  We thus obtain standard subspaces \(\tV_L  \subset \cH_L\) and \(\tV_\nu \subset \cH_\nu\) such that \(\tV = \tV_L \hotimes \tV_\nu\).
  \begin{lem}
    \label{lem:tensor-fact-std-semigrp-intersection}
    The Lie wedge of the subsemigroup \(S := S_{\tV_L} \cap S_{\tV_\nu} \subset G\) generates \(\g\). In particular, \(\1 \in G\) is contained in the closure of the interior of \(S\).
  \end{lem}
  \begin{proof}
    The Lie wedges of \(S_{\tV_L}\) and \(S_{\tV_\nu}\) can be computed using \eqref{eq:sv-lie-wedge-3grad} (cf.\ \cite[Cor.\ 5.5]{Ne21}).
    Moreover, we have \(\L(S) = \L(S_{\tV_L}) \cap \L(S_{\tV_\nu})\).
    It is obvious that \(\g_0(h) \subset \L(S)\).
    It remains to show that \(C_{\pi_L} \cap C_{\nu_\gamma} \cap \g_{\pm 1}(h)\) generates \(\g_{\pm 1}(h)\).
    We first recall from \cite[Thm.\ 3.17]{Oeh20b} that the intersection of \(\g_{\pm 1}(h)\) with \(C_{\nu_\gamma} = \gamma^{-1}(\hsp(V,\Omega)_+)\) generates \(\g_{\pm 1}(h)\).
    As a consequence of Theorem \ref{thm:rep-adm-pos-factors}(c)(d), there exists a minimal invariant cone \(\Wmin \subset \g_C\) with \(\Wmin \subset C_{\nu_\gamma}\) which is also contained in \(C_{\pi_L}\) because \((\pi_L, \cH_L)\) is a highest weight representation (cf.\ \cite[Thm.\ X.4.1]{Ne00}).
    Furthermore, the intersection \(\Wmin_{,\fl} := \Wmin \cap \fl\) is a minimal invariant cone of \(\fl\).
    From Lemma \ref{lem:euler-cone-intersection-reductive}, we obtain
    \[\fl_{-1}(h) \cap C_{\pi_L} = \fl_{-1}(h) \cap \Wmin = \fl_{-1}(h) \cap C_{\nu_\gamma} = \fl_{-1}(h) \cap \gamma^{-1}(\hsp(V,\Omega)_+),\]
    which generates \(\fl_{-1}(h) = \g_{-1}(h)\). 

    Let \(x = (v,z,y) \in C_{\nu_\gamma} \cap (V_1 \times \fz \times \fl_1(h_\fl)) = C_{\nu_\gamma} \cap \g_1(h)\).
    From \eqref{eq:hsp-plus-cone}, we see that we must have \(y \in C_{\nu_\gamma} \cap \fl_1(h_\fl)\) so that the previous discussion and Lemma \ref{lem:euler-cone-intersection-reductive} imply that \(y \in \Wmin\).
    Since \(V \times \fz \subset C_{\pi_L}\) holds trivially, this shows that \(x \in C_{\pi_L}\), so that \(C_{\pi_L} \cap C_{\nu_\gamma} \cap \g_1(h)\) generates \(\g_1(h)\).
    Hence, \(\L(S)\) generates \(\g = \g_C + \R h\).
    The last claim follows from Remark \ref{rem:tensor-dist-semigrp-cond}.
  \end{proof}

After these preparations, we obtain the following theorem, which is the main result of this paper:

\begin{thm}
  \label{thm:nets-adm-general}
  There exists a real subspace \(\sE \subset \tV^{-\infty}\) which satisfies the conditions {\rm (C1)-(C3)} of {\rm Theorem \ref{thm:netstand-constr}} and is invariant under the action of the subgroups \(L^{\tau}\) and \(V_1\).
\end{thm}
\begin{proof}
  By Proposition \ref{prop:adm-arep-commutant}, the restriction of \((\pi, \cH)\) to \(G\) is either irreducible or a direct sum of two irreducible representations for which the positive cone is pointed and generating in \(\g_C\). By applying Theorem \ref{thm:rep-adm-pos-factors} to each of these irreducible subrepresentations, we see that \((\pi_L, \cH_L)\) is either a unitary highest weight representation or the direct sum of two unitary highest weight representations of \(G_C\), respectively \(L \cong G_C / N\) because \(N = \Heis(V,\Omega) \rtimes_\alpha \R_+^\times\) acts trivially on \(\cH_L\) (cf.\ \cite[Lem.\ IX.2.19]{Ne00}).

  We can thus apply \cite[Thm.\ 5.3]{NO21} (see also \cite[Rem.\ 5.4]{NO21} and \cite[Rem.\ 4.10]{NO21}) to obtain a real subspace \(\sE_L \subset \cH_{\pi_L, \ext, J_L}^{-\infty} \subset \tV_L^{-\infty}\) that satisfies the conditions (C1)-(C3) of Theorem \ref{thm:netstand-constr} for the representation \((\pi_L, \cH_L)\) of \(L\) and the Euler element \(h_\fl \in \fl\).
  Moreover, \(\sE_L\) is invariant under \(L^{\tau}\) and \(L_{-1} = G_{-1}\), and it is a cyclic space of distribution vectors with respect to the action of \(L_1 := \exp_L(\fl_1(h_\fl))\).

  On the other hand, we recall that the real subspace \(\sE_\nu \subset \cH_{\nu,\ext,J_\nu}^{-\infty} \subset \tV_\nu^{-\infty}\) (cf.\ \eqref{eq:std-dist-H-nu}) is cyclic under the action of \(N\) (cf.\ Theorem \ref{thm:std-dist-H-nu}) and is fixed pointwise by \(L_1\) (cf.\ Proposition \ref{prop:H-nu-dist-invariance}).
  By Proposition \ref{prop:boxtensor-dist-prod} every element in \(\sE_L \otimes \sE_\nu\) extends to a continuous antilinear functional on \(\cH^\infty\).
  We denote the subspace spanned by these extensions by \(\sE\).
  Since \(\1 \in G\) is contained in the closure of the interior of \(S_{\tV_\nu} \cap S_{\tV_L}\) by Lemma \ref{lem:tensor-fact-std-semigrp-intersection}, we can apply Lemma \ref{lem:std-boxtensor-dists} and obtain \(\sE \subset \tV^{-\infty}\).
  The invariance of \(\sE\) under \(L^{\tau}\), \(V_1\), and \((\exp(th))_{t \in \R}\) follows from the invariance properties of \(\sE_L\) and \(\sE_\nu\).

  In order to show that \(\sE\) is \(G\)-cyclic, we only have to observe that \(\sE\) is non-trivial, so that the irreducibility of \((\pi, \cH)\) automatically implies that \(\sE\) is \(G\)-cyclic by Remark \ref{rem:netstand-constr-cyclic-irred}.
\end{proof}

\section{Nets on homogeneous spaces}
\label{sec:nets-hom-spaces}

As we have already mentioned in the introduction, we can use our construction from Section \ref{sec:nets-std-adm} to obtain nets of standard subspaces on homogeneous spaces, which are the space-times that these nets are often constructed on in AQFT.
In this section, we discuss various examples of this kind.

The general idea is that, having constructed a covariant isotone net of cyclic subspaces indexed by open subsets of a Lie group \(G\), we naturally obtain a covariant isotone net on a homogeneous space \(G/P\), where \(P \subset G\) is a closed subgroup, via the quotient map \(q : G \rightarrow G/P\).

Throughout this section, we preserve the context from the beginning of Section \ref{sec:nets-adm-general}.
We apply the construction in Theorem \ref{thm:nets-adm-general} to the real subspace of distribution vectors \(\sE\) of the antiunitary representation \((\pi, \cH)\) of \(G_\tau\), resp.\ the subspace \(\sE_P\) that we obtain from Theorem \ref{thm:std-dist-parabolic-embedding} for the representation \((\nu, \cH_\nu)\) of \(\tilde\HCSp(V,\Omega)\).

\begin{example}
  \label{ex:std-nets-homogeneous-spaces}
  Similar to the hermitian simple case (cf.\ \cite[Sec.\ 5.2]{NO21}), the construction from Theorem \ref{thm:nets-adm-general} can be used to construct a net of cyclic subspaces on the homogeneous space \(M := G/G^{\tau}\) which satisfies the Haag--Kastler axioms from the introduction.
  Let \(q_M : G \rightarrow M\) be the canonical quotient map.
  For every non-empty open region \(\cO \subset M\), we define the cyclic subspace
  \[\tV_{M}(\cO) := \tV(q_M^{-1}(\cO)) = \sH_{\sE}(q_M^{-1}(\cO)).\]
  This defines a \(G\)-covariant isotone net on \(M\).
  Recall the subsemigroup \(S := G^\tau\exp(C_+^o \oplus C_-^o) \subset S_\tV\) (cf.\ Theorem \ref{thm:netstand-constr}), which acts as a wedge domain of the original net on \(G\).
  If we define \(\cW := q_M(S)\), then \(G^\tau \subset S\) implies
  \[\tV_{M}(\cW) = \tV(q_M^{-1}(\cW)) = \tV(SG^\tau) = \tV(S).\]
  The same argument can also be applied to the net of cyclic subspaces on \(\tilde\HCSp(V,\Omega)\) which we obtained from the embedding of \(\hcsp(V,\Omega)\) into \(\sp(\tilde V, \tilde \Omega)\) (cf.\ Theorem \ref{thm:std-dist-parabolic-embedding}).
  In this case, the space \(\sE_P\) of distribution vectors is even invariant under the closed subgroup \(P_- := G^\tau G_{-1}\), so that we obtain a net of cyclic subspaces on \(G/P_-\) in the same way as above: If \(q_- : G \rightarrow G/P_-\) denotes the canonical quotient map, then we define for every open set \(\cO \subset G/P_-\) a cyclic subspace by
  \[\tV_{G/P_-}(\cO) := \tV(q_-^{-1}(\cO)) = \sH_{\sE_P}(q_-^{-1}(\cO)).\]
  The wedge domain \(\cW_- := q_-(S)\) then satisfies
  \begin{equation}
    \label{eq:std-nets-homspace-inv}
    \tV_{G/P_-}(\cW_-) = \tV(SP_-) = \tV(SG_{-1}) = \tV(S).
  \end{equation}
  because \(\sE_P\) is \(G_{-1}\)-invariant (cf.\ \cite[Lem.\ 2.11]{NO21}).
\end{example}

\begin{example}
  (1) We consider the subgroup \(L \subset G = \Heis(V,\Omega) \times (L \times \R_+^\times)\) (cf.\ Section \ref{sec:nets-adm-general}) as the homogeneous space \(G/N\), where \(N = \Heis(V,\Omega) \rtimes \R_+^\times \subset G\).
  The quotient map \(G \rightarrow G/N\) then corresponds to the projection
  \[p_L : G = \Heis(V,\Omega) \rtimes (L \times \R_+^\times) \rightarrow L, \quad (v,z,g,t) \mapsto g.\]
  From the net \(\cO \mapsto \tV(\cO) = \sH_{\sE}(\cO)\) on \(G\), we then obtain a net on \(L\) defined by
  \[\cO_L \mapsto \tV_L(\cO_L) := \tV(p_L^{-1}(\cO_L)) \quad \text{for } \cO_L \subset L \text{ open.}\]
  Let \(S_L := \exp(C_-^o)L^\tau\exp((C_+ \cap \fl)^o)\), which is an open subgroup of \(L\). 
  Since \([\fl,\fl]\) is a Levi-complement of \(\g\), \cite[Lem.\ 2.1(c)]{Ne20} implies that the preimage of \(S_L\) under \(p_L\) is the wedge domain \(S = G^\tau\exp(C_+^o \oplus C_-^o)\) of the net on \(G\).
  As a result, we have \(\tV(S) = \tV_L(S_L)\).

  We note that, even though the Lie algebra of \(L\) is reductive and quasihermitian, this construction of a net on \(L\) is different from that in \cite{NO21} because the pair of modular objects of \(\tV_L(S_L)\) is determined by \(h = h_\fl + \frac{1}{2}\id_V \not\in \fl\) and the involution \(\tau_G \not\in L_\tau\) (cf.\ Theorem \ref{thm:netstand-constr}(c)).
  Moreover, the restriction of the representation \((\pi, \cH)\) to \(L_\tau\) is in general not irreducible.

  (2) Recall the example of the circle \(\bS^1\) as a homogeneous space \(\PSL(2,\R) / P\) from the introduction, where \(P\) is a two-dimensional subgroup.
  Suppose that \(G = \tilde\HCSp(\R^2)\), so that \(\fl = \sp(2,\R) = \fsl(2,\R)\).
  We assume that \(P\) and \(h\) are chosen in such a way that \(\L(P) = \fl_{-1}(h) \oplus \fl_0(h)\) and we denote by \(\tilde P \subset L \cong \tilde\PSL(2,\R)\) the integral subgroup with Lie algebra \(\L(P)\).
  Then \(L/\tilde P \cong \R\) is a homogeneous space and a covering of \(\bS^1\).
  We denote the corresponding (lifted) quotient map by \(q : L \rightarrow L/\tilde P\).
  We consider the net \(\cO \mapsto \tV(\cO) = \sH_{\sE_P}(\cO)\) on \(G\) from Theorem \ref{thm:std-dist-parabolic-embedding} induced by the representation \((\nu, \cH_\nu)\).
  Then we apply (1) to obtain a net \(\cO \mapsto \tV_L(\cO)\) on \(L\). The same construction as in Example \ref{ex:std-nets-homogeneous-spaces} then leads to a net indexed by open subsets on \(L/\tilde P\).
  Because of the invariance properties of \(\sE_P\) and because \(p_L^{-1}(L_{-1}) = G_{-1}\) (cf.\ \cite[Lem.\ 2.1]{Ne20}), a similar computation as in \eqref{eq:std-nets-homspace-inv} shows that the cyclic subspace assigned to \(q(S_L) \subset L/\tilde P\) is standard.
\end{example}

\begin{example}
  We give an example of a homogeneous space where the preimages of the open subsets are too large to determine whether they are assigned to standard subspaces in the induced net of cyclic subspaces.

  Suppose that \(G = \tilde\HCSp(V,\Omega) = \Heis(V,\Omega) \rtimes (\tilde\Sp(V,\Omega) \times \R_+^\times)\).
  We have a natural action of \(\HCSp(V,\Omega)\) on \(V\) given by
  \[(v,z,g,t).w := e^tgw + v, \quad (v,z,g,t) \in V \times T \times \Sp(V,\Omega) \times \R_+^\times, w \in V,\]
  which we lift to an action of \(G\) via the corresponding covering homomorphism \(G \rightarrow \HCSp(V,\Omega)\).
  Then \(V \cong G / P\), where \(P = T \rtimes (\tilde\Sp(V,\Omega) \times \R_+^\times)\). We denote by \(q : G \rightarrow V\) the projection onto \(V\).

  Consider the net \(\cO \mapsto \tV(\cO)\) on \(G\) that is induced by the unitary representation \((\nu, \cH_\nu)\) of \(G\) and the distribution subspace \(\sE_P \subset \cH^{-\infty}\) from Theorem \ref{thm:std-dist-parabolic-embedding}.
  For an open subset \(\cO_V \subset V\), we set \(\tV_V(\cO_V) := \tV(q^{-1}(\cO_V))\), so that we obtain a covariant isotone net on \(V\).
  Recall that the Euler element \(h\) of \(\g\) which we fixed in the beginning of Section \ref{sec:nets-adm-general} determines a decomposition \(V = V_{-1} \oplus V_1\), where \(V_{-1} \subset \g_0(h)\) and \(V_{1} \subset \g_1(h)\).

  Consider a non-empty open subset \(\cO_V = \cO_{-1} \times \cO_{1} \subset V_{-1} \times V_1\) and let \(\cO := q^{-1}(\cO_V)\). 
  Let \(v_1 \in \cO_1\). Then there exists \(x \in \sp(V,\Omega)_1(h)\) and \(z \in \fz\) such that \((v_1,z,x) \in C_+ = \hsp(V,\Omega)_+ \cap \g_{1}(h)\) (cf.\ Remark \ref{rem:rep-nu-pos-cone-bounded-cone}).
  From the definition of \(\hsp(V,\Omega)_+\), it is straightforward to see that we also have \((v_1,-z,-x) \in -\hsp(V,\Omega)_+ \cap \g_1(h)\).
  In particular, we have \(\cO \not\subset S = G^\tau\exp(C_-^o \oplus C_+^o)\), because the preimage of \(\cO_1\) under \(q\) cannot be contained in \(C_+\).
\end{example}

\section{Perspective}
\label{sec:nets-adm-perspectives}

Consider a 1-connected Lie group \(G_C\) with admissible Lie algebra \(\g_C = \g(\fl,V,\fz,\beta)\).
If the dimension of the center \(\fz\) of \(\g_C\) is larger than \(1\), then the non-trivial Euler derivations \(D \in \der(\g_C)\) are in general significantly more complicated than in the one-dimensional case (cf.\ \cite[Thm.\ 3.14]{Oeh20b}).
In particular, the derivation may induce a non-trivial 3-grading on \(\fz\) (cf.\ \cite[Ex.\ 3.20]{Oeh20b}).
As a result, there exists in general more than one class of irreducible representations of \(G = G_C \rtimes \exp(\R D)\) of strictly positive energy. This can be seen by considering the action of the automorphism group induced by \(D\) on the characters on the center of \(G_C\), by a similar argument as in Lemma \ref{lem:conf-heis-irrep} (see also Remark \ref{rem:mackey-method-induced-rep}).

This suggests that the results in this paper on the representation theory of extended admissible Lie groups \(G\) as above can not immediately be generalized to the case where \(\dim \fz > 1\). We will therefore discuss this case in the upcoming paper \cite{Oeh21}.

The higher dimensional analogs of the Jacobi algebra, which is the prototypical example for our construction, are the admissible Lie algebras \(\hsp(V,\beta) := \heis(V,\beta) \rtimes \sp(V,\beta)\), where \(\beta\) is a \(\fz\)-valued skew-symmetric bilinear map on \(V\) for which there exists a linear functional \(f \in \fz^*\) such that \(f \circ \beta\) is non-degenerate.
We refer to \cite[Sec.\ 5]{Neu00} for a first step towards classifying the admissible Lie algebras \(\g(\fl,V,\fz,\beta)\) with \(\dim \fz > 1\) in terms of their \(\fl\)-module structure and the bilinear maps \(\beta\).

\section*{Acknowledgement}

We are most grateful to Karl-Hermann Neeb for many helpful discussions on the topic of this paper and for giving feedback on earlier versions of the manuscript.
We also thank the referee for several valuable suggestions on how to improve the quality of this paper.

\appendix

\section{Jordan algebras}
\label{sec:appendix-jordan-algebras}

In this appendix, we recall all the basic facts about Jordan algebras which are needed in this paper.
\begin{definition}
  \begin{enumerate}
    \item A (real) Jordan algebra is pair \((E, \cdot)\) consisting of a real vector space \(E\) and a bilinear map
      \[\cdot : E \times E \rightarrow E, \quad (x,y) \mapsto x \cdot y, \quad \text{such that}\]
      \[x \cdot y = y \cdot x \quad \text{and} \quad x \cdot (x^2 \cdot y) = x^2 \cdot (x \cdot y) \quad \text{for all } x,y \in E.\]
    \item A real Jordan algebra is called \emph{euclidean} if there exists a scalar product \(\la \cdot, \cdot \ra\) on \(E\) with the property that
      \begin{equation}
        \label{eq:jordan-algebra-inv-scprod}
        \la x \cdot y, z \ra = \la y, x \cdot z \ra \quad \text{for all } x,y,z \in E.
      \end{equation}
    \item A Jordan algebra is called \emph{simple} if it does not contain a non-trivial ideal.
  \end{enumerate}
\end{definition}

\begin{definition}
  Let \(E\) be a euclidean Jordan algebra with unit element \(e\). A \emph{Jordan frame} is a subset \(F = \{c_1,\ldots,c_r\} \subset E\) of mutually orthogonal idempotents of \(E\) with the property that \(\sum_{k=1}^r c_k = e\) and every \(c \in F\) is \emph{primitive}, i.e.\ it is non-zero and cannot be written as the sum of two non-zero orthogonal idempotents. We define \(\rk E := r = |F|\) as the \emph{rank of \(E\)}.
\end{definition}

Let \(E\) be a simple euclidean Jordan algebra with unit \(e\). For an idempotent \(c \in E\), the corresponding left multiplication map \(L(c)\) is diagonalizable over \(E\) with \(\spec(L(c)) \subset \{0, \frac{1}{2}, 1\}\) (cf.\ \cite[Prop.\ III.1.3]{FK94}). We write \(E_\lambda(c)\) for the \(\lambda\)-eigenspace of \(L(c)\).

Let \(F = \{c_1,\ldots,c_r\}\) be a Jordan frame of \(E\).
Set
\[E_i := \R c_i \quad \text{and} \quad E_{ij} := E_{\frac{1}{2}}(c_i) \cap E_{\frac{1}{2}}(c_j) \quad \text{for } i,j \in \{1,\ldots,r\}, i \neq j.\]
For \(1 \leq k \leq r\), we define
\begin{equation}
  \label{eq:jordan-minor-def}
  E^{(k)} := E_1(c_1 + \ldots + c_k) = \bigoplus_{i=1}^k E_i \oplus \bigoplus_{1 \leq i < j \leq k} E_{ij}.
\end{equation}
Then \(E^{(k)}\) is a euclidean Jordan subalgebra of \(E\) with unit element \(c_1 + \ldots + c_k\), and we have \(E = E^{(r)}\). Moreover, one can show that \(E^{(k)}\) is also simple (cf.\ \cite[Prop.\ 4.1]{GT11}). A complete list of the subalgebras \(E^{(k)}\) for all simple euclidean Jordan algebras up to isomorphy can for example be found in \cite[Table 3]{Oeh20a}.

\subsection{The Kantor--Koecher--Tits construction}
\label{sec:appendix-KKT}

We recall the well-known \emph{Kantor--Koecher--Tits construction} (or KKT-construction), which gives a one-to-one correspondence between hermitian simple Lie algebras of tube type and simple euclidean Jordan algebras (cf.\ \cite{Ka64},\cite{Koe69},\cite{Ti62}).

\begin{definition}
  Let \(\g\) be a Lie algebra. A triple \((h,x,y) \in \g^3\) is called an \(\fsl_2\)-triple if
  \[[h,x] = 2x, \quad [h,y] = -2y, \quad \text{and} \quad [x,y] = h.\]
\end{definition}

Let \(\g\) be a hermitian simple Lie algebra of tube type with a Cartan involution \(\theta\). Let \(0 \neq h \in \g^{-\theta}\) be an Euler element (cf.\ Definition \ref{def:euler-element}). Then there exist \(x,y \in \g\) such that \((2h,x,y)\) is an \(\fsl_2\)-triple and \(\theta(x) = -y\). Moreover, the multiplication
\[a \cdot b := \tfrac{1}{2}[[a, y], b], \quad a,b \in \g_1(h),\]
defines the structure of a simple euclidean Jordan algebra \(E := \g_1(h)\) with unit element \(x\).
The rank of the Jordan algebra \(\g_1(h)\) equals the real rank of \(\g\).
The isomorphy type of this Jordan algebra does not depend on the choice of the triple \((h,x,y)\).
Conversely, one can construct from every simple euclidean Jordan a hermitian simple Lie algebra of tube type such that the original Jordan algebra is of the above form.
The Table below lists all simple euclidean Jordan algebras and the corresponding hermitian simple Lie algebras up to isomorphy. 

\begin{table}[H]
  \centering
  \begin{tabular}{|c|c|c|}
    \hline
    \(\g\) & \(E\) & \(\rk_\R \g = \rk E\) \\
    \hline
    \(\su(n,n)\) & \(\Herm(n,\C)\) & \(n\) \\
    \hline
    \(\sp(2n,\R)\) & \(\Sym(n,\R)\) & \(n\) \\
    \hline
    \(\so^*(4n)\) & \(\Herm(n,\H)\) & \(n\) \\
    \hline
    \(\so(2,n) \, (n \geq 3)\) & \(\R \times \R^{n-1}\) & \(2\) \\
    \hline
    \(\fe_{7(-25)}\) & \(\Herm(3,\mathbb{O})\) & \(3\) \\
    \hline
  \end{tabular}
  \label{table:kkt}
\end{table}

Let now \(h \in \g\) be such that \(\ad(h)\) induces a non-trivial 5-grading on \(\g\) of the form
\[\g = \g_{-1}(h) \oplus \g_{-\frac{1}{2}}(h) \oplus \g_0(h) \oplus \g_{\frac{1}{2}}(h) \oplus \g_1(h).\]
Then \(\g_t := \g_{-1}(h) \oplus [\g_{-1}(h), \g_1(h)] \oplus \g_1(h)\) is an ideal of \(\g\) and a hermitian simple Lie algebra of tube type with \(r_t := \rk_\R \g_t \leq \rk_\R \g\).
Moreover, the Jordan algebra \((\g_t)_1(h) = \g_1(h)\) corresponding to \(\g_t\) is isomorphic to \(E^{(r_t)}\) (cf.\ \cite[Sec.\ 3.1]{Oeh20a}).

\subsection{Riesz measures}
\label{sec:app-riesz-measures}

In this section, we recall the definition of Riesz measures of simple euclidean Jordan algebras (cf.\ \cite[Ch.\ VII]{FK94}).
Let \(E\) be a simple euclidean Jordan algebra of dimension \(n\).
Fix a Jordan frame \(\{c_1,\ldots,c_r\}\) and let \(d := \dim(E_{ij})\) for \(1 \leq i < j \leq r\) (the dimension of \(E_{ij}\) does not depend on \(i\) and \(j\)).
If \(E = \R\), then we set \(d := 1\).
Moreover, let \(\Omega := \{x^2 : x \in E\}^o\) be the interior of the cone of squares of \(E\).
For \(x \in E\), we denote by \(\Delta(x) := \det(x)\) the Jordan determinant of \(x\).

Fix a euclidean measure on \(E\) and let \(\alpha \in \C\) with \(\Re(\alpha) > (r-1)\frac{d}{2}\). For a Schwartz function \(\varphi \in \cS(E)\), we define
\[\Gamma_\Omega(\alpha) := \int_\Omega e^{-\tr(x)} \Delta(x)^{\alpha - \frac{n}{r}}\, dx \quad \text{and}\]
\begin{equation*}
  R_\alpha(\varphi) := \frac{1}{\Gamma_\Omega(\alpha)} \int_\Omega \varphi(x) \Delta(x)^{\alpha - \frac{n}{r}}\, dx.
\end{equation*}
Then \(R_\alpha\) is a tempered distribution (cf.\ \cite[Thm.\ VII.2.2]{FK94}). Moreover, the function \(\varphi_\alpha(x) := \Delta(x)^{-\alpha}\) is the Laplace transform of \(R_\alpha\) (cf.\ \cite[Cor.\ VII.1.3]{FK94}).

\begin{thm}
  \(R_\alpha\) is a positive measure on \(E\) if and only if \(\alpha\) is contained in the set
  \begin{equation}
    \label{eq:wallach-set}
    \{0, \tfrac{d}{2}, \ldots, \tfrac{d}{2}(r-1)\}\, \cup\, ] \tfrac{d}{2}(r-1), \infty [.
  \end{equation}
\end{thm}
\begin{proof}
  cf.\ \cite[Thm.\ VII.3.1]{FK94}.
\end{proof}

We say that \(R_\alpha\) is the \emph{Riesz measure on \(E\) with parameter \(\alpha\)} if \(\alpha\) is contained in the set \eqref{eq:wallach-set}. In this case, we also denote this measure by \(\mu_\alpha\).

\section{Smooth vectors of unitary representations}
\label{sec:app-nets-smooth}
In the first part of this appendix, we recall the definition of the space of smooth vectors of a strongly continuous unitary representation of a Lie group and the definition of the space of distribution vectors of such a representation. With respect to the notation and definitions, we mostly follow \cite{NO21}.

Let \((\pi, \cH)\) be a strongly continuous unitary or antiunitary representation of a connected finite-dimensional Lie group \(G\) with Lie algebra \(\L(G) = \g\).
A vector \(v \in \cH\) is called \emph{smooth} if the orbit map \(\pi^v : G \rightarrow \cH\) with \(\pi^v(g) = \pi(g)v\) is smooth.
We denote the space of smooth vectors of \((\pi, \cH)\) by \(\cH_{\pi}^\infty\) or simply \(\cH^\infty\) if there is no ambiguity.
An equivalent condition for a vector \(v \in \cH\) to be smooth is the smoothness of the matrix coefficient \(\pi^{v,v}(g) := \la \pi(g)v, v\ra\) (cf.\ \cite[Thm.\ 7.2]{Ne10}).
It is well known that \(\cH^\infty\) is a dense \(G\)-invariant subspace of \(\cH\). For \(x \in \g\), let
\[\partial\pi(x) : \cD_x \rightarrow \cH, \quad \partial\pi(x)v := \bigderat{0} \pi(\exp(tx))v,\]
where \(\cD_x := \{v \in \cH : \derat{0} \pi(\exp(tx))v \text{ exists}\}\).
Let \(\dd\pi(x) := \partial\pi(x)\lvert_{\cH^\infty}\).
Then \(\dd\pi : \g \rightarrow \End(\cH^\infty)\) defines a Lie algebra representation of \(\g\) by essentially skew-adjoint operators on \(\cH^\infty\), which extends to a representation of the universal enveloping algebra \(\cU(\g)\).
We endow \(\cH^\infty\) with the locally convex topology induced by the seminorms \(p_D(v) := \|\dd\pi(D)v\|\) for \(D \in \cU(\g), v \in \cH^\infty\).
Then \(\cH^\infty\) is a Fr\'echet space with respect to this topology and the inclusion \(\cH^\infty \hookrightarrow \cH\) is continuous.

We denote by \(\cH^{-\infty}\) the space of \emph{distribution vectors}, i.e.\ the space of continuous antilinear functionals on \(\cH^\infty\).
The space \(\cH\) can be embedded into \(\cH^{-\infty}\) via \(v \mapsto \la \cdot, v\ra\). We endow \(\cH^{-\infty}\) with a representation of \(G\), resp.\ \(\cU(\g)\), via
\[\pi^{-\infty}(g)\eta := \eta \circ \pi(g^{-1}) \quad \text{and} \quad \dd\pi^{-\infty}(D)\eta := \eta \circ \dd\pi(D^*)\]
for \(\eta \in \cH^{-\infty}, D \in \cU(\g)\), and \(g \in G\) if \(\pi(g)\) is unitary. Here, \(D \mapsto D^*\) denotes the involution on \(\cU(\g)\) that is uniquely determined by \(x \mapsto -x\) for \(x \in \g\).
If \(\pi(g)\) is antiunitary, we set
\[\pi^{-\infty}(g)\eta := \oline{\eta \circ \pi(g^{-1})}.\]

Let \(M(G)\) denote the set of finite complex Borel measures of \(G\). For two measures \(\mu, \nu \in M(G)\), the convolution \(\mu * \nu\) is defined via the Riesz representation Theorem as the unique complex Borel measure on \(G\) with
\[\int_G f(g) \, d(\mu * \nu)(g) = \int_G \int_G f(gh) \, d\mu(g)d\nu(h) \quad \text{for all } f \in C_c^\infty(G).\]
Moreover, we set \(\mu^*(E) := \oline{\mu(E^{-1})}\), where \(E \subset G\) is a Borel set. With respect to these operations and the addition of measures, \(M(G)\) is a \(*\)-algebra.
Let \(\mu_G\) be a left Haar measure on \nolinebreak\(G\). Then the \(*\)-algebra \(L^1(G) := L^1(G, \mu_G)\), endowed with the usual convolution product, can be embedded into \(M(G)\) as a two sided ideal via \(f \mapsto \mu_f\), where \(\mu_f = f \cdot \mu_G\) (cf.\ \cite[Prop.\ IV.15]{Fa00}). 

For \(\mu \in M(G)\), we define \(\pi(\mu)\) as the unique bounded operator on \(\cH\) with
\[\la \pi(\mu) v, w\ra = \int_G \la \pi(g)v, w\ra \, d\mu(g) \quad \text{for all } v,w \in \cH.\]
Then the map \(\mu \mapsto \pi(\mu)\) defines a \(*\)-representation of \(M(G)\) by bounded operators on \(\cH\) (cf.\ \cite[Thm.\ IV.11]{Fa00}). If \(f \in L^1(G)\), then we set \(\pi(f) := \pi(\mu_f)\).

If \(\varphi \in C_c^\infty(G) \subset L^1(G)\), then it is well known that \(\pi(\varphi)(\cH) \subset \cH^\infty\) and that \(\pi(\varphi) : \cH \rightarrow \cH^\infty\) is continuous with respect to the \(C^\infty\)-topology on \(\cH^\infty\) (cf.\ e.g.\ \cite[p.\ 56]{Mag92}).
As a consequence, we obtain a representation of the convolution algebra \(C_c^\infty(G)\) on \(\cH^{-\infty}\) via
\[\pi^{-\infty}(\varphi)\eta := \eta \circ \pi(\varphi^*) \quad \text{for } \varphi \in C_c^\infty(G), \eta \in \cH^{-\infty}.\]

\subsection{Tensor products of unitary representations}
\label{sec:app-tensor-uni-rep}

In this section, we consider, for \(k=1,2\), strongly continuous unitary representations \((\pi_k, \cH_k)\) of finite dimensional Lie groups \(G_k\) on Hilbert spaces \(\cH_k\) and the representation \((\pi_1 \otimes \pi_2, \cH_1 \hotimes \cH_2)\) on the Hilbert space tensor product \(\cH := \cH_1 \hotimes \cH_2\) of the Lie group \(G := G_1 \times G_2\) with
\[\pi(g_1,g_2) := \pi_1(g_1) \otimes \pi_2(g_2), \quad g_1 \in G_1, g_2 \in G_2.\]
For a Lie group \(H\) with unitary representations \((\rho_k, \cH_k)\) \((k=1,2)\), we denote the \emph{inner tensor product representation} on the tensor product \(\cH_1 \hotimes \cH_2\) by
\[\rho(g) := (\rho_1 \boxtimes \rho_2)(g) := \rho_1(g) \otimes \rho_2(g), \quad g \in H,\]
which is the same as the restriction of the tensor product representation \((\rho_1 \otimes \rho_2, \cH_1 \hotimes \cH_2)\) to the diagonal subgroup \(\{(g,g) : g \in H\} \subset H \times H\).

\begin{lem}
  \label{lem:smoothvec-tensor-dense}
  \begin{enumerate}
    \item The space \(\cH_{1,\pi_1}^\infty \otimes \cH_{2,\pi_2}^\infty\) is dense in \(\cH_{\pi}^\infty\) with respect to the \(C^\infty\)-topology.
    \item The space \(\cH_{1,\rho_1}^\infty \otimes \cH_{2,\rho_2}^\infty\) is dense in \(\cH_{\rho}^\infty\) with respect to the \(C^\infty\)-topology.
  \end{enumerate}
\end{lem}
\begin{proof}
  (a) Since \(\cH_{1,\pi_1}^\infty\) is dense in \(\cH_1\) and \(\cH_{2,\pi_2}^\infty\) is dense in \(\cH_2\), the tensor product space \(\cH_{1,\pi_1}^\infty \otimes \cH_{2,\pi_2}^\infty\) is dense in \(\cH_1 \hotimes \cH_2\).  Moreover, it is \(G\)-invariant since both factors of the tensor product are \(G\)-invariant. Let \(v \in \cH_{1,\pi_1}^\infty\) and \(w \in \cH_{2,\pi_2}^\infty\). Then we have \(v \otimes w \in \cH_{\pi}^\infty\) because the matrix coefficient satisfies \((\pi_1 \otimes \pi_2)^{v \otimes w, v \otimes w} = \pi_1^{v,v}\pi_2^{w,w}\). Hence, the space \(\cH_{1,\pi_1}^\infty \otimes \cH_{2,\pi_2}^\infty\) consists of smooth vectors for \(\pi_1 \otimes \pi_2\). Now the claim follows from \cite[Thm.\ 1.3]{Po72}.

  (b) is shown in the same way as (a).
\end{proof}

\begin{lem}
  \label{lem:smoothvec-tensor-dist-ext}
  Let \(f_1 \in \cH_{1,\pi_1}^{-\infty}\) and \(f_2 \in \cH_{2,\pi_2}^{-\infty}\). Then the antilinear functional \(f_1 \otimes f_2\) on \(\cH_{1,\pi_1}^\infty \otimes \cH_{2,\pi_2}^\infty\) determined by \((f_1 \otimes f_2)(v_1 \otimes v_2) = f_1(v_1)f_2(v_2)\) for \(v_1 \in \cH_{1,\pi_1}^\infty, v_2 \in \cH_{2,\pi_2}^\infty,\) extends uniquely to a continuous antilinear functional on \(\cH_\pi^{\infty}\).
\end{lem}
\begin{proof}
  We recall from \cite[Lem.\ A.2(b)]{NO21} that \(\cH_{k,\pi_k}^{-\infty} = \spann (\dd\pi^{-\infty}(\cU(\g_k))\cH_k)\). Hence, there exist \(D_1^{(1)},\ldots,D_n^{(1)} \in \cU(\g_1), D_1^{(2)},\ldots,D_m^{(2)} \in \cU(\g_2)\), and \(v_1,\ldots, v_n \in \cH_1, w_1,\ldots,w_m \in \cH_2\), such that
  \[f_1 = \sum_{k=1}^n \dd\pi_1^{-\infty}(D_k^{(1)*})v_k \quad \text{and} \quad f_2 = \sum_{\ell = 1}^m \dd\pi_2^{-\infty}(D_\ell^{(2)*})w_\ell.\]
  The continuous antilinear functional \(f := \sum_{k=1}^n \sum_{\ell = 1}^m \dd\pi^{-\infty}(D_k^{(1)*} \otimes D_\ell^{(2)*}) (v_k \otimes w_\ell) \in \cH_\pi^{-\infty}\)  is an extension of \(f_1 \otimes f_2\) since, for \(z_1 \in \cH_{1,\pi_1}^\infty\) and \(z_2 \in \cH_{2,\pi_2}^\infty\), we have
  \begin{align*}
    f(z_1 \otimes z_2) &= \sum_{k=1}^n\sum_{\ell = 1}^m \la \dd\pi(D_k^{(1)} \otimes D_\ell^{(2)})(z_1 \otimes z_2),v_k \otimes w_\ell \ra\\
                       &= \sum_{k=1}^n\sum_{\ell = 1}^m \la (\dd\pi_1(D_k^{(1)})z_1) \otimes (\dd\pi_2(D_\ell^{(2)})z_2), v_k \otimes w_\ell \ra \\
                       &=\sum_{k=1}^n\sum_{\ell = 1}^m \la \dd\pi_1(D_k^{(1)})z_1, v_k \ra \la \dd\pi_2(D_\ell^{(2)})z_2, w_\ell\ra \\
                       &= f_1(z_1)f_2(z_2).
  \end{align*}
  Moreover, the continuous extension \(f\) is unique because \(\cH_{1,\pi_1}^\infty \otimes \cH_{2,\pi_2}^\infty\) is dense in \(\cH_\pi^\infty\).
\end{proof}

\begin{lem}
  \label{lem:smoothvec-tensor-dist-urep}
  Let \(f_1 \in \cH_{1,\pi_1}^{-\infty}\) and \(f_2 \in \cH_{2,\pi_2}^{-\infty}\). For \(g_1 \in G_1\) and \(g_2 \in G_2\), we have
  \[\pi^{-\infty}(g_1,g_2)(f_1 \hotimes f_2) = (\pi_1^{-\infty}(g_1) f_1) \hotimes (\pi_2^{-\infty}(g_2) f_2).\]
\end{lem}
\begin{proof}
  Since the distribution vectors in the above equation are continuous on \(\cH_\pi^{\infty}\) and coincide on the subspace \(\cH_{1,\pi_1}^{\infty} \otimes \cH_{2,\pi_2}^{\infty}\), the claim follows from Lemma \ref{lem:smoothvec-tensor-dense}.
\end{proof}

\begin{prop}
  \label{prop:boxtensor-dist-prod}
  Let \(H_1,H_2\) be closed subgroups of \(H\) with Lie algebras \(\fh_1,\fh_2 \subset \fh := \L(H)\). Let \(\eta_k \in \spann\{\dd\rho_k^{-\infty}(\cU(\fh_k))\cH_k\} \subset \cH_{k,\rho_k}^{-\infty}\) for \(k=1,2\). If \(\rho_1\lvert_{H_2}\) is trivial and if \(\eta_2\) is fixed by \(\rho_2^{-\infty}(H_1)\), then \(\eta_1 \otimes \eta_2\) extends uniquely to a continuous antilinear functional \(\eta_1 \hotimes \eta_2\) on \(\cH^\infty_{\rho}\).
\end{prop}
\begin{proof}
  By assumption, there exists \(v_1^{(k)},\ldots,v_{n_k}^{(k)} \in \cH_k\) and \(D_{1}^{(k)},\ldots,D_{n_k}^{(k)} \in \cU(\fh_k)\) such that
  \[\eta_k = \sum_{\ell = 1}^{n_k} \dd\rho_k^{-\infty}(D_\ell^{(k)})v_{\ell}^{(k)} \quad \text{for } k = 1,2.\]
  Because of the continuity of the action of \(H\) on \(\cH_{k,\rho_k}^\infty\) and the continuity of \(\eta_2\) on \(\cH_{2,\rho_2}^\infty\) in the \(C^\infty\)-topology (cf.\ \cite[Thm.\ 4.4]{Ne10}), we have
  \begin{equation}
    \label{eq:boxtensorprod-inv-dist}
    \dd\rho_1(\fh_2)\cH_1^\infty = \{0\} \quad \text{and} \quad \dd\rho_2^{-\infty}(\fh_1)\eta_2 = \{0\}.
  \end{equation}
  Fix \(\ell_1 \in \{1,\ldots,n_1\}\). On the dense subspace \(\cH_{1,\rho_1}^\infty \otimes \cH_{2,\rho_2}^\infty \subset \cH_{\rho}^\infty\), we have the equality
  \[\sum_{\ell_2 = 1}^{n_2} \dd\rho^{-\infty}(D_{\ell_2}^{(2)})(v_{\ell_1}^{(1)} \otimes v_{\ell_2}^{(2)}) \overset{\eqref{eq:boxtensorprod-inv-dist}}{=} \sum_{\ell_2 = 1}^{n_2} v_{\ell_1}^{(1)} \otimes \dd\rho_2^{-\infty}(D_{\ell_2}^{(2)})v_{\ell_2}^{(2)} = v_{\ell_1}^{(1)} \otimes \eta_2\]
  which shows that the antilinear functional on the right hand side extends to a continuous antilinear functional on \(\cH_{\rho}^\infty\). In addition, we have the equality of
  \[\sum_{\ell_1 = 1}^{n_1} \dd\rho^{-\infty}(D_{\ell_1}^{(1)}) (v_{\ell_1}^{(1)} \otimes \eta_2) \overset{\eqref{eq:boxtensorprod-inv-dist}}{=} \sum_{\ell_1=1}^{n_1} \dd\rho_1^{-\infty}(D_{\ell_1}^{(1)})v_{\ell_1}^{(1)} \otimes \eta_2 = \eta_1 \otimes \eta_2\]
  on the dense subspace \(\cH_{1,\rho_1}^\infty \otimes \cH_{2,\rho_2}^\infty\), which shows that \(\eta_1 \otimes \eta_2\) can be extended to a continuous antilinear functional on \(\cH_\rho^\infty\). The uniqueness of the extension follows from Lemma \ref{lem:smoothvec-tensor-dense} and the completeness of \(\cH_\rho^\infty\).
\end{proof}

\end{document}